\documentclass[12pt]{article}

\usepackage{graphicx}
\usepackage{subfigure}
\usepackage{latexsym,amsmath,amsfonts,amscd, amsthm}
\usepackage{changebar}
\usepackage{color}
\usepackage{bm}
\usepackage{tikz}
\usetikzlibrary{arrows,backgrounds,snakes}

\usepackage{multirow}

\newtheoremstyle{plainNoItalics}{}{}{\normalfont}{}{\bfseries}{.}{ }{}

\theoremstyle{plain}
\newtheorem{thm}{Theorem}[section]

\theoremstyle{plainNoItalics}

\newtheorem{defn}[thm]{Definition} 
\newtheorem{rem}[thm]{Remark}
\newtheorem{prop}[thm]{Proposition}
\newtheorem{exa}[thm]{Example}

\renewcommand{\theequation}{\thesection.\arabic{equation}}

\newcommand{\f}{\frac}

\newcommand{\beq}{\begin{equation}}
\newcommand{\eeq}{\end{equation}}
\newcommand{\beqa}{\begin{eqnarray}}
\newcommand{\eeqa}{\end{eqnarray}}
\newcommand{\bit}{\begin{itemize}}
\newcommand{\eit}{\end{itemize}}
\newcommand{\bedef}{\begin{defn}}
\newcommand{\edefn}{\end{defn}}
\newcommand{\bpro}{\begin{prop}}
\newcommand{\epro}{\end{prop}}

\newcommand{\bx}{\bf x}
\newcommand{\bv}{\bf v}

\begin{document}

\baselineskip=1.6pc


\begin{center}
{\bf
High Order Maximum Principle Preserving Semi-Lagrangian Finite Difference WENO schemes for the Vlasov Equation
}
\end{center}
\vspace{.2in}
\centerline{
Tao Xiong \footnote{Department of
Mathematics, University of Houston, Houston, 77204. E-mail:
txiong@math.uh.edu.}
Jing-Mei Qiu \footnote{Department of Mathematics, University of Houston,
Houston, 77204. E-mail: jingqiu@math.uh.edu.
The first and second authors are supported by Air Force Office of Scientific Computing YIP grant FA9550-12-0318, NSF grant DMS-1217008 and University of Houston.}
Zhengfu Xu \footnote{Department of
Mathematical Science, Michigan Technological University, Houghton, 49931. E-mail: zhengfux@mtu.edu. Supported by NSF grant DMS-1316662.}
Andrew Christlieb
\footnote{
Department of Mathematics, Michigan State University, East Lansing, MI, 48824. E-mail: andrewch@math.msu.edu
}
}

\bigskip
\centerline{\bf Abstract}
\vspace{0.2cm}
In this paper, we propose the parametrized maximum principle preserving (MPP) flux limiter, originally developed in 
{\em [Z. Xu, Math. Comp., (2013), in press]}, 
to the semi-Lagrangian finite difference weighted essentially non-oscillatory scheme for solving the Vlasov equation.
The MPP flux limiter is proved to maintain up to fourth order accuracy for the semi-Lagrangian finite difference
scheme without any time step restriction. Numerical studies on the Vlasov-Poisson system demonstrate the performance of the proposed method and its ability in preserving the positivity of the probability distribution function while maintaining the high order accuracy. 
%

\vspace{1cm}

\newpage

\newpage

\section{Introduction}
\label{sec1}
\setcounter{equation}{0}
\setcounter{figure}{0}
\setcounter{table}{0}

In this paper, we will consider the Vlasov-Poisson (VP) system
\beqa
\label{eq:vp1}
\partial_t f + {\bv} \cdot \nabla_{\bx} f + E(t,{\bx}) \cdot \nabla_{\bv} f=0,\\
\label{eq:vp2}
E(t,{\bx})=-\nabla \phi(t,{\bx}), \qquad -\Delta \phi(t,{\bx})=\rho(t,{\bx}).
\eeqa
on the domain $(0, T]\times\Omega$, where $\Omega=\Omega_{\bx} \times \mathbb{R}^n$
and $\Omega_{\bx}\subset\mathbb{R}^n$.
It is an important system for modelling the collisionless plasma. 
The Vlasov equation (\ref{eq:vp1}) is a kinetic equation that describes the time
evolution of the probability distribution function (PDF) $f(t,{\bx},{\bv})$ of finding an electron
at position ${\bx}$ with velocity ${\bv}$ at time $t$. $E(t,{\bx})$ is the electric field
and $\phi(t,{\bx})$ is a self-consistent electrostatic potential function described
by the Poisson equation (\ref{eq:vp2}). The probability distribution function
$f$ is coupled with the long range fields via the charge density $\rho(t,{\bx})=\int_{\mathbb{R}^n} f(t,{\bx},{\bv})d{\bv}-1$,
with uniformly distributed infinitely massive ions in the background. The
total charge neutrality condition $\int_{\Omega_{\bx}} (\rho(t,{\bx})-1)d{\bx}=0$ is imposed. 
The Vlasov equation (\ref{eq:vp1}) in the hyperbolic form enjoys the maximum principle preserving (MPP) and positivity preserving (PP) properties. That is, let
\[
f_{m}= \underset{x,v}{\text{min}}(f(x,v,0)) \ge 0, \quad f_{M}=\underset{x,v}{\text{max}}(f(x,v,0)),
\]
the solution at later time is still within the interval $[f_{m}, f_{M}]$ and stays positive.

The VP system has been extensively studied numerically. Popular numerical methods, besides the semi-Lagrangian (SL) approaches which will be reviewed in the next paragraph, include 
particle-in-cell (PIC) methods \cite{eastwood1986particle,birdsall2005plasma,hockney2010computer}, Lagrangian particle methods \cite{barnes1986hierarchical,evstatiev2013variational}, weighted essentially non-oscillatory (WENO) coupled with Fourier collocation \cite{zhou2001numerical}, Fourier-Fourier spectral methods \cite{klimas1987method,klimas1994splitting},  finite volume methods \cite{fijalkow1999numerical, filbet2001conservative}, continuous finite element methods \cite{zaki1988finite1,zaki1988finite2}, Runge-Kutta discontinuous Galerkin methods \cite{heath2012discontinuous, ayuso2009discontinuous, de2012discontinuous, cheng2012study, energy-cons-VP}, and methods in references therein. 

In this paper, we focus on the SL method, which has been proposed and applied for a wide range of applications, e.g. in atmospheric modeling and simulations \cite{staniforth1991semi,Guo2013discontinuous}, in capturing the moving interface via solving level set equations \cite{strain1999semi}, in fluid dynamics \cite{xiu2001semi} and in kinetic simulations \cite{filbet2003comparison, umeda2006comparison}. Compared to the Eulerian approach, the SL approach also has a fixed set of numerical meshes, but in each time step evolution, the information propagates along the characteristics in the Lagrangian fashion. The SL method can be designed as accurate as the Eulerian approach. In addition, the SL method is free of time step restriction by utilizing the characteristic method in the temporal direction. The design of SL approach requires the tracking of characteristics forward or backward in time, which could be challenging for nonlinear problems. To avoid such difficulty in solving the VP system, the dimensional splitting approach was proposed in \cite{cheng1976integration}. Along this line, SL methods were developed in various settings. For example, in the finite difference framework, different interpolation strategies, such as the cubic spline interpolation \cite{sonnendrucker1999semi}, the cubic interpolated propagation \cite{nakamura1999cubic}, the high order WENO interpolation in a non-conservative form \cite{carrillo2007nonoscillatory} and in a conservative form \cite{qiu2010conservative,qiu2011aconservative,qiu2011bconservative} are proposed. There have also been many work in designing the SL methods in the finite volume framework for the VP system \cite{filbet2001conservative} and for the guiding center Vlasov model \cite{crouseilles2010conservative}, in the finite element discontinuous Galerkin (DG) framework \cite{rossmanith2011positivity, qiu2011positivity} and with a hybrid finite difference-finite element approach \cite{guo2013hybrid}. 

Our main interest in this paper is to develop a MPP, and in particular PP, SL finite difference WENO scheme for the VP system (\ref{eq:vp1}). 
The main challenge of designing MPP (and/or PP) schemes within the SL WENO framework is to maintain the designed high order of accuracy from the conservative WENO approximation. Meanwhile, it is desired that no additional restrictive CFL constraint will be introduced.
There have been some research efforts on designing MPP and PP high order SL schemes. For example, in  \cite{filbet2001conservative, crouseilles2010conservative}, PP property is preserved in the finite volume framework. However, the PP preservation of the PDF is accompanied with sacrificing high order spatial accuracy. Recently, in  \cite{rossmanith2011positivity,qiu2011positivity}, the SL discontinuous Galerkin (DG) methods to solve the VP system is coupled  with the MPP limiters, that were originally proposed by Zhang and Shu \cite{zhang2010maximum}. In these approaches, the limiters are applied to the reconstructed polynomials (for finite volume) or the representing polynomials (for DG). In general, the MPP (or PP) property together with the maintenance of high order accuracy is much more challenging to achieve in the finite difference framework than in the finite volume and finite element DG framework via limiting polynomials, see \cite{zhang2011maximum}.

We propose to generalize the recently developed parametrized MPP flux limiter \cite{mpp_xu} to a conservative SL finite difference WENO method solving the VP system. The original parametrized flux limiter for 1D scalar conservation laws  was later extended to the two-dimensional case \cite{liang2013parametrized}. Xiong et al. \cite{xiong2013parametrized} proposed to apply the parametrized MPP flux limiter to the final Runge-Kutta (RK) stage only, with significantly improved time step restriction for maintenance of high order accuracy, leading to much reduced computational cost. It has also been proved in \cite{xiong2013parametrized} that the parametrized MPP flux limiter can maintain up to third order accuracy both in space and in time for nonlinear scalar conservation laws. 

In order to apply the parametrized MPP flux limiter in \cite{mpp_xu} for solving the VP system, we start with the dimensional splitting approach. The parametrized MPP flux limiter is proposed for the conservative high order SL finite difference WENO scheme \cite{qiu2011aconservative, qiu2010conservative}. We mimic the proof in \cite{xiong2013parametrized} to prove that the parametrized MPP flux limiter for the SL finite difference scheme solving linear advection equations can maintain up to fourth order accuracy without any time step constraint.
We also apply the parametrized flux limiter proposed in \cite{xiong2013parametrized} to the RK finite difference WENO approximation of the VP system without operator splitting. Through numerical studies on weak and strong Landau damping, two stream instabilities, and KEEN waves, we show that both methods perform very well with the designed MPP properties, while maintaining high order accuracy and mass conservation.

The rest of the paper is organized as follows. In Section 2, the SL finite
difference WENO scheme is reviewed and the parametrized MPP flux limiting procedure is proposed for the SL WENO scheme. We also prove that the 
parametrized MPP flux limiter maintains up to fourth order accuracy without additional time step restriction. Numerical
studies of the scheme are presented in Section 3.
Conclusions are made in Section 4. 

\section{The parametrized MPP flux limiter for the SL WENO scheme}
\label{sec2}
\setcounter{equation}{0}
\setcounter{figure}{0}
\setcounter{table}{0}

In this section, we will briefly describe the SL finite difference WENO scheme for solving Vlasov equation (\ref{eq:vp1}) in one dimension (both physical and phase spaces). We adopt the method in \cite{qiu2011bconservative}, which is in a conservative flux difference form and is suitable to be coupled with the newly developed parametrized MPP flux limiter \cite{mpp_xu}. 

For the one-dimensional problem, periodic boundary conditions are imposed in x-direction
on a finite domain $\Omega_x=[0, L]$ and in v-direction 
with a cut-off domain $[-V_c, V_c]$ with $V_c$ chosen to be large enough to guarantee $f(x,v,t)=0$ for $|v|\ge V_c$.
The domain $\Omega=[0, L]\times[-V_c, V_c]$ is discretized by the the computational grid 
\[
0 = x_\frac12 < x_\frac32 < \cdots < x_{N_x+\frac12} = L, \quad 
-V_c = v_\frac12< v_\frac32 < \cdots < v_{N_v+\frac12} = V_c.
\]
Let $x_i=\frac12(x_{i-\frac12}+x_{i+\frac12})$ and $v_j=\frac12(v_{j-\frac12}+v_{j+\frac12})$
to be the middle point of each cell. The uniform mesh sizes in each direction are $\Delta x=x_{i+\frac12}-x_{i-\frac12}$ and $\Delta v=v_{j+\frac12}-v_{j-\frac12}$.

Following \cite{cheng1976integration}, the Vlasov equation is dimensionally split to the following form
\begin{eqnarray}
\partial_t f + v \cdot \nabla_x f = 0, \label{eq1}\\
\partial_t f + E(t,x) \cdot \nabla_v f = 0.\label{eq2}
\end{eqnarray}
Using the second order Strang splitting strategy, the numerical solution is updated from time level $t^n$ to time level $t^{n+1}$ by solving
equation (\ref{eq1}) for half a time step, then solving equation (\ref{eq2}) by a full time step,
followed by solving equation (\ref{eq1}) for another half a time step.
Note that each split equation (\ref{eq1}) or (\ref{eq2}) still preserves the MPP (and/or PP) property. 

In the following, we will take the prototype 1D linear advection equation 
\begin{equation}
u_t+a u_x=0
\label{eq3}
\end{equation}
with constant $a$, to present the SL finite difference scheme with the parametrized MPP flux limiters.

\subsection{Review of SL finite difference WENO scheme}
\label{sec21}

The SL finite difference WENO scheme proposed in \cite{qiu2011bconservative} is based on
integrating equation (\ref{eq3}) in time over $[t^n, t^{n+1}]$,
\begin{equation}
u(x,t^{n+1})=u(x,t^n)- \mathcal{F}(x)_x,
\label{eq4}
\end{equation}
where
\begin{equation}
\mathcal{F}(x)=\int_{t^n}^{t^{n+1}} a u(x, \tau) d\tau .
\label{eq5}
\end{equation}
By introducing a sliding average function $\mathcal{H}(x)$
\begin{equation}
\mathcal{F}(x)=\frac{1}{\Delta x}\int_{x-\frac{\Delta x}{2}}^{x+\frac{\Delta x}{2}} \mathcal{H}(\zeta)d\zeta ,
\label{eq6}
\end{equation}
with
\begin{equation}
\mathcal{F}(x)_x=\frac{1}{\Delta x}\left(\mathcal{H}(x+\frac{\Delta x}{2})-\mathcal{H}(x-\frac{\Delta x}{2})\right),
\label{eq7}
\end{equation}
the evaluation of equation (\ref{eq4}) at the grid point $x_i$ can be written in a 
conservative form
\begin{equation}
u^{n+1}_i=u^n_i-\frac{1}{\Delta x}\left(\mathcal{H}(x_{i+\frac12})-\mathcal{H}(x_{i-\frac12})\right),
\label{eq8}
\end{equation}
where $\mathcal{H}(x_{i+\frac12})$ is called the flux function. A SL finite difference WENO reconstruction is used to approximate the numerical flux function $\mathcal{H}(x_{i+\frac12})$ based on its cell averages
\begin{equation}
\bar{\mathcal{H}}_j=\f{1}{\Delta x}\int_{x_{j-\f12}}^{x_{j+\f12}}\mathcal{H}(\zeta)d\zeta, \quad j=i-p,\cdots,i+q,
\label{eq9}
\end{equation}
with
\begin{equation}
\bar{\mathcal{H}}_j=\mathcal{F}(x_j)=\int_{x^{\star}_j}^{x_j}u(\zeta,t^n)d\zeta, \notag
\end{equation}
where $x^{\star}_j = x_j-a \Delta t$ is the point tracing from the grid point $(x_j, t^{n+1})$ along characteristics back to the time level $t^n$. The last equality above is essential for the SL scheme, and is obtained via following characteristics. $\int_{x^{\star}_j}^{x_j}u(\zeta,t^n)d\zeta$ can be reconstructed from $\{u^n_i\}^{N_x}_{i=1}$ for each $j$. In summary, the SL finite difference WENO scheme procedure in evolving equation (\ref{eq3}) from $t_n$ to $t_{n+1}$ is as follows:
\begin{enumerate}
\item \vspace{-.1in}
At each of the grid points at time level $t^{n+1}$, say $(x_{i}, t^{n+1})$,  trace the characteristic back to time level $t_n$ at $x^\star_{i} = x_i - a \Delta t$.
\item \vspace{-.1in}
Reconstruct $\mathcal{F}(x_i) = \int_{x_i^\star}^{x_i} u(\zeta,t^n)d\zeta$ from $\{u^n_j\}_{j=1}^{N_x}$. We use $\mathcal{R}_1$
to denote this reconstruction procedure
\beq
\label{eq: R1}
\mathcal{R}_1[x_i^\star, x_i] (u^n_{i-p_1}, \cdots, u^n_{i+q_1}),
\eeq
to approximate $\mathcal{F}(x_i)$, where $(i-p_1, \cdots, i+q_1)$ indicates the stencil used in the reconstruction. $\mathcal{R}_1[a, b]$ indicates the reconstruction of $\int_{a}^b u(\zeta, t)d\zeta$.
\item \vspace{-.1in}
Reconstruct $\{\mathcal{H}(x_{i+\frac12})\}_{i=0}^{N_x}$ from $\{\bar{\mathcal{H}}_i\}_{i=1}^{N_x}$. We use $\mathcal{R}_2$
to denote this reconstruction procedure
\beq
\label{eq: R2}
\mathcal{H}_{i+\frac12}\doteq \mathcal{R}_2 (\bar{\mathcal{H}}_{i-p_2}, \cdots, \bar{\mathcal{H}}_{i+q_2}),
\eeq
to approximate $\mathcal{H}(x_{i+\frac12})$. Here $(i-p_2, \cdots, i+q_2)$ indicates the stencil used in the reconstruction.
\item \vspace{-.1in}
Update the solution $\{u^{n+1}_i\}_{i=1}^{N_x}$ by
\beq
\label{eq: 1d_cons_2}
u^{n+1}_i = u^n_i - \frac{1}{\Delta x}( \mathcal{H}_{i+\frac12} - \mathcal{H}_{i-\frac12}),
\eeq
with numerical fluxes $\mathcal{H}_{i\pm\frac12}$ computed in the previous step.
\end{enumerate}
When the reconstruction stencils in $\mathcal{R}_1$ and $\mathcal{R}_2$ above only involve one neighboring point value of the solution, then the scheme reduces to a first order monotone scheme when the time step is within CFL restriction.
We let $h_{i+\f12}$ denote the first order flux.
The proposed SL finite difference scheme can be designed to be of high order accuracy by including more points in the  stencil for $\mathcal{R}_2 \circ \mathcal{R}_1$ (the composition of $\mathcal{R}_1$ and $\mathcal{R}_2$), to reconstruct the numerical flux
\begin{equation}
\label{weno_flux}
\mathcal{H}_{i+\frac12} = \mathcal{R}_2 \circ \mathcal{R}_1 (u^n_{i-p}, \cdots, u^n_{i+q}),
\end{equation}
where $(i-p, \cdots, i+q)$ indicates the stencil used in the reconstruction process.
The WENO mechanism can be introduced in the reconstruction procedures in order to realize a stable and non-oscillatory capture of fine scale structures. 
In the Appendix, we provide formulas to obtain the high order fluxes $\mathcal{H}_{i+\frac12}$ in \eqref{weno_flux} for the fifth order SL finite difference WENO scheme.
For more details, we refer to \cite{qiu2011bconservative}. 


For the case of large time step $|a|\Delta t > \Delta x$, if $a>0$, $x^{\star}_i=x_i-a \Delta t$ is no longer
inside $(x_{i-1}, x_i]$, let $i^{\star}$ to be the index such that $x^{\star}_i\in (x_{i^{\star}-1}, x_{i^{\star}}]$ and $\xi=\frac{x_{i^{\star}}-x^{\star}_i}{\Delta x}$, we have
\begin{eqnarray}
h_{i+\f12}&=&\sum_{j=i^{\star}+1}^i \Delta x u^n_j + (x_{i^{\star}}-x^{\star}_i)u^n_{i^{\star}}, \label{eq13}\\
\mathcal{H}_{i+\f12}&=&\sum_{j=i^{\star}+1}^i \Delta x u^n_j + \mathcal{H}_{i^{\star}+\f12}, 
\label{eq14}
\end{eqnarray} 
where $\mathcal{H}_{i^{\star}+\f12}$ is reconstructed in the same fashion as (\ref{weno_flux}), but replacing $i$ by $i^{\star}$.
Similarly, if $a<0$, let $i^{\star}$ to be the index such that $x^{\star}_i\in (x_{i^{\star}}, x_{i^{\star}+1}]$ and $\xi=\frac{x_{i^{\star}}-x^{\star}_i}{\Delta x}$, we have
\begin{eqnarray}
h_{i+\f12}&=&-\sum_{j=i+1}^{i^{\star}} \Delta x u^n_j + (x_{i^{\star}}-x^{\star}_i)u^n_{i^{\star}+1},\label{eq15}\\
\mathcal{H}_{i+\f12}&=&-\sum_{j=i+1}^{i^{\star}} \Delta x u^n_j + \mathcal{H}_{i^{\star}+\f12}. \label{eq16}
\end{eqnarray}
It is numerically demonstrated in \cite{qiu2011bconservative} that the proposed high order SL WENO method works very well in Vlasov simulations with extra large time step evolution.

\subsection{Parametrized MPP flux limiters}
\label{sec22}

In this subsection, we propose a parametrized MPP flux limiter, as proposed in \cite{mpp_xu, xiong2013parametrized}, for the high order SL finite difference WENO scheme (\ref{eq: 1d_cons_2}).

For simplicity, Let $u_{m}= \underset{x}{\text{min}}(u(x, 0))$ and $u_{M}=\underset{x}{\text{max}}(u(x, 0))$ as the minimum and maximum values of the initial condition. It has been known that the numerical solutions updated by \eqref{eq: 1d_cons_2} with the first order monotone flux satisfy the maximum principle. 
The MPP flux limiters are designed as modifying the high order numerical flux towards the first order monotone flux in the following way,
\begin{eqnarray}
\label{eq: fl}
\tilde{\mathcal H}_{i+\frac12} =\theta_{i+\frac12} ( \mathcal H_{i+\frac12}- h_{i+\frac12})+ h_{i+\frac12}
\end{eqnarray}
where 
$\theta_{i+\frac12} \in [0, 1]$ is the parameter to be designed to take advantage of the first order monotone flux $ h_{i+\frac12}$ in the MPP property and to take advantage of the high order flux $ \mathcal H_{i+\frac12}$ in the high order accuracy. 

Below is a detailed procedure of designing $\theta_{i+\frac12}$, in order to guarantee the MPP property of the numerical solutions, yet to choose $\theta_{i+\frac12}$'s to be as close to $1$ as possible for high order accuracy. For each $\theta_{i+\frac12}$ in limiting the numerical flux $\tilde{\mathcal H}_{i+\frac12}$ as in equation \eqref{eq: fl}, we are looking for the upper bounds $\Lambda_{+\frac12, I_{i}}$ and $\Lambda_{-\frac12, I_{i+1}}$, such that,  for all
\beq
\label{eq: theta}
\theta_{i+\frac12}\in [0, \Lambda_{+\frac12, I_{i}}] \cap  [0, \Lambda_{-\frac12, I_{i+1}}], 
\eeq
the updated numerical solutions $u^{n+1}_{i}$ and $u^{n+1}_{i+1}$ by the SL WENO scheme \eqref{eq: 1d_cons_2} with the modified numerical fluxes \eqref{eq: fl} are within $[u_m, u_M]$.
Let
\[
\Gamma^M_i=u_{M}-u^n_i+ \f{1}{\Delta x} ( h_{i+\frac12}- h_{i-\frac12}),
\quad
\Gamma^m_i=u_{m}-u^n_i+ \f{1}{\Delta x} ( h_{i+\frac12}- h_{i-\frac12}),
\]
the MPP property of the first order monotone flux guarantees 
\beq
\label{eq: monotone}
\Gamma^M_i\ge 0, \quad \Gamma^m_i \le 0.
\eeq
To ensure $u^{n+1}_i \in [u_m, u_M]$ with $\tilde{\mathcal H}_{i+\frac12}$ as in equation \eqref{eq: fl}, it is sufficient to have
\begin{eqnarray}
\label {eq: umax}
\f{1}{\Delta x} \theta_{i-\frac12} ( \mathcal H_{i-\frac12}- h_{i-\frac12}) - \f{1}{\Delta x} \theta_{i+\frac12} ( \mathcal H_{i+\frac12}-
 h_{i+\frac12})-\Gamma^M_i &\le& 0, \\
\label {eq: umin}
\f{1}{\Delta x}  \theta_{i-\frac12} ( \mathcal H_{i-\frac12}- h_{i-\frac12}) - \f{1}{\Delta x} \theta_{i+\frac12} ( \mathcal H_{i+\frac12}-
 h_{i+\frac12})-\Gamma^m_i &\ge&0.
\end{eqnarray}
The linear decoupling of $\theta_{i\pm\frac12}$,
subject to the constraints
(\ref{eq: umax}) and (\ref{eq: umin}), is achieved via a case-by-case discussion based on the sign of
\[
F_{i\pm\frac12}\doteq \frac{1}{\Delta x}\left(\mathcal H_{i\pm\frac12}- h_{i\pm\frac12}\right),
\]
as outlined below.
\begin{enumerate}
\item
Assume
\[
\theta_{i-\frac12} \in [0, \Lambda^M_{-\frac12, I_i}], \quad
\theta_{i+\frac12} \in [0, \Lambda^M_{+\frac12, I_i}],
\]
where $\Lambda^M_{-\frac12, I_i}$ and $\Lambda^M_{+\frac12, I_i}$ are the upper bounds of $\theta_{i\pm\frac12}$,
subject to the constraint \eqref{eq: umax}.
\begin{enumerate}
\item If $F_{i-\frac12}\le 0$ and $F_{i+\frac12}\ge 0$,
\[(\Lambda^M_{-\frac12, I_i}, \Lambda^M_{+\frac12, I_i})=(1, 1).
\]
\item  If $F_{i-\frac12}\le 0$ and $F_{i+\frac12}< 0$,
\[
(\Lambda^M_{-\frac12, {I_i}}, \Lambda^M_{+\frac12, {I_i}})=(1, \min(1, \frac{\Gamma^M_i}{- F_{i+\frac12}})).
\]
\item  If $F_{i-\frac12}> 0$ and $F_{i+\frac12}\ge 0$,
\[
(\Lambda^M_{-\frac12, {I_i}}, \Lambda^M_{+\frac12, {I_i}})=(\min(1, \frac{\Gamma^M_i}{ F_{i-\frac12}}), 1).
\]
\item  If $F_{i-\frac12}> 0$ and $F_{i+\frac12}< 0$,
\bit
\item
If equation  \eqref{eq: umax} is satisfied with $(\theta_{i-\frac12}, \theta_{i+\frac12})=(1, 1)$, then
\[
(\Lambda^M_{-\frac12, {I_i}}, \Lambda^M_{+\frac12, {I_i}})=(1, 1).
\]
\item Otherwise,
\[
(\Lambda^M_{-\frac12, {I_i}}, \Lambda^M_{+\frac12, {I_i}})=(\frac{\Gamma^M_i}{ F_{i-\frac12}-  F_{i+\frac12}},\frac{\Gamma^M_i}{ F_{i-\frac12}-  F_{i+\frac12}} ).
\]
\eit
\end{enumerate}
\item Similarly assume
\[
\theta_{i-\frac12} \in [0, \Lambda^m_{-\frac12, I_i}], \quad
\theta_{i+\frac12} \in [0, \Lambda^m_{+\frac12, I_i}],
\]
where $\Lambda^m_{-\frac12, I_i}$ and $\Lambda^m_{+\frac12, I_i}$
are the upper bounds of $\theta_{i\pm\frac12}$, subject to the constraint \eqref{eq: umin}.
\begin{enumerate}
\item If $F_{i-\frac12}\ge 0$ and $F_{i+\frac12}\le 0$,
\[
(\Lambda^m_{-\frac12, I_i}, \Lambda^m_{+\frac12, I_i})=(1, 1).
\]
\item  If $F_{i-\frac12}\ge 0$ and $F_{i+\frac12}> 0$,
\[
(\Lambda^m_{-\frac12, {I_i}}, \Lambda^m_{+\frac12, {I_i}})=(1, \min(1, \frac{\Gamma^m_i}{- F_{i+\frac12}})).
\]
\item  If $F_{i-\frac12}< 0$ and $F_{i+\frac12}\le 0$,
\[
(\Lambda^m_{-\frac12, {I_i}}, \Lambda^m_{+\frac12, {I_i}})=(\min(1, \frac{\Gamma^m_i}{ F_{i-\frac12}}), 1).
\]
\item  If $F_{i-\frac12}< 0$ and $F_{i+\frac12}> 0$,
\bit
\item
If equation  \eqref{eq: umin} is satisfied with $(\theta_{i-\frac12}, \theta_{i+\frac12})=(1, 1)$, then
\[
(\Lambda^m_{-\frac12, {I_i}}, \Lambda^m_{+\frac12, {I_i}})=(1, 1).
\]
\item Otherwise,
\[
(\Lambda^m_{-\frac12, {I_i}}, \Lambda^m_{+\frac12, {I_i}})=(\frac{\Gamma^m_i}{ F_{i-\frac12}-  F_{i+\frac12}},\frac{\Gamma^m_i}{ F_{i-\frac12}-  F_{i+\frac12}} ).
\]
\eit
\end{enumerate}
\end{enumerate}
Notice that the range of $\theta_{i+\frac12}$ \eqref{eq: theta} is determined by the need to ensure
both the upper bound \eqref{eq: umax} and the lower bound \eqref{eq: umin}
of numerical solutions in both cell $I_{i}$ and $I_{i+1}$. Thus
the locally defined limiting parameter is given as
\begin{eqnarray}
\label{eq: limit}
\theta_{i+\frac12}=\min(\Lambda_{+\frac12, {I_i}}, \Lambda_{-\frac12, {I_{i+1}}}),
\end{eqnarray}
with 
$\Lambda_{+\frac12, {I_i}} =
\min(\Lambda^M_{+\frac12, {I_i}}, \Lambda^m_{+\frac12, {I_{i}}})$ and 
$\Lambda_{-\frac12, {I_{i+1}}} =
\min(\Lambda^M_{-\frac12, {I_{i+1}}}, \Lambda^m_{-\frac12, {I_{i+1}}})$.
The modified flux in equation \eqref{eq: fl} with the $\theta_{i+\frac12}$ designed above ensures the maximum principle.
Such modified flux is consistent and preserves the maximum principle by its design, since it is a convex combination ($\theta_{i+\frac12}\in [0, 1]$) of a high order flux $ \mathcal H_{i+\frac12}$ and the first order flux $ h_{i+\frac12}$.
The mass conservation property is preserved, due to the flux difference form of the scheme.  

\begin{rem} (Machine zero) Due to rounding floating point arithmetic, equations \eqref{eq: monotone}, \eqref{eq: umax} and \eqref{eq: umin} are enforced only at the level of machine precision. In some extreme cases, numerical solutions may go out of bound, but only at the level of machine precision. 
\end{rem}

In \cite{xiong2013parametrized}, it was proved that the MPP flux limiter for the RK finite difference scheme can maintain up to third order accuracy both in space and in time for the nonlinear scalar equation $u_t+f(u)_x=0$. In Theorem \ref{thm: accuracy} below, we prove the proposed MPP limiter for the SL finite difference scheme can maintain up to fourth order accuracy for solving the linear advection equation (\ref{eq3}) {\em without any time step restriction}.
The generalization of the proof to maintain up to fifth order accuracy would be much more technical and will be investigated in our future work.

\begin{thm}
\label{thm: accuracy}
Consider solving the linear advection equation \eqref{eq3} using the finite difference SL method (\ref{eq: 1d_cons_2})
with a fourth order reconstruction procedure. Assume the global error,
\beq
e^n_j = |u^n_j - u(x_j, t^n)| = \mathcal{O}(\Delta x^4), \quad \forall n, j.
\label{assumption}
\eeq
Consider applying the parametrized MPP flux limiter to the numerical fluxes, 
then
\begin{eqnarray}
\left|\frac{1}{\Delta x}\left(\mathcal H_{j+\frac12}-\tilde{\mathcal H}_{j+\frac12}\right)\right|=\mathcal{O}(\Delta x^4), \quad \forall j,
\label{trc}
\end{eqnarray}
without any time step restriction.
\end{thm}

\begin{proof} 
Without loss of generality, we consider $a=1$ with $0\le \xi=\frac{\Delta t}{\Delta x}\le 1$.
The case of $\Delta t > \Delta x$ can be reduced to $\Delta t \le \Delta x$ by shifting the numerical solution with several whole grid points.  
Without specifying, in the following, we use $u_j$ instead of $u^n_j$ and use $u(x)$ instead of $u(x,t^n)$. Since the difference between $u(x_j, t^n)$ and $u^n_j$ is of high order due to assumption \eqref{assumption}, we use $u(x_j, t^n)$ and $u^n_j$ interchangeably when such high order difference allows. The high order flux $\mathcal{H}_{j+\f12}$ is taken to be (\ref{weno_flux}) with a 4th order finite difference reconstruction, the first order monotone flux is $h_{j-\f12}=\Delta t u_{j-1}$. 

We mimic the proof in \cite{xiong2013parametrized} and only consider the most challenging case: case (b) for the maximum
value part. The proof for other cases would be similar to that in \cite{xiong2013parametrized}.

We consider case (b) when
\begin{eqnarray}
\label{caseb}
\Lambda_{+\frac12, {I_j}}=\frac{\Gamma^M_j}{-F_{j+\frac12}}<1.
\end{eqnarray}
To prove (\ref{trc}), 
it suffices to prove
\beq
u_M-\left(u_j-\xi(\hat{H}_{j+\frac{1}{2}}-u_{j-1})\right) = \mathcal{O}(\Delta x^4),
\label{aeq401}
\eeq
if $u_M-\left(u_j-\xi(\hat{H}_{j+\frac{1}{2}}-u_{j-1})\right)<0$. For the SL method, we have 
\beq
\hat{H}_{j+\frac{1}{2}}=\frac{1}{\Delta t}\mathcal{H}_{j+\f12}=\frac{1}{\Delta t}\int_{t^n}^{t^{n+1}}H(x_{j+\frac{1}{2}},\tau)d\tau,
\label{aeq402}
\eeq
where $H(x,t)$ is defined by
\beq
u(x,t)=\frac{1}{\Delta x}\int^{x+\frac{\Delta x}{2}}_{x-\frac{\Delta x}{2}}H(\xi,t)d\xi.
\label{aeq403}
\eeq
$H(x,t)$ follows the same characteristics as the linear advection equation, hence
\beq
\hat{H}_{j+\frac{1}{2}}
=\frac{1}{\Delta t}\int_{0}^{\Delta t}H(x_{j+\frac{1}{2}}-\tau,t^n)d\tau.
\label{aeq404}
\eeq

Let $I_1=u_j-\xi(\hat{H}_{j+\frac{1}{2}}-u_{j-1})$. We approximate $H(x,t^n)$ by a 4th order reconstructed polynomial from the solution based on the stencil $S=\{u_{j-2},u_{j-1}, u_j, u_{j+1}\}$, with which
\begin{eqnarray}
I_1&=& u_j-\xi\left(\frac{1}{\Delta t}\int_{0}^{\Delta t}H(x_{j+\frac{1}{2}}-\tau,t^n)d\tau-u_{j-1}\right)\nonumber \\
&=&\frac{1}{24}(26u_j-10u_{j-1}+2u_{j-2}+6u_{j+1})+\frac{\xi}{24}(9u_j+3u_{j-1}-u_{j-2}-11u_{j+1})\nonumber \\
&&+\frac{\xi^2}{24}(-14u_j+10u_{j-1}-2u_{j-2}+6u_{j+1})+\frac{\xi^3}{24}(3u_j-3u_{j-1}+u_{j-2}-u_{j+1}). \label{aeq406}
\end{eqnarray}

We first consider the case when the maximum point $x_M \in (x_{j-1}, x_{j})$ at $t^n$, with $u_M=u(x_M)$, $u'_M=0$ and $u''_M\le 0$. $(x_{j-1}, x_j)$ is the dependent domain for the exact solution $u(x_j,t^{n+1})$ when $0\le\xi\le1$. We perform Taylor expansions around $x_M$ up to $4$th order and obtain
\begin{eqnarray}
u_{j+1}&=&u_M+u'_M(x_j-x_M+\Delta x)+u''_M\frac{(x_j-x_M+\Delta x)^2}{2}+u'''_M\frac{(x_j-x_M+\Delta x)^3}{6}+\mathcal{O}(\Delta x^4),\nonumber\\
u_j&=&u_M+u'_M(x_j-x_M)+u''_M\frac{(x_j-x_M)^2}{2}+u'''_M\frac{(x_j-x_M)^3}{6}+\mathcal{O}(\Delta x^4), \nonumber \\
u_{j-1}&=&u_M+u'_M(x_j-x_M-\Delta x)+u''_M\frac{(x_j-x_M-\Delta x)^2}{2}+u'''_M\frac{(x_j-x_M-\Delta x)^3}{6}+\mathcal{O}(\Delta x^4),\nonumber \\
u_{j-2}&=&u_M+u'_M(x_j-x_M-2\Delta x)+u''_M\frac{(x_j-x_M-2\Delta x)^2}{2}+u'''_M\frac{(x_j-x_M-2\Delta x)^3}{6}+\mathcal{O}(\Delta x^4).\nonumber 
\end{eqnarray}
Since $u'_M=0$, we can write $I_1$ to be
\beq
I_1=u_M+u''_M\frac{\Delta x^2}{2}R_2+u'''_M\frac{\Delta x^3}{6}R_3+\mathcal{O}(\Delta x^4),
\label{aeq407}
\eeq
where
\begin{eqnarray}
R_2&=&\frac{5}{6}\xi+\frac{1}{2}\xi^2-\frac{1}{3}\xi^3+(\xi^2-3\xi)z+z^2, \nonumber \\
R_3&=&\frac{1}{4}(-4\xi+\xi^2-2\xi^3+\xi^4)+\frac{1}{4}(10\xi+6\xi^2-4\xi^3)z
   +\frac{1}{4}(-18\xi+6\xi^2)z^2+z^3, \nonumber
\end{eqnarray}
and $z=(x_j-x_M)/\Delta x \in (0,1)$.

We consider a quantity denoted by $I_2$, which can be written as follows
\begin{eqnarray}
I_2&=&\beta_1 u(x_M)+(1-\beta_1)u(x_M+c_1\Delta x) \nonumber \\
   &=& u_M+u''_M\frac{\Delta x^2}{2} (1-\beta_1)c_1^2 + u'''_M\frac{\Delta x^3}{6}(1-\beta_1)c_1^3+\mathcal{O}(\Delta x^4),
   \label{aeq408}
\end{eqnarray}
with parameters $\beta_1$ and $c_1$ to be determined. If $0\le\beta_1\le 1$, $I_2\le u_M$.

Since $u''_M\le 0$, if we can find $c_1= \mathcal{O}(1)$, $0\le \beta_1\le 1$ such that
\begin{eqnarray}
(1-\beta_1)c_1^2 \le R_2, \qquad (1-\beta_1)c_1^3 = R_3,
\label{aeq409}
\end{eqnarray}
we would have
\begin{eqnarray}
| I_1 - u_M | = \mathcal{O}(\Delta x^4),
\label{aeq410}
\end{eqnarray}
when (\ref{aeq407}) is compared with \eqref{aeq408} under the assumption (\ref{caseb}), thus proving (\ref{aeq401}) in this case. In the following,
to determine the parameters $\beta_1$ and $c_1$ satisfying (\ref{aeq409}), we first need $(1-\beta_1)R_3^2\le R_2^3$, that is
\begin{eqnarray}
1-\beta_1\le \min_{0\le z \le 1}\left(\frac{R_2^3}{R_3^2}\right), \quad \text{with} \quad 0\le\xi\le 1.
\label{aeq411}
\end{eqnarray}
Using Mathematica, we have
\begin{eqnarray}
\min_{0\le z \le 1}\left(\frac{R_2^3}{R_3^2}\right)=
\begin{cases}
\frac{54-108\xi+72\xi^2-16\xi^3}{54-81\xi+27\xi^2}, &\quad 0\le\xi\le \min(|z|^3,\f12), \\
\frac{1}{2}-\frac{3}{2}\sqrt{\frac{3}{2}}\sqrt{\frac{2-7\xi+9\xi^2-5\xi^3+\xi^4}{27-32\xi+9\xi^2}}
, &\quad \min(|z|^3,\f12)\le\xi\le 1,
\end{cases}
\label{aeq412}
\end{eqnarray}
and
\begin{eqnarray}
\frac{64}{81}\le\frac{54-108\xi+72\xi^2-16\xi^3}{54-81\xi+27\xi^2}\le 1, &\quad\text{if}\quad 0\le\xi \le \frac{1}{2},
\label{aeq413} \\
0\le\frac{1}{2}-\frac{3}{2}\sqrt{\frac{3}{2}}\sqrt{\frac{2-7\xi+9\xi^2-5\xi^3+\xi^4}{27-32\xi+9\xi^2}}\le 1, &\quad\text{if}\quad 0\le\xi\le 1.
\label{aeq414}
\end{eqnarray}
Thus, we can choose $\beta_1=1-\min_{0\le z \le 1}\left(\frac{R_2^3}{R_3^2}\right) \in [0, 1]$.  
Let $c_1=\left(\frac{R_3}{1-\beta_1}\right)^{1/3}$. Below, we verify that $|c_1|$ is bounded.

If $0\le\xi \le\min(|z|^3, \f12)$, from (\ref{aeq413}), we have
\begin{eqnarray*}
|c_1|&\le&\left(\frac{\max_{0\le\xi\le 1, 0\le z\le 1}|R_3|}{\min_{0\le\xi\le\frac{1}{2}}(1-\beta_1)}\right)^{1/3} 
\le\left(\frac{81}{64}\max_{0\le\xi\le,0\le z\le 1}|R_3|\right)^{1/3}
\le 3.\\
\end{eqnarray*}
Otherwise if $\min(|z|^3, \f12) \le\xi\le 1$, we have
\begin{eqnarray}
\frac{R_3}{1-\beta_1}=\Lambda(\xi)(\frac{1}{4}(-4+\xi-2\xi^2+\xi^3)+\frac{1}{4}(10+6\xi-4\xi^2)z +\frac{1}{4}(-18+6\xi)z^2+\frac{z^3}{\xi}),
\label{aeq415}
\end{eqnarray}
with
\begin{eqnarray}
\frac{216}{125}\le\Lambda(\xi)=\frac{\xi}{\frac{1}{2}-\frac{3}{2}\sqrt{\frac{3}{2}}\sqrt{\frac{2-7\xi+9\xi^2-5\xi^3+\xi^4}{27-32\xi+9\xi^2}}}\le 2, \label{aeq416} \\
\frac{1}{4}\le|\frac{1}{4}(-4+\xi-2\xi^2+\xi^3)+\frac{1}{4}(10+6\xi-4\xi^2)z
+\frac{1}{4}(-18+6\xi)z^2|\le 8,
\label{aeq417} \\
\left|\frac{z^3}{\xi}\right| \le 2.
\label{aeq418}
\end{eqnarray}
In this case we also have $|c_1|\le(2(8+2))^{1/3}\le 3$. 

Now if $x_M \notin I_j$, however there is a local maximum point $x^{loc}_M$ inside $(x_{j-1}, x_{j})$, the above analysis still holds. Therefore we consider that $u(x)$ reaches its local maximum $u^{loc}_M$ at $x^{loc}_M=x_{j-1}$ with $u'_{j-1} \le 0$ and $z=(x_j-x^{loc}_M)/\Delta x=(x_j-x_{j-1})/\Delta x=1$. Based on the Taylor expansion, we can rewrite (\ref{aeq406}) to be
\beq
I_1=\f12(2-\xi)(1-\xi)I_3+\left(1-\f12(2-\xi)(1-\xi)\right)u_{j-1}+\mathcal{O}(\Delta x^4),
\label{aeq419}
\eeq
with
\beq
I_3=u_{j-1}+u'_{j-1}\Delta x + u''_{j-1}\frac{\Delta x^2}{2}\frac{3-2\xi}{3}+u'''_{j-1}\frac{\Delta x^3}{6}\f12(2-\xi)(1-\xi).
\label{aeq420}
\eeq
Similar to (\ref{aeq408}), we consider $I_4$ with the following form
\beqa
I_4 &=& \beta_2 u(x_{j-1}-\Delta x)+(1-\beta_2-\beta_3)u(x_{j-1})+\beta_3 u(x_{j-1}+\Delta x) \nonumber \\
    &=& u_{j-1}+(\beta_3-\beta_2)u'_{j-1}\Delta x+(\beta_2+\beta_3)u''_{j-1}\frac{\Delta x^2}{2}
    +(\beta_3-\beta_2)u'''_{j-1}\frac{\Delta x^3}{6} + \mathcal{O}(\Delta x^4), \nonumber \\
\label{aeq421}
\eeqa
with $\beta_2, \beta_3, \beta_2+\beta_3\in[0, 1]$ to be determined. Since $u'_{j-1}\le 0$, comparing (\ref{aeq420}) with (\ref{aeq421}), we would like to find $\beta_2$, $\beta_3$ satisfying
\beq
(\beta_3-\beta_2)\le 1, \quad \beta_2+\beta_3=\frac{3-2\xi}{3}, \quad \beta_3-\beta_2=\f12(2-\xi)(1-\xi),
\label{aeq422}
\eeq
from which we would have $|I_3-u_M|=\mathcal{O}(\Delta x^4)$ under the assumption that $I_1\ge u_M$ in (\ref{aeq419}), and $I_4\le u_M$.  It can then be established that $|I_1-u_M|=\mathcal{O}(\Delta x^4)$ if $I_1\ge u_M$. (\ref{aeq422}) can be easily achieved with 
\beqa
\beta_2&=&\f12\left(\frac{3-2\xi}{3}+\f12(2-\xi)(1-\xi)\right)=\f{1}{12}(3-\xi)(4-3\xi) 
\label{aeq423}, \\
\quad \beta_3&=&\f12\left(\frac{3-2\xi}{3}-\f12(2-\xi)(1-\xi)\right)=\f{1}{12}\xi(5-3\xi).
\label{aeq424}
\eeqa

Finally, if $x^{loc}_M=x_j$ with $u'_j\ge0$ and $z=(x_j-x^{loc}_M)/\Delta x=(x_j-x_{j})/\Delta x=0$, we have
\beq
I_1=u_j+u'_j\Delta x R_1+u''_j\frac{\Delta x^2}{2} R_2 +u'''_j\frac{\Delta x^3}{6}R_3+\mathcal{O}(\Delta x^4),
\label{aeq425}
\eeq
where
\beq
R_1=-\f12\xi(3-\xi),\quad
R_2=\f16\xi(5-2\xi)(1+\xi), \quad
R_3=-\f14\xi(4-\xi+2\xi^2-\xi^3). \nonumber
\eeq
Now let
\beqa
I_5&=&\beta_4 u_j+(1-\beta_4)u(x_j+c_2\Delta x) \nonumber \\
   &=&u_j+(1-\beta_4)c_2 u'_j\Delta x+(1-\beta_3)c_2^2 u''_j\frac{\Delta x^2}{2}
   +(1-\beta_4)c_2^3 u'''_j\frac{\Delta x^3}{6} + \mathcal{O}(\Delta x^4). \nonumber \\
\label{aeq426}
\eeqa
Similar to the other cases, since $u'_j\ge 0$, 
if we can find $\beta_4 \in [0, 1]$ and $c_2$ bounded, such that,
\beqa
(1-\beta_4)c_2 \ge R_1, \quad
(1-\beta_4)c_2^2= R_2, \quad
(1-\beta_4)c_2^3= R_3, \label{aeq427}
\eeqa
we would have $|I_1-u_M|=\mathcal{O}(\Delta x^4)$, if $I_1\ge u_M$.
(\ref{aeq427}) can be satisfied by letting
\beq
c_2=R_3/R_2, \quad \beta_4=1-R_2/c_2^2,
\label{aeq430}
\eeq
with $|c_2| \le 6/5$ and $\beta_4\in[0,1]$. To this end, we have proved (\ref{aeq401}).
\end{proof}

\section{Numerical simulations}
\label{sec5}
\setcounter{equation}{0}
\setcounter{figure}{0}
\setcounter{table}{0}

In this section, we test the 5th order SL finite difference WENO scheme with
the parametrized MPP flux limiters, denoted as ``SL-WL", to solve the Vlasov equation. 
We will compare it with the 5th order finite difference WENO scheme with the 4th order RK  time discretization and MPP flux limiters in \cite{xiong2013parametrized}, denoted as ``RK-WL". These two methods without MPP flux limiters are denoted to be ``SL-WO" and ``RK-WO", respectively. A fast Fourier transform (FFT) method is used to solve the 1D Poisson equation. For the Vlasov equation, the time step at time level $t^n$ is chosen to be $\Delta t=CFL /(\alpha_x/\Delta x + \alpha^n_v/\Delta v)$, where $\alpha_x=\max_{j} |v_j|$ and $\alpha^n_v=\max_{i}|E(x_i,t^n)|$, and we take the CFL number
to be $0.8$ for the SL method and $0.6$ for the RK time method, unless specified. 

\begin{exa}
In the first example, we take the advection equation
\begin{equation}
\partial_t f + \nabla_x f + \nabla_v f=0,
\label{advect}
\end{equation}
with initial condition
\begin{equation}
f(0,x,v)=\sin^4(x+v),
\label{adinit}
\end{equation}
and periodic boundary condition on both directions on the domain $[0, 2\pi]\times[-\pi,\pi]$, to test the orders of accuracy for the SL WENO scheme with the parametrized MPP flux limiters.
In Table \ref{tab1}, the $L^1$ and $L^{\infty}$ errors and orders are shown for the
SL WENO scheme at time $T=1$, here the time step is $\Delta t=CFL \Delta x/2$ with $CFL=0.8$ and $CFL=2.2$. The expected $5th$ order accuracy is observed. 
With limiters, the numerical solutions are strictly positive with $f_{min} \ge 0$.

\begin{table}
\centering
\caption{$L^1$ and $L^{\infty}$ errors and orders for the advection equation (\ref{advect})
with initial condition (\ref{adinit}) at $T=1$. ``WL" denotes the scheme with limiters, ``WO" denotes the scheme without limiters.
``$f_{min}$" is the minimum numerical solution of the SL WENO schemes. Mesh size $N_x=N_v=N$.}
\vspace{0.2cm}
  \begin{tabular}{|c|c|c|c|c|c|c|c|}
    \hline
&    & $N$  &  $L^1$ error &    order& $L^\infty$ error & order &  $f_{min}$ \\ \hline
\multirow{8}{*}{$CFL=0.8$}&\multirow{4}{*}{WL} & 
         40 &     2.74E-03 &       --&     5.53E-03 &       --&     1.97E-13 \\ \cline{3-8}
&    &   80 &     2.35E-04 &     3.55&     7.95E-04 &     2.80&     3.48E-05 \\ \cline{3-8}
&    &  160 &     7.16E-06 &     5.04&     3.51E-05 &     4.50&     1.09E-13 \\ \cline{3-8}
&    &  320 &     1.95E-07 &     5.20&     8.99E-07 &     5.29&     8.60E-09 \\ \cline{2-8}
&\multirow{4}{*}{WO} & 
         40 &     2.86E-03 &       --&     6.06E-03 &       --&    -1.76E-03 \\ \cline{3-8}
&    &   80 &     2.68E-04 &     3.42&     9.67E-04 &     2.65&    -2.20E-04 \\ \cline{3-8}
&    &  160 &     7.52E-06 &     5.15&     3.42E-05 &     4.82&    -1.03E-05 \\ \cline{3-8}
&    &  320 &     1.95E-07 &     5.27&     8.99E-07 &     5.25&    -3.75E-08 \\ \hline
\multirow{8}{*}{$CFL=2.2$}&\multirow{4}{*}{WL} & 
         40 &     1.68E-03 &       --&     3.75E-03 &       --&     1.94E-13 \\ \cline{3-8}
&    &   80 &     9.87E-05 &     4.09&     3.76E-04 &     3.32&     1.25E-05 \\ \cline{3-8}
&    &  160 &     2.90E-06 &     5.09&     1.51E-05 &     4.64&     3.64E-07 \\ \cline{3-8}
&    &  320 &     7.61E-08 &     5.25&     3.55E-07 &     5.41&     5.27E-10 \\ \cline{2-8}
&\multirow{4}{*}{WO} &     
         40 &     1.80E-03 &       --&     3.83E-03 &       --&    -8.45E-04 \\ \cline{3-8}
&    &   80 &     1.20E-04 &     3.91&     4.91E-04 &     2.96&    -1.21E-04 \\ \cline{3-8}
&    &  160 &     2.94E-06 &     5.35&     1.50E-05 &     5.03&    -3.21E-06 \\ \cline{3-8}
&    &  320 &     7.61E-08 &     5.27&     3.55E-07 &     5.40&    -1.48E-08 \\ \hline
  \end{tabular}
\label{tab1}
\end{table}


\end{exa}

\begin{exa}
In the second example, we consider the rigid body rotating problem 
\begin{equation}
\partial_t f - v \nabla_x f + x \nabla_v f=0.
\label{rigid}
\end{equation}
First we choose a similar smooth initial condition as in \cite{guo2013hybrid}
\begin{equation}
f(0,x,v)=
\begin{cases}
\cos^6(r), \quad & r<\pi/2, \\
0, \quad & \text{otherwise}.
\end{cases}
\label{rigidinit}
\end{equation}
on the computational domain $[-\pi,\pi]\times[-\pi,\pi]$ with periodic boundary condition on both directions, where $r=\sqrt{x^2+y^2}$. Similarly as the first example, we use the SL WENO scheme with or without the parametrized MPP flux limiters to compute up to time $T=2\pi$ for a period. In Table \ref{tab12}, the $L^1$ and $L^{\infty}$ errors and orders are shown for the time step 
$\Delta t=CFL \Delta x/(2\pi)$ with $CFL=0.8$ and $CFL=2.2$. Due to the symmetric property of the initial condition, the 5th order spatial error dominates. With the limiters, the numerical solutions are strictly positive.

\begin{table}
\centering
\caption{$L^1$ and $L^{\infty}$ errors and orders for the rigid body rotating problem (\ref{rigid}) with initial condition (\ref{rigidinit}) at $T=2 \pi$. ``WL" denotes the scheme with limiters, ``WO" denotes the scheme without limiters.
``$f_{min}$" is the minimum numerical solution of the SL WENO schemes. Mesh size $N_x=N_v=N$.}
\vspace{0.2cm}
  \begin{tabular}{|c|c|c|c|c|c|c|c|}
    \hline
&    & $N$  &  $L^1$ error &    order& $L^\infty$ error & order &  $f_{min}$ \\ \hline
\multirow{8}{*}{$CFL=0.8$}&\multirow{4}{*}{WL} 
     &  40 &     5.56E-04 &       --&     1.56E-02 &       --&     1.22E-13 \\ \cline{3-8}
&    &  80 &     3.69E-05 &     3.91&     7.17E-04 &     4.44&     4.44E-14 \\ \cline{3-8}
&    & 160 &     1.32E-06 &     4.80&     2.31E-05 &     4.95&     4.75E-16 \\ \cline{3-8}
&    & 320 &     3.82E-08 &     5.11&     7.22E-07 &     5.00&     1.53E-19 \\ \cline{2-8}
&\multirow{4}{*}{WO} 
     &  40 &     5.68E-04 &       --&     1.56E-02 &       --&    -8.32E-05 \\ \cline{3-8}
&    &  80 &     4.05E-05 &     3.81&     7.17E-04 &     4.44&    -4.48E-05 \\ \cline{3-8}
&    & 160 &     1.33E-06 &     4.93&     2.31E-05 &     4.95&    -7.41E-08 \\ \cline{3-8}
&    & 320 &     3.83E-08 &     5.11&     7.22E-07 &     5.00&    -2.63E-09 \\ \hline
\multirow{8}{*}{$CFL=2.2$}&\multirow{4}{*}{WL} 
     &  40 &     5.31E-04 &       --&     1.55E-02 &       --&     1.83E-13 \\ \cline{3-8}
&    &  80 &     3.62E-05 &     3.88&     7.07E-04 &     4.46&     8.45E-15 \\ \cline{3-8}
&    & 160 &     1.29E-06 &     4.81&     2.28E-05 &     4.96&     7.54E-18 \\ \cline{3-8}
&    & 320 &     3.72E-08 &     5.11&     7.10E-07 &     5.00&     1.89E-23 \\ \cline{2-8}
&\multirow{4}{*}{WO} 
     &  40 &     5.44E-04 &       --&     1.55E-02 &       --&    -8.45E-05 \\ \cline{3-8}
&    &  80 &     3.97E-05 &     3.78&     7.07E-04 &     4.46&    -4.45E-05 \\ \cline{3-8}
&    & 160 &     1.29E-06 &     4.94&     2.28E-05 &     4.96&    -7.03E-08 \\ \cline{3-8}
&    & 320 &     3.73E-08 &     5.11&     7.10E-07 &     5.00&    -2.49E-09 \\ \hline
  \end{tabular}
\label{tab12}
\end{table}

Then we consider the initial condition that includes a slotted disk, a cone as well as a smooth hump
on the computational domain $[-\pi,\pi]\times[-\pi,\pi]$ as in \cite{xiong2013parametrized}, see Fig. \ref{fig01}. For this example, in order to clearly observe the difference of the scheme between with and without the MPP flux limiters, we use the SL WENO scheme but with the linear weights (\ref{linearw1}) and (\ref{linearw2})
instead of the nonlinear weights (\ref{nonlinearw}).
In Fig. \ref{fig02}, we display the cuts of the numerical solution at $T=12\pi$ under the time step $\Delta t=CFL \Delta x/(2\pi)$ with $CFL=0.8$. With limiters, the numerical
solutions are within the range of $[0, 1]$; they are not if without limiters, similar to the results of the RK WENO method with linear weights in \cite{xiong2013parametrized}. For the results of the SL WENO scheme
with nonlinear weights, the minimum and maximum values of the numerical solutions are $0.0000000000000$ and  $0.9999999999998$ up to $13$ decimal places, while they are $-0.0020444053893$ and $1.014419874020$ if without limiters. We omit the figures to save space.          

\begin{figure}
\centering
\includegraphics[width=3in,clip]{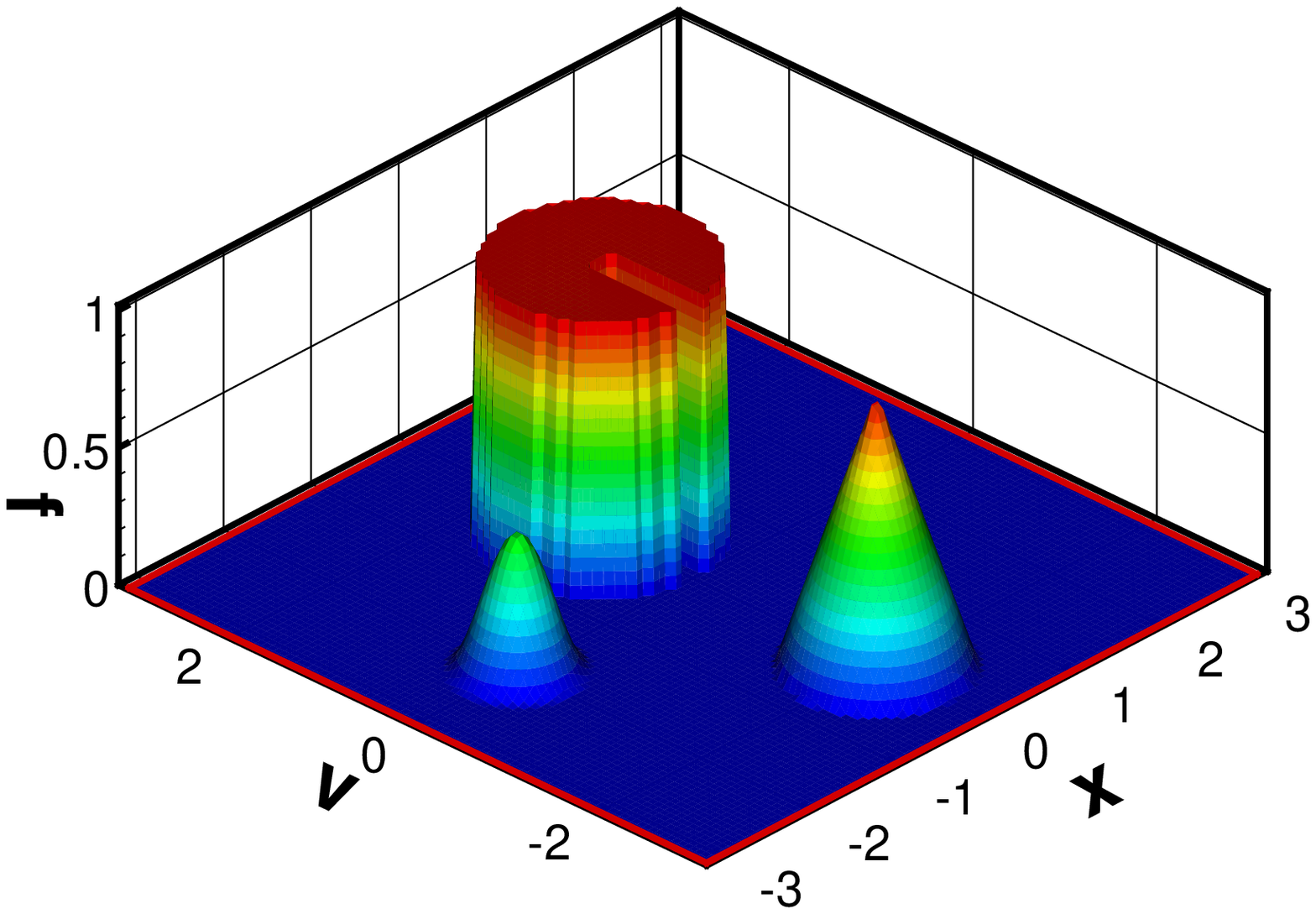},
\includegraphics[width=3in,clip]{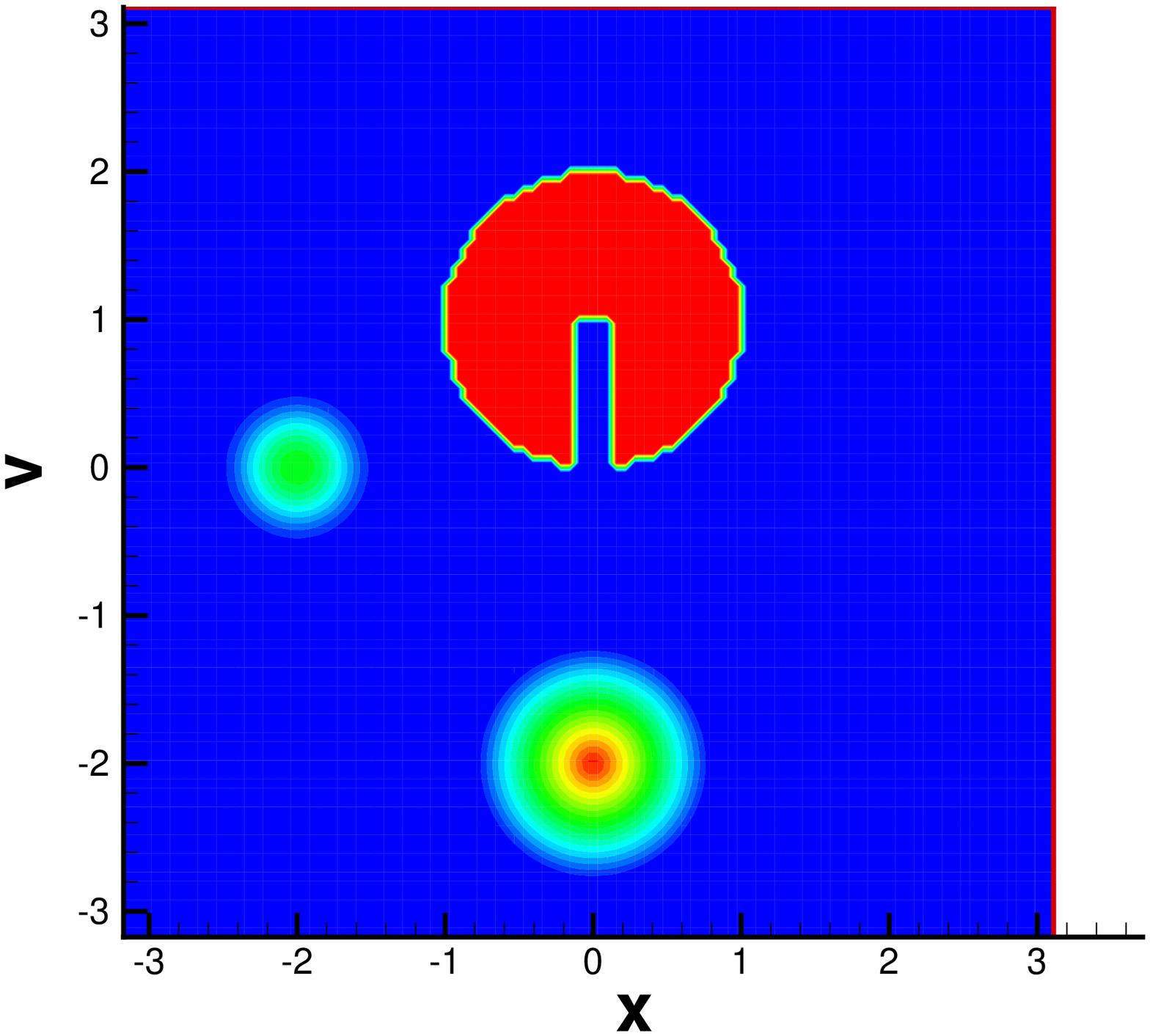}
\caption{The initial profile for equation (\ref{rigid}). $N_x \times N_v = 100 \times 100$.}
\label{fig01}
\end{figure}

\begin{figure}
\centering
\includegraphics[width=3in,clip]{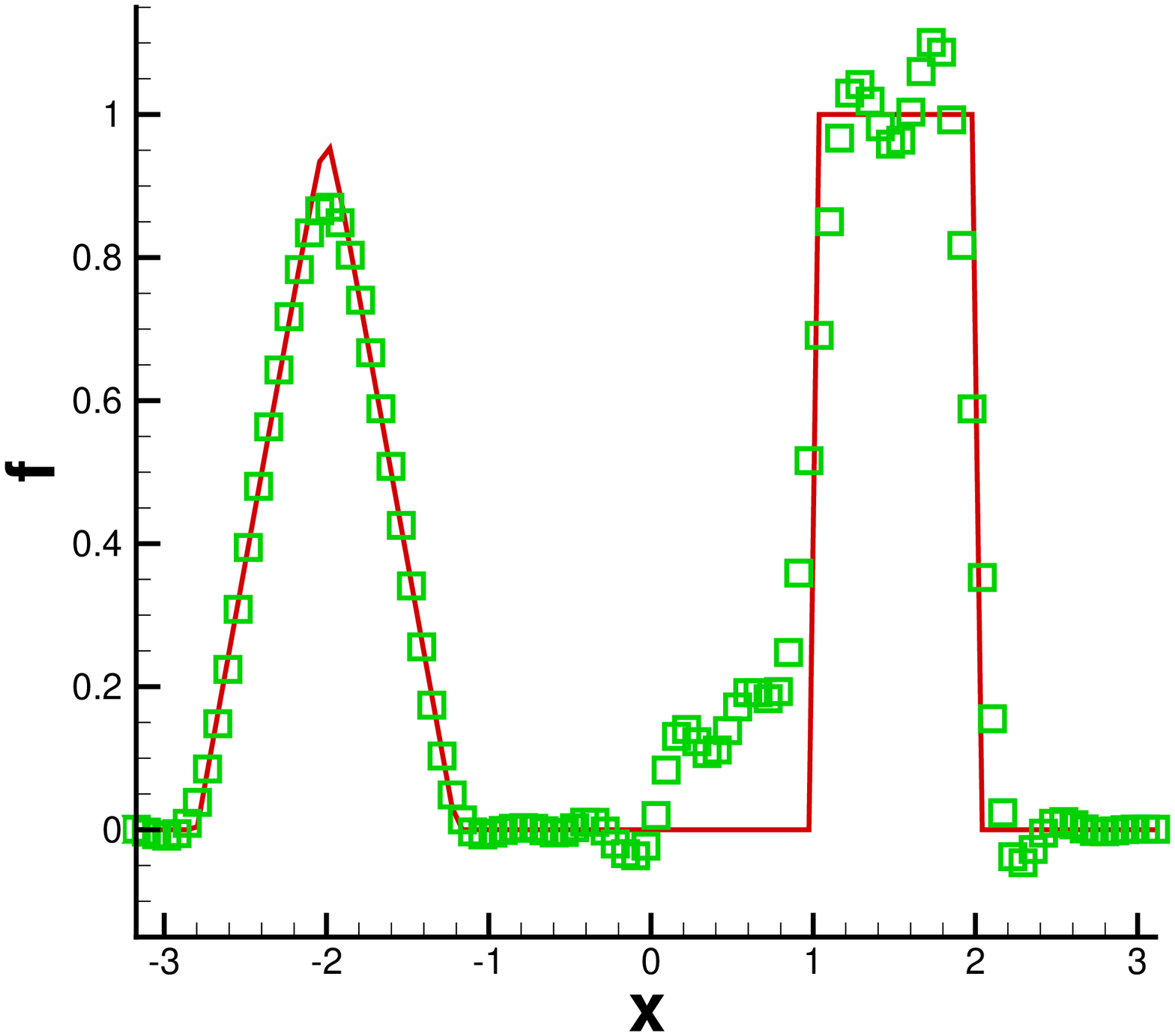},
\includegraphics[width=3in,clip]{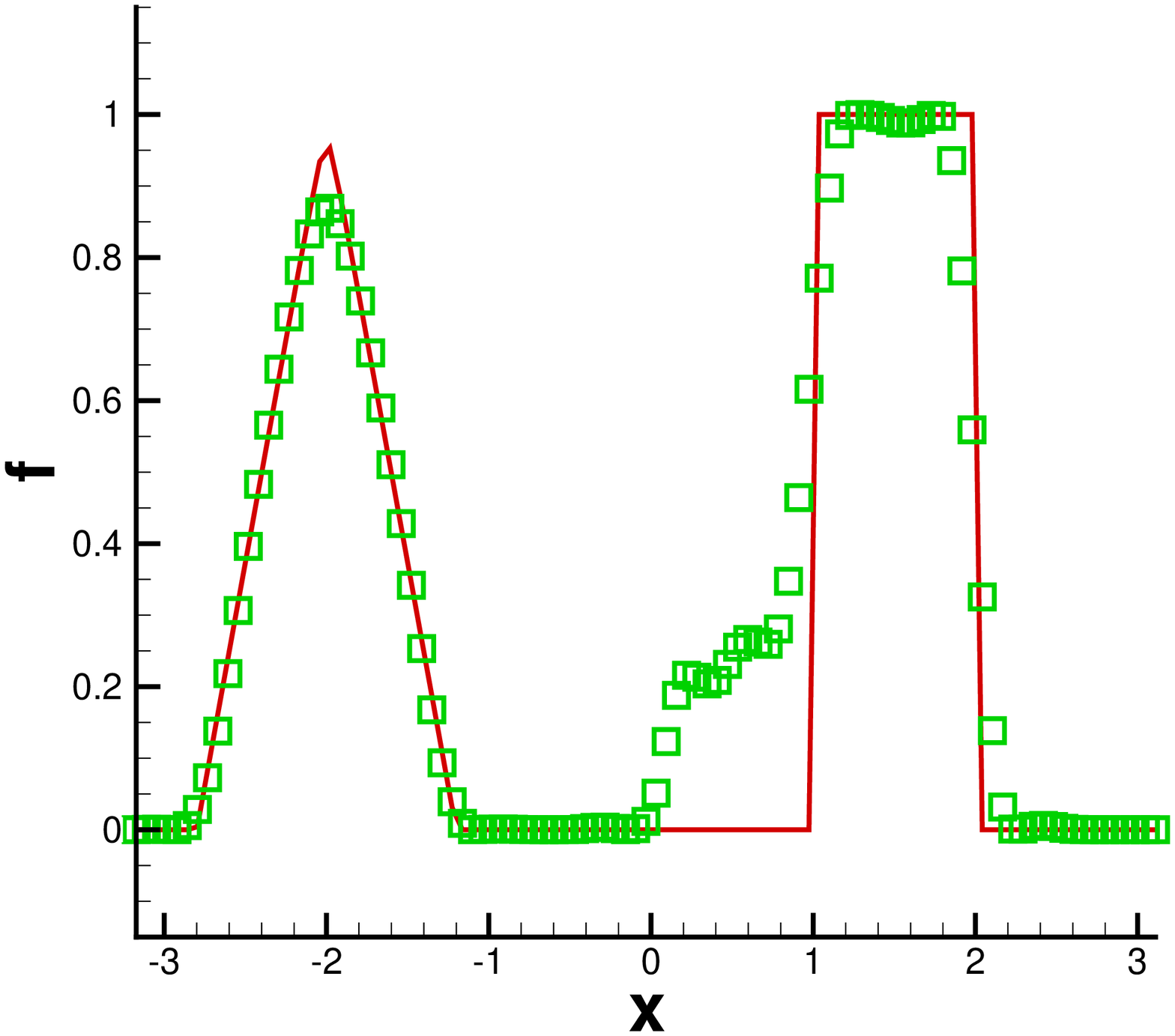}\\
\includegraphics[width=3in,clip]{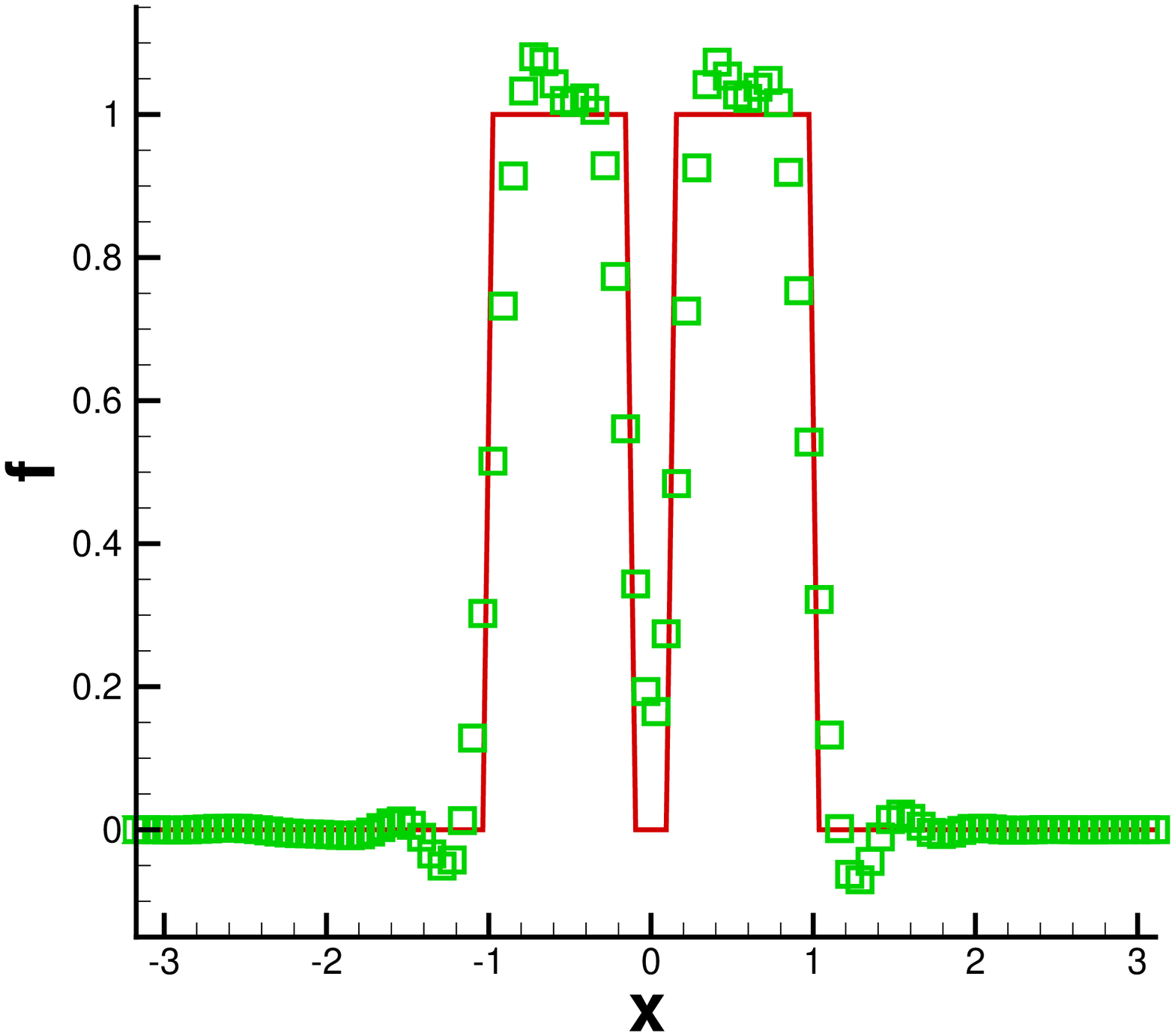},
\includegraphics[width=3in,clip]{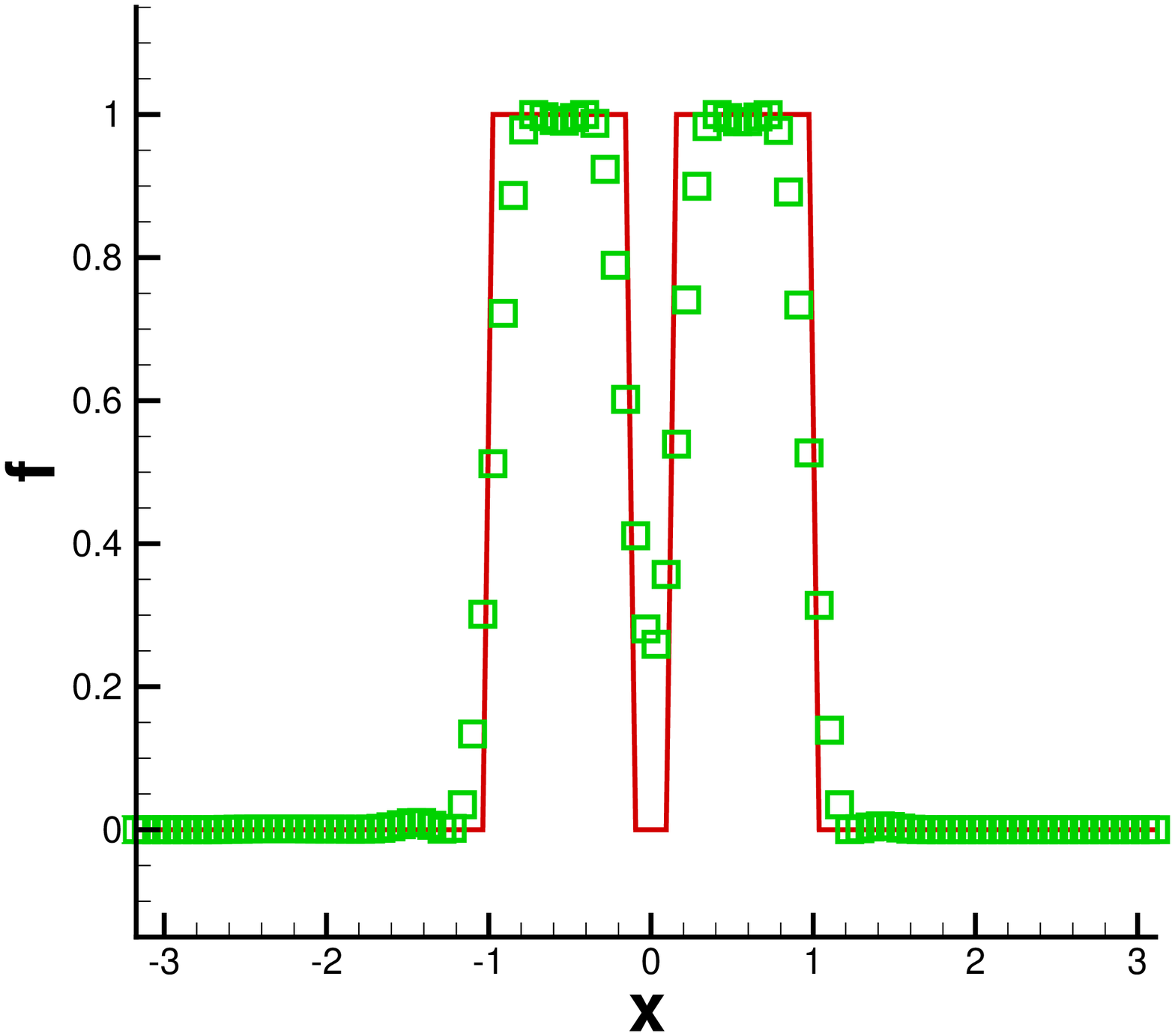}\\
\includegraphics[width=3in,clip]{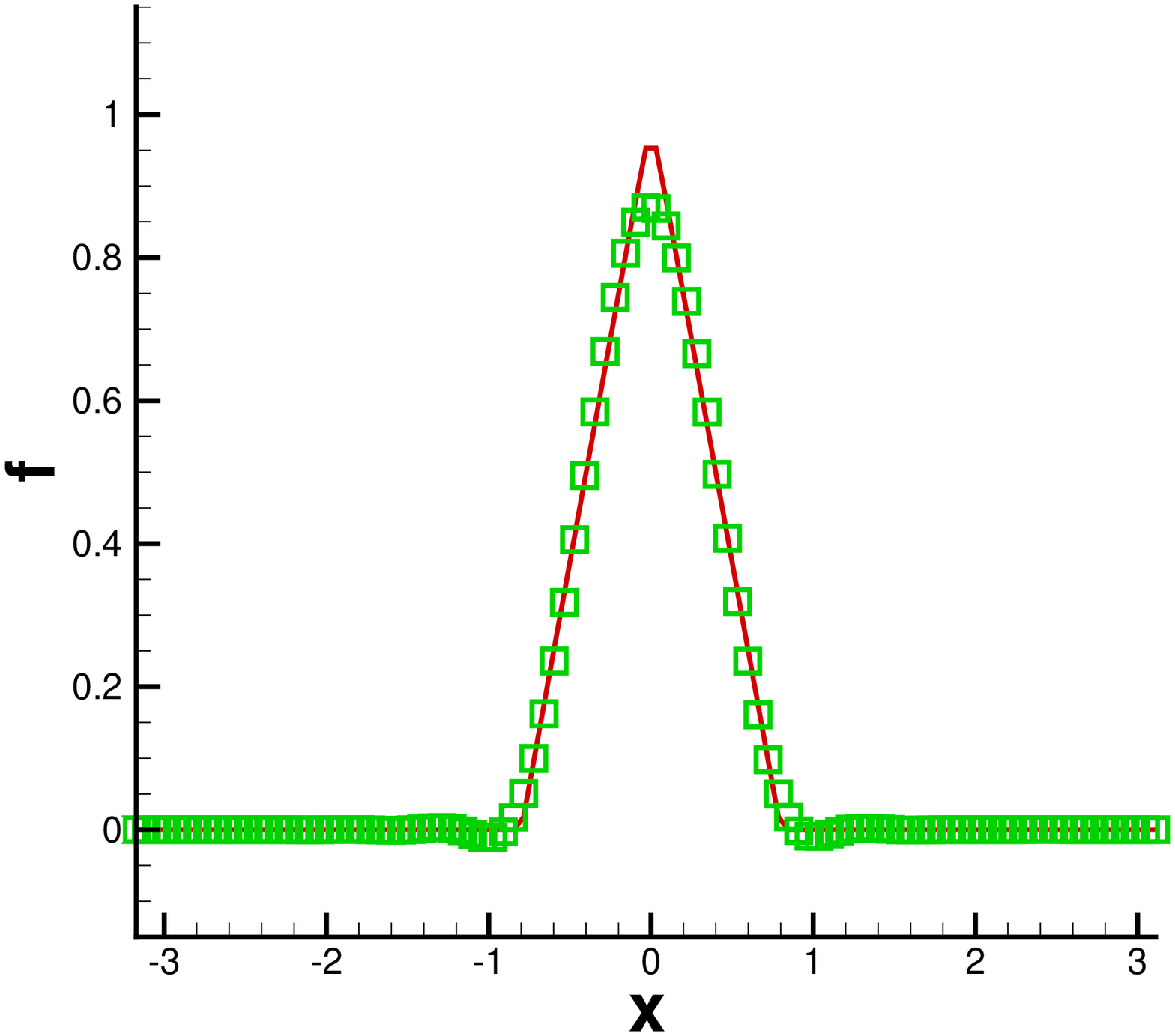},
\includegraphics[width=3in,clip]{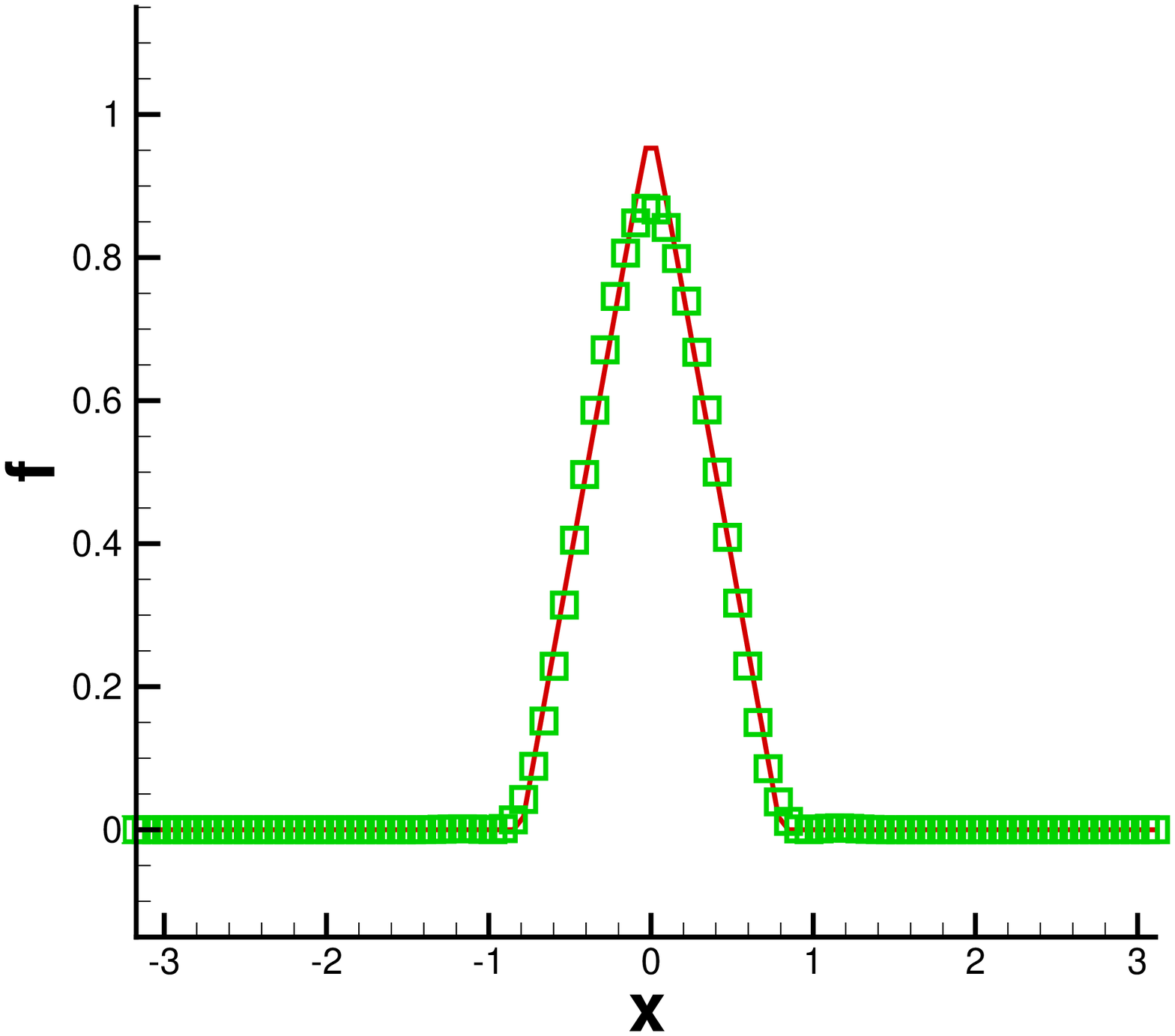}\\
\caption{Slides of numerical solutions for equation (\ref{rigid}) with initial condition in Fig. \ref{fig01} at $T=12\pi$. $ N_x \times N_v = 100 \times 100 $. Left: without limiter; Right: with limiter.
Cuts along $x=0$, $y=0.8$ and $y=-2$ from top to bottom, respectively. Solid line: exact solution; Symbols: numerical solution.}
\label{fig02}
\end{figure}

\end{exa}

Next, we consider solving the VP system (\ref{eq:vp1}) and (\ref{eq:vp2}). Some classical quantities for the probability distribution function are known to be conserved in time:
\begin{itemize}
\item $L^p$ norm, $1\le p \le \infty$:
\beq
\label{lp}
\|f\|_p=\left(\int_x\int_v|f(t,x,v)|^pdxdv\right)^{1/p},
\eeq

\item Energy:
\beq
\label{energy}
\text{Energy} = \int_v\int_x f(t,x,v)v^2 dx dv +\int_x E^2(t,x)dx,
\eeq

\item Entropy:
\beq
\label{entropy}
\text{Entropy} = \int_v \int_x f(t,x,v)log(f(x,v,t))dxdv.
\eeq
\end{itemize}
In our simulations, we measure the evolution of these quantities to demonstrate its
good performance.
With the parametrized MPP flux limiter, the numerical solution of $f(t,x,v)$ 
would be nonnegative, and the numerical schemes are also conservative themselves, 
the $L^1$ norm is expected to be a constant up to the machine error. Schemes without the 
MPP flux limiters may produce negative $f(t,x,v)$ somewhere, and could not conserve the $L^1$ norm.  
In the following computation, we take $V_C=2\pi$ unless otherwise specified, 
and the mesh is $N_x=80$ and $N_v=160$ with periodic boundary conditions on both
directions.

\begin{exa}
We first consider the VP system, with the following initial condition:
\beq
f(0,x,v)=\f{1}{\sqrt{2\pi}}\cos^4(k x)\exp(-\f{v^2}{2}),
\label{vpsmooth}
\eeq
with periodic boundary condition on both directions on the computational domain
$[0, L]\times[-V_c,V_c]$, where $k=0.5$ and $L=2\pi/k=4\pi$. This example is specifically designed to demonstrate that high order accuracy of the original SL scheme is preserved with the MPP flux limiter when solving the VP system. We take $V_c=20$ to minimize the error from the domain cutoff. In Table \ref{tab2}, the $L^1$ and $L^{\infty}$ errors and orders are shown for the SL WENO scheme with $CFL=0.8$ and $CFL=2.2$ at $T=0.01$. $5th$ order accuracy is observed for both $CFL$ conditions. The minimum numerical solutions $f_{min}$ 
are around machine error if with limiters, otherwise it could be negative.

\begin{table}
\centering
\caption{$L^1$ and $L^{\infty}$ errors and orders for the VP system (\ref{eq:vp1}) and (\ref{eq:vp2})
with initial condition (\ref{vpsmooth}) at $T=0.01$. ``WL" denotes the scheme with limiters, ``WO" denotes the scheme without limiters. ``$f_{min}$" is the minimum of the numerical solution of the SL WENO schemes. Mesh size $N_v=2N_x$.}
\vspace{0.2cm}
  \begin{tabular}{|c|c|c|c|c|c|c|c|}
    \hline
&   &$N_x$&  $L^1$ error &    order& $L^\infty$ error & order &  $f_{min}$ \\ \hline
\multirow{8}{*}{$CFL=0.8$}&\multirow{4}{*}{WL}      
    &  40 &     1.07E-06 &       --&     5.22E-05 &       --&     4.60E-73 \\ \cline{3-8}
&   &  80 &     4.05E-08 &     4.72&     1.89E-06 &     4.79&    -1.96E-84 \\ \cline{3-8}
&   & 160 &     1.15E-09 &     5.13&     6.55E-08 &     4.85&     1.84E-87 \\ \cline{3-8}
&   & 320 &     3.32E-11 &     5.12&     1.80E-09 &     5.18&     9.90E-88 \\ \cline{2-8}
    & \multirow{4}{*}{WO} 
    &  40 &     1.08E-06 &       --&     5.22E-05 &       --&    -3.19E-06 \\ \cline{3-8}
&   &  80 &     4.05E-08 &     4.74&     1.89E-06 &     4.79&    -1.89E-10 \\ \cline{3-8}
&   & 160 &     1.15E-09 &     5.13&     6.55E-08 &     4.85&    -1.21E-10 \\ \cline{3-8}
&   & 320 &     3.32E-11 &     5.12&     1.80E-09 &     5.18&    -1.95E-11 \\ \hline
\multirow{8}{*}{$CFL=2.2$}&\multirow{4}{*}{WL}
    &  40 &     1.07E-06 &       --&     5.22E-05 &       --&    -7.90E-79 \\ \cline{3-8}
&   &  80 &     4.04E-08 &     4.72&     1.89E-06 &     4.79&    -4.45E-86 \\ \cline{3-8}
&   & 160 &     1.15E-09 &     5.13&     6.46E-08 &     4.87&     1.85E-87 \\ \cline{3-8}
&   & 320 &     3.79E-11 &     4.93&     1.70E-09 &     5.25&     9.92E-88 \\ \cline{2-8}
    &\multirow{4}{*}{WO} 
    &  40 &     1.08E-06 &       --&     5.22E-05 &       --&    -3.17E-06 \\ \cline{3-8}
&   &  80 &     4.04E-08 &     4.74&     1.89E-06 &     4.79&    -1.76E-10 \\ \cline{3-8}
&   & 160 &     1.15E-09 &     5.13&     6.46E-08 &     4.87&    -1.18E-10 \\ \cline{3-8}
&   & 320 &     3.79E-11 &     4.93&     1.70E-09 &     5.25&    -1.92E-11 \\ \hline
  \end{tabular}
\label{tab2}
\end{table}

\end{exa}

\begin{exa} (Landau damping)
We consider the Landau damping for the VP system
with the initial condition:
\beq
\label{landau}
f(0,x,v)=\f{1}{\sqrt{2\pi}}(1+\alpha \cos(k x))\exp(-\f{v^2}{2}),
\eeq
The length of the domain in the $x$-direction is $L=\f{2\pi}{k}$. 
For the strong Landau damping with $\alpha=0.5$ and $k=0.5$, we plot the time evolution
of electric field in $L^2$ norm and $L^\infty$ norm in Figure \ref{fig21}, with the linear decay rate $\gamma_1=-0.2812$ and $\gamma_2=0.0770$ \cite{guo2013hybrid, cheng1976integration} plotted in the same plots.
Time evolution of the relative deviations of discrete $L^1$ norm, $L^2$ norm, kinetic energy and
entropy are reported in Figures \ref{fig22}. No much difference is observed between the performance of schemes with RK and SL time discretization.
For this case, with the MPP flux limiter, the $L^1$ norm is preserved up to the machine error, however it is not for schemes without the limiter. For the case of weak Landau damping with $\alpha=0.01$ and $k=0.5$,
little difference can be seen between with and without limiters, we omit them here due to similarity.

\begin{figure}
\centering
\includegraphics[width=3in,clip]{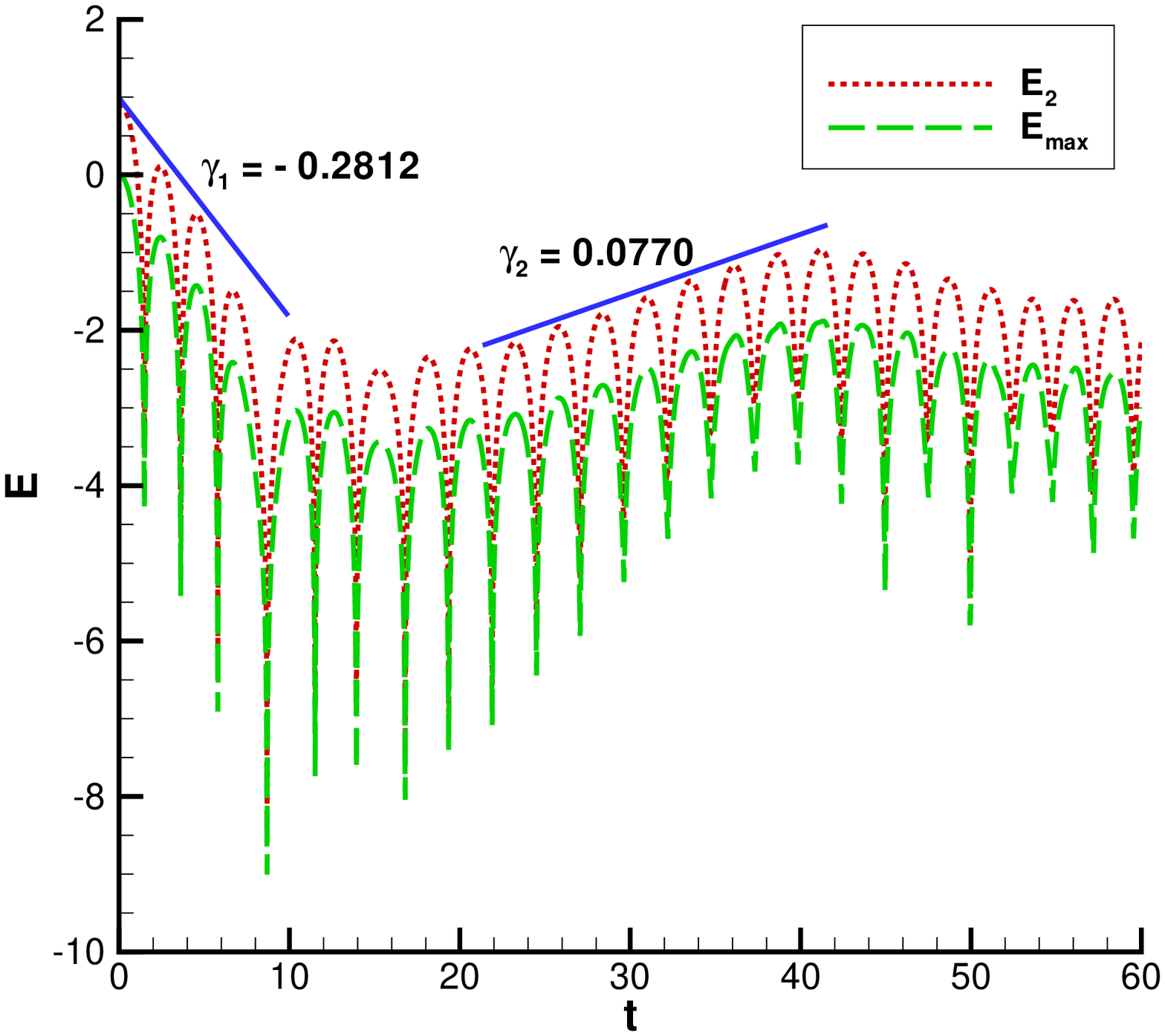},
\includegraphics[width=3in,clip]{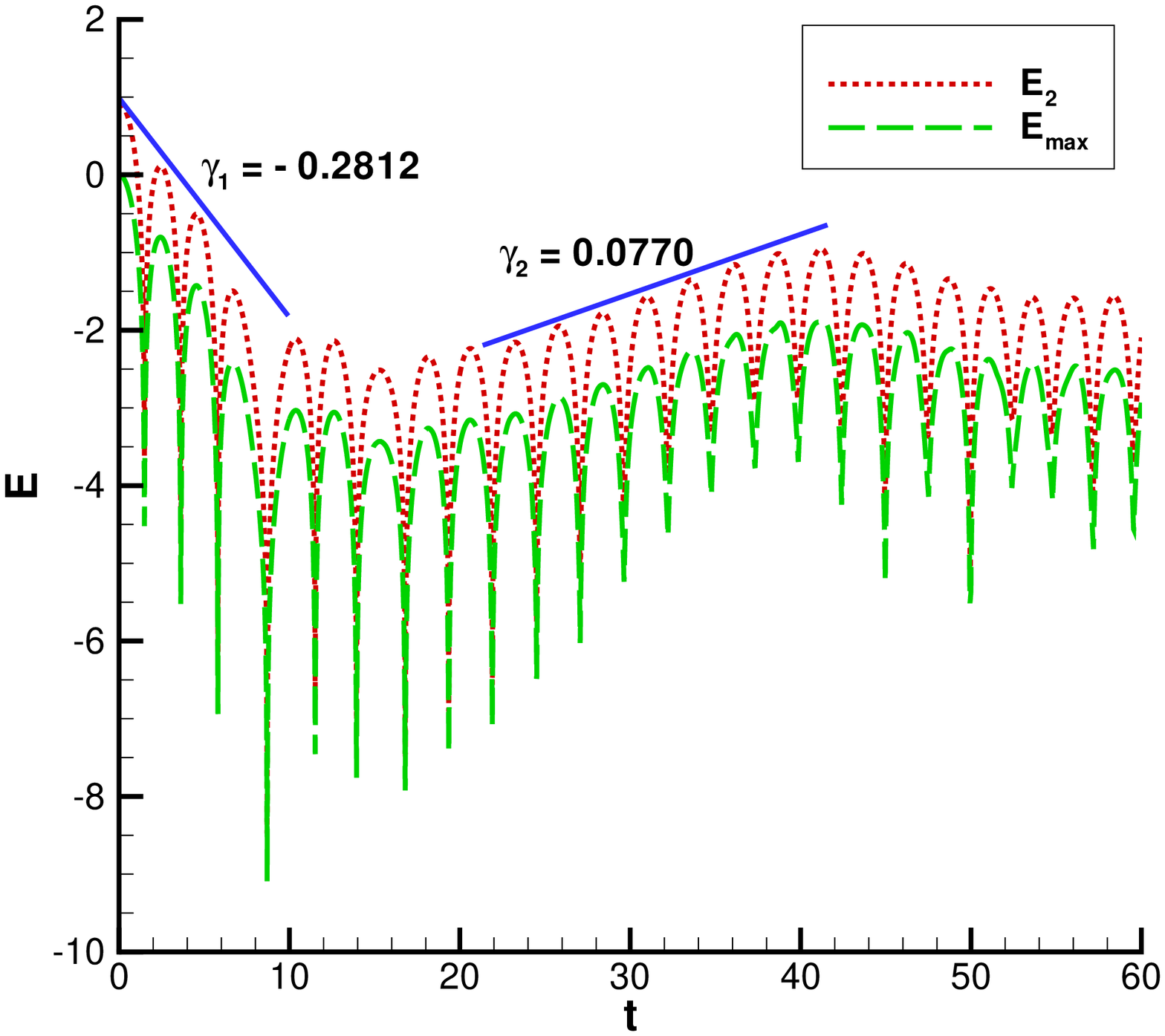}\\
\caption{Strong Landau damping with the initial condition (\ref{landau}). 
Time evolution of the electric field in $L^2$ norm ($E_2$) and $L^\infty$ norm ($E_{max}$) (logarithmic value).
Mesh: $N_x \times N_v = 80 \times 160 $. Left: SL-WL; Right: RK-WL.}
\label{fig21}
\end{figure}

\begin{figure}
\centering
\includegraphics[width=3in,clip]{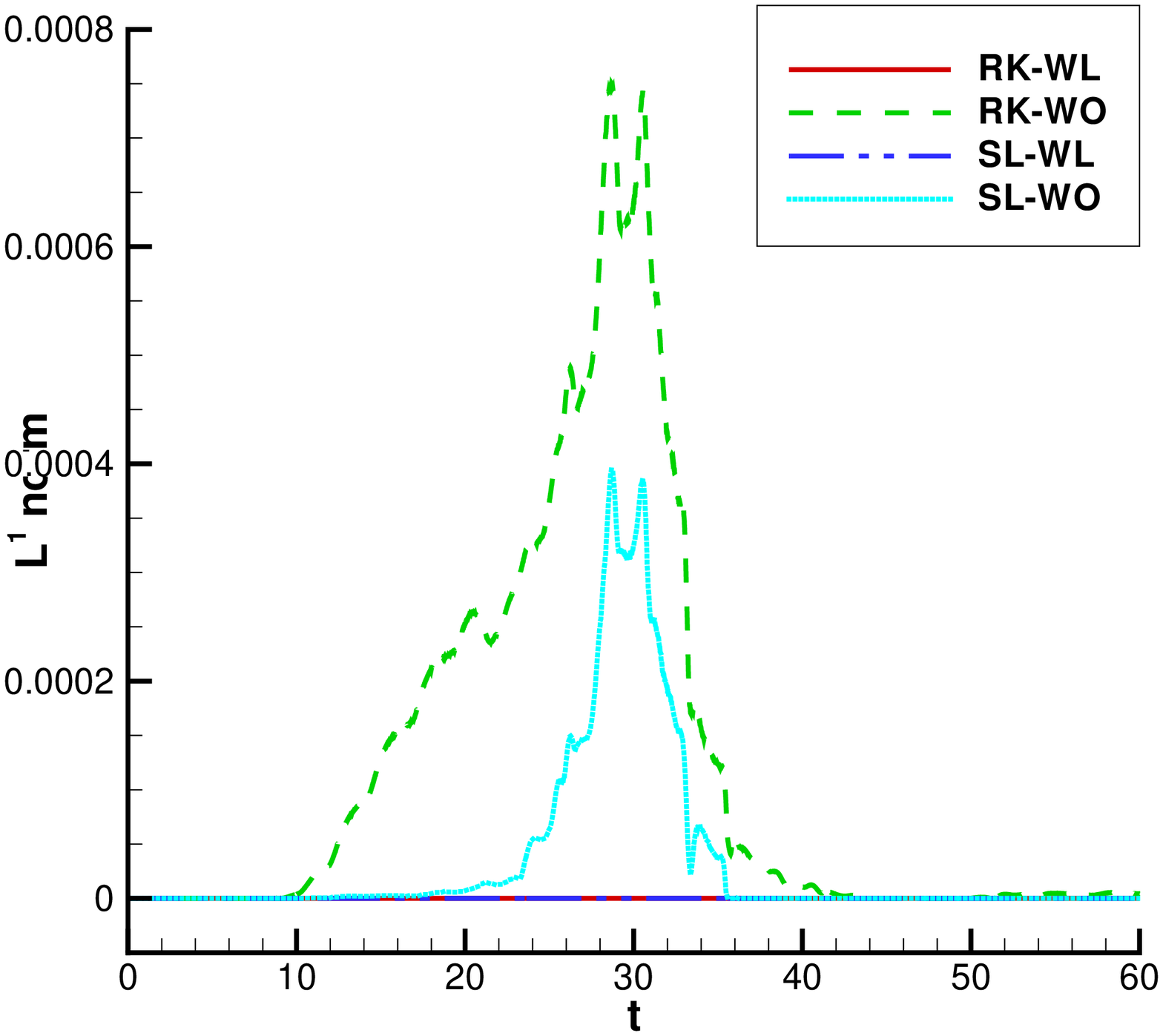},
\includegraphics[width=3in,clip]{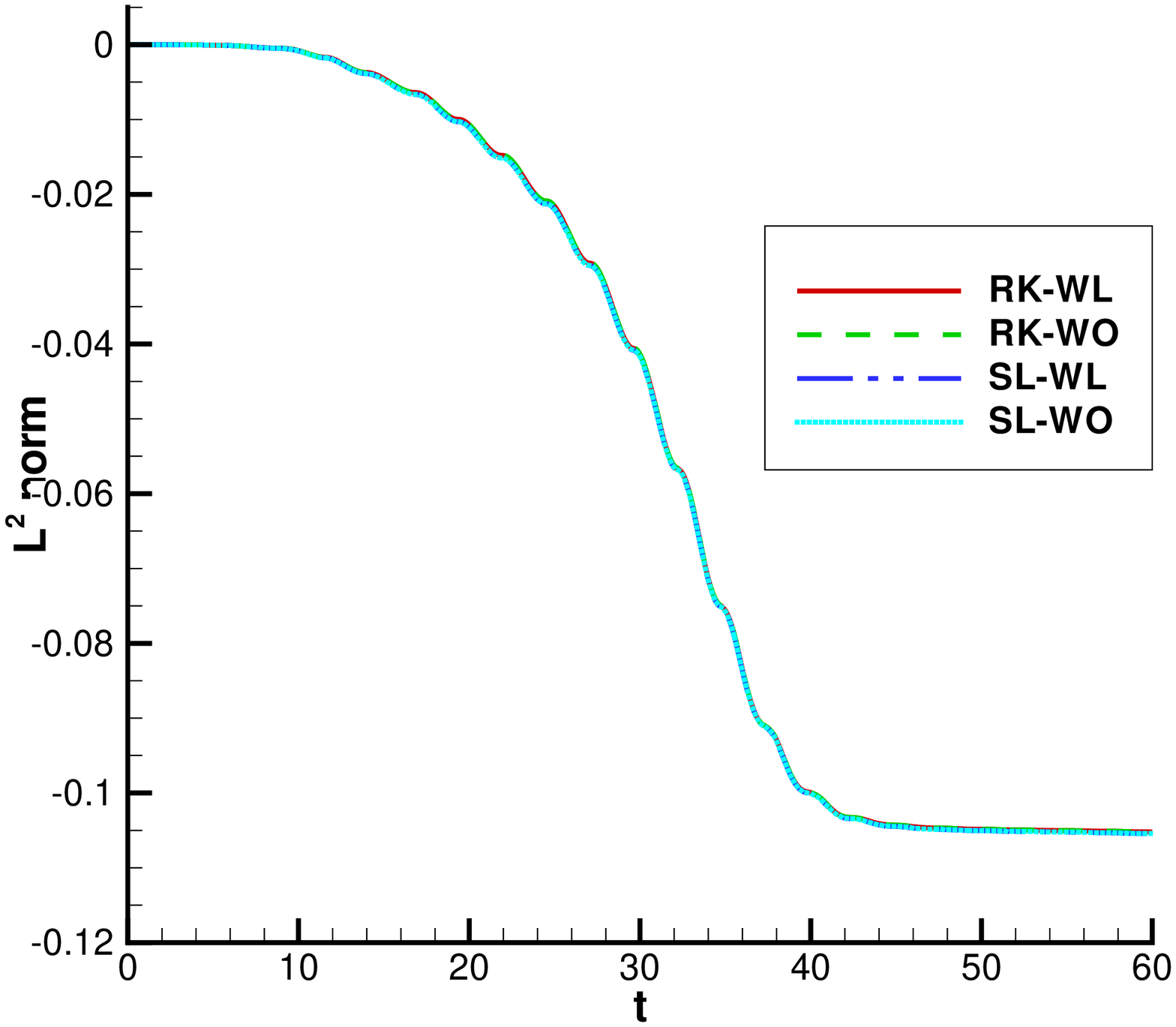}\\
\includegraphics[width=3in,clip]{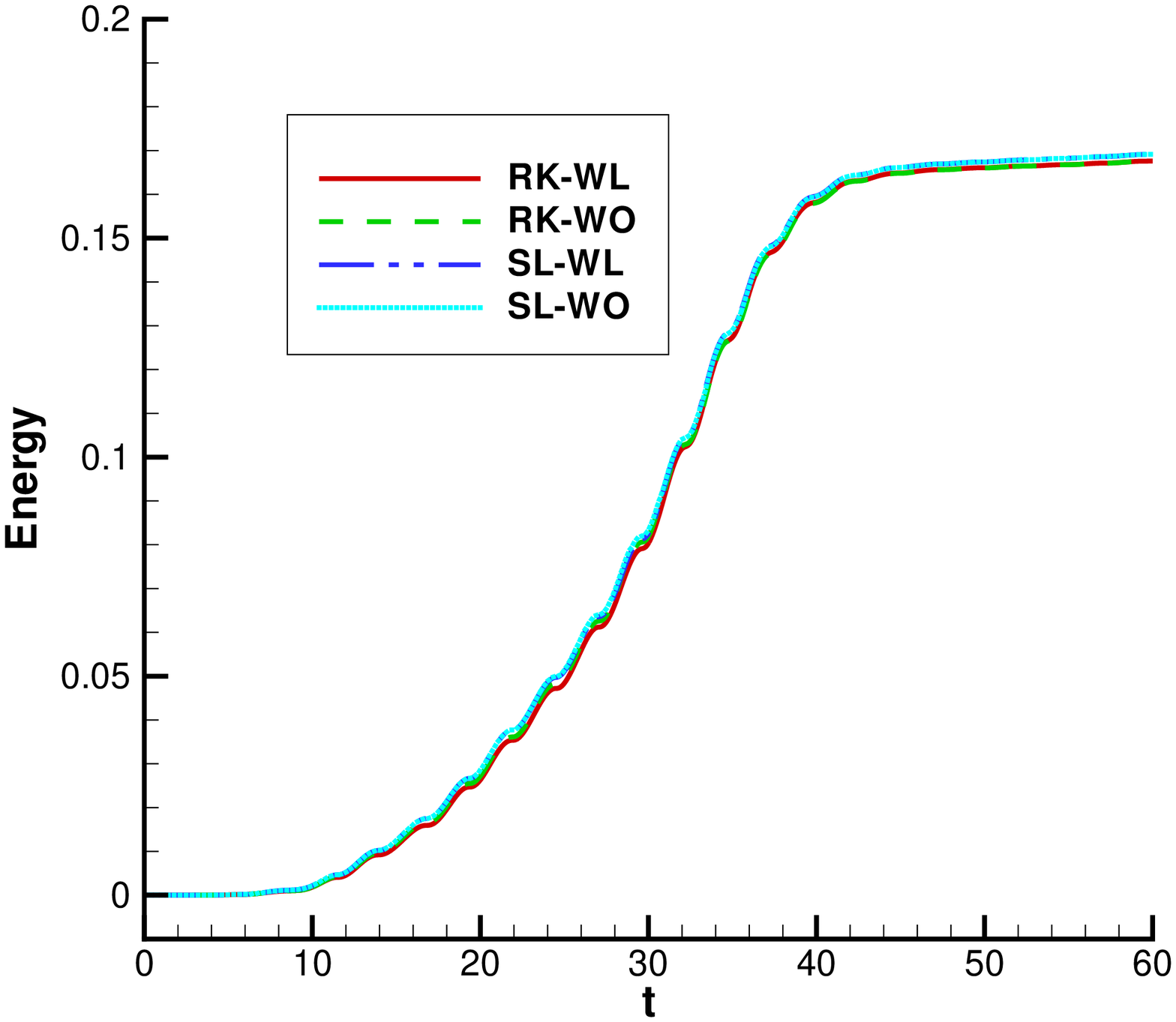},
\includegraphics[width=3in,clip]{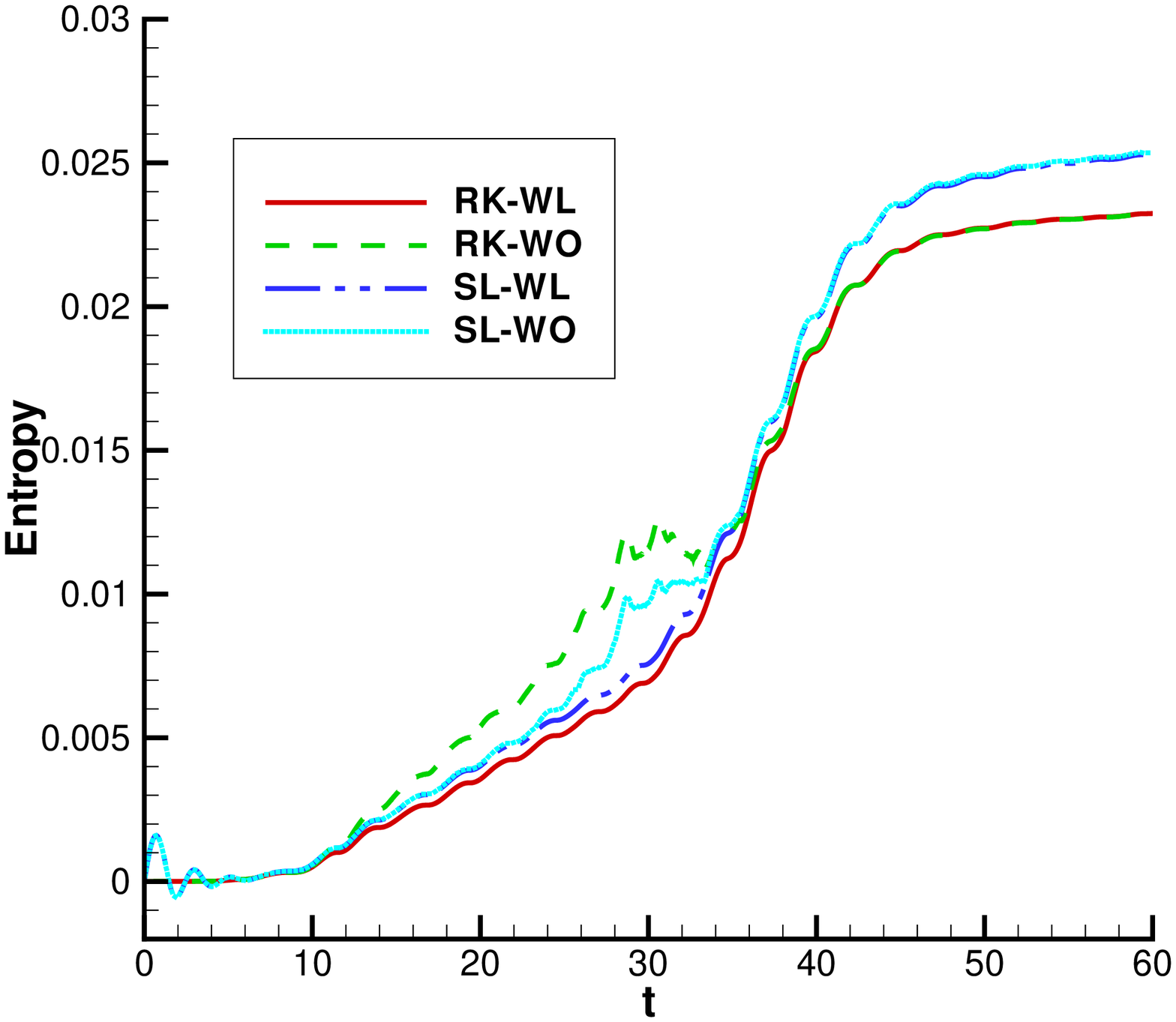}\\
\caption{Strong Landau damping with the initial condition (\ref{landau}). Time evolution of the relative
deviations of discrete $L^1$ norm (upper left), $L^2$ norm (upper right), kinetic
energy (lower left) and entropy (lower right). $N_x \times N_v = 80 \times 160 $.}
\label{fig22}
\end{figure}

\end{exa}

\begin{exa} (Two stream instability)
We consider the symmetric two stream instability problem with the initial condition
\beq
\label{s2stream}
f(0,x,v)=\f{1}{2v_{th}\sqrt{2\pi}}\left[\exp\left(-\f{(v-u)^2}{2v_{th}^2}\right)
+\exp\left(-\f{(v+u)^2}{2v_{th}^2}\right)\right](1+\alpha\cos(kx)),
\eeq
where $\alpha=0.05$, $u=0.99$, $v_{th}=0.3$ and $k=\f{2}{13}$.
The time evolution
of the electric field in $L^2$ norm and $L^\infty$ norm is plotted in Figure \ref{fig41}.
The time evolution of the relative deviations of discrete $L^1$ norm, $L^2$ norm, kinetic energy and
entropy are reported in Figures \ref{fig42}. The MPP flux limiters play a very
good role on the $L^1$ norm, with very less effect on other discrete properties.

Similarly, the two stream instability problem with an unstable initial distribution
function
\beq
\label{2stream}
f(0,x,v)=\f{2}{7\sqrt{2\pi}}(1+5v^2)(1+\alpha((\cos(2kx)+\cos(3kx))/1.2+\cos(kx))\exp(-\f{v^2}{2}),
\eeq
where $\alpha=0.01$ and $k=0.5$, has very similar results with and without limiters. We 
omit them to save some space. 

\begin{figure}
\centering
\includegraphics[width=3in,clip]{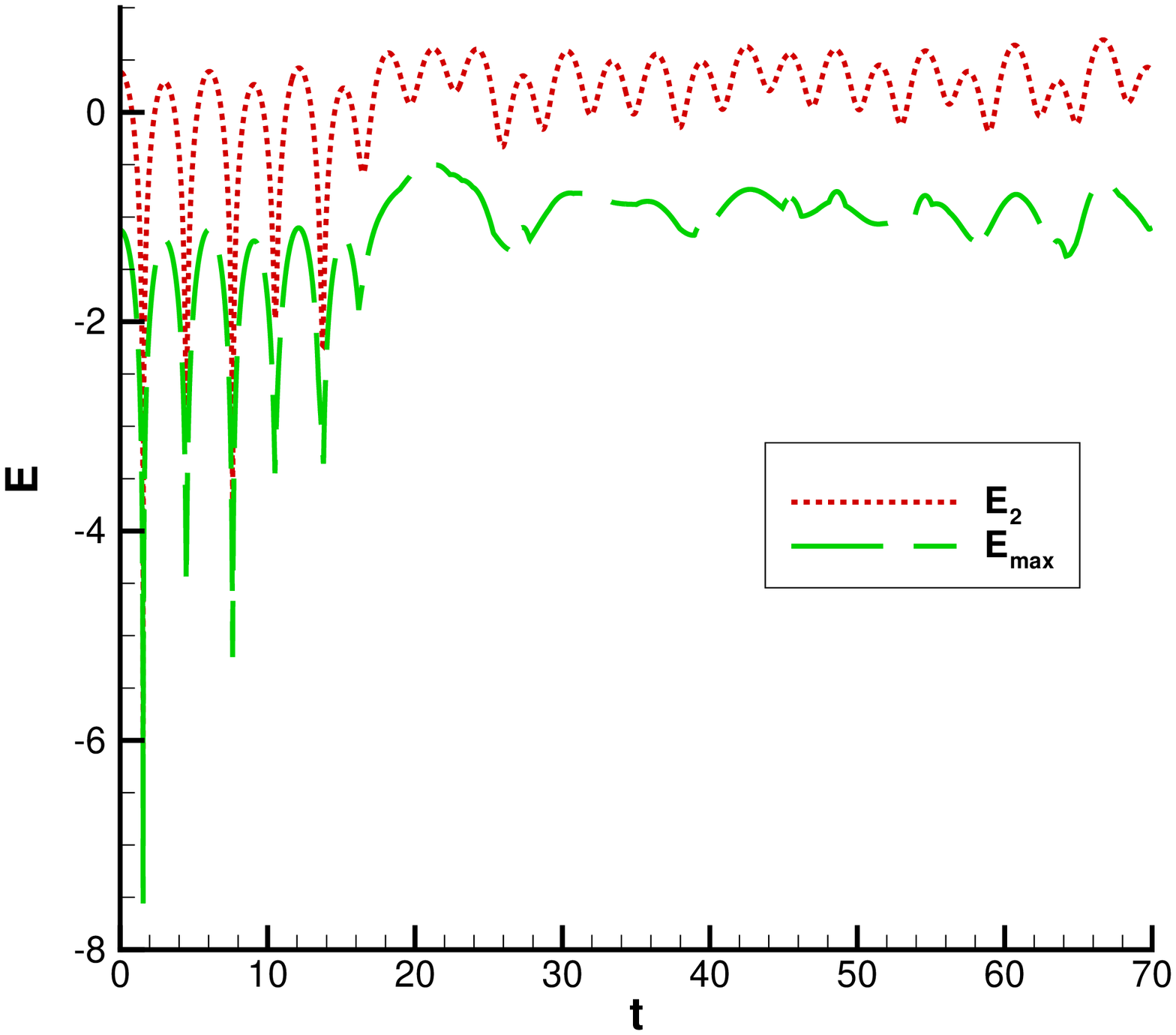},
\includegraphics[width=3in,clip]{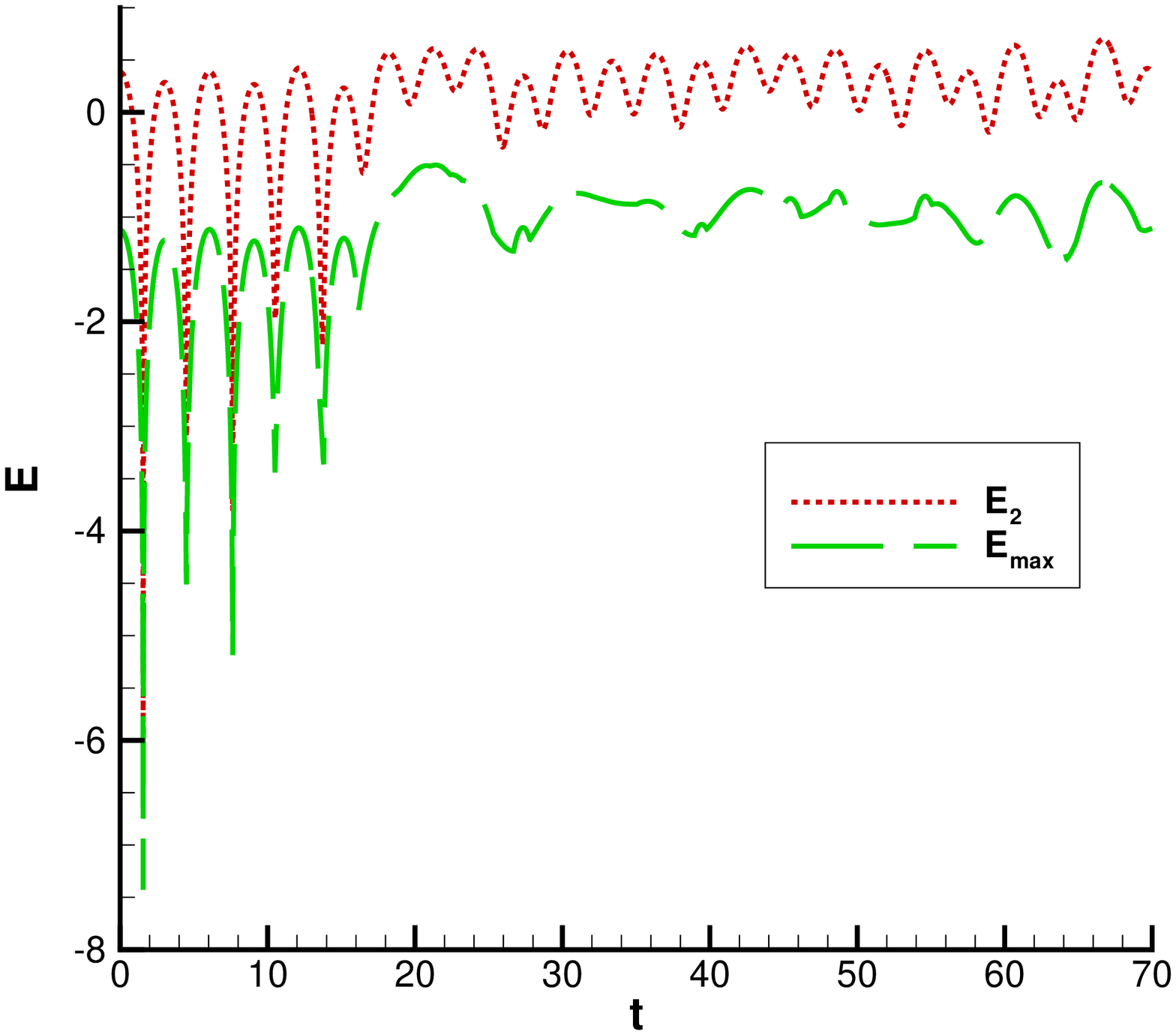}\\
\caption{Symmetric two stream instability with the initial condition (\ref{s2stream}). 
Time evolution of the electric field in $L^2$ norm ($E_2$) and $L^\infty$ norm ($E_{max}$) (logarithmic value).
$N_x \times N_v = 80 \times 160 $. Left: SL-WL; Right: RK-WL.}
\label{fig41}
\end{figure}

\begin{figure}
\centering
\includegraphics[width=3in,clip]{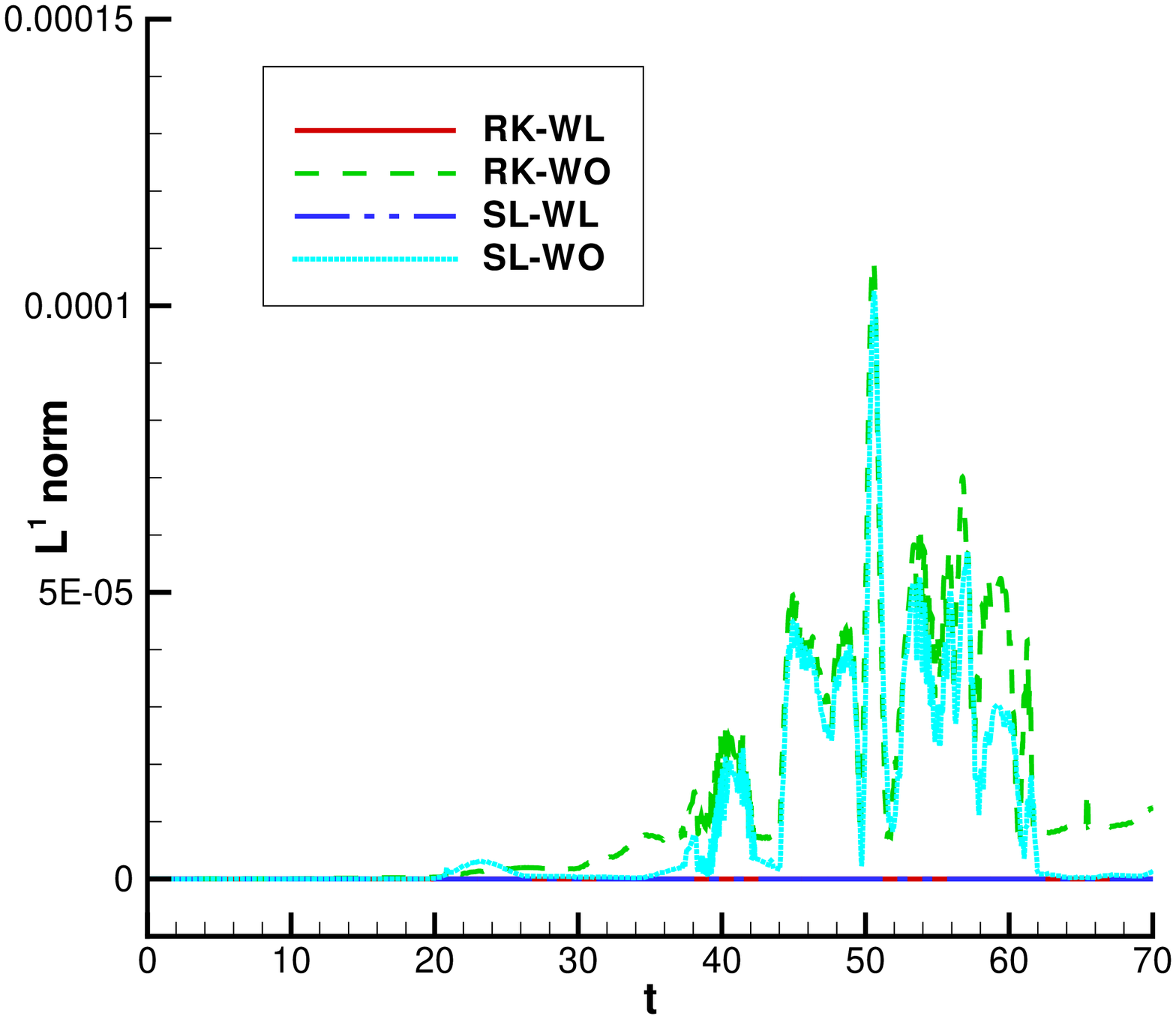},
\includegraphics[width=3in,clip]{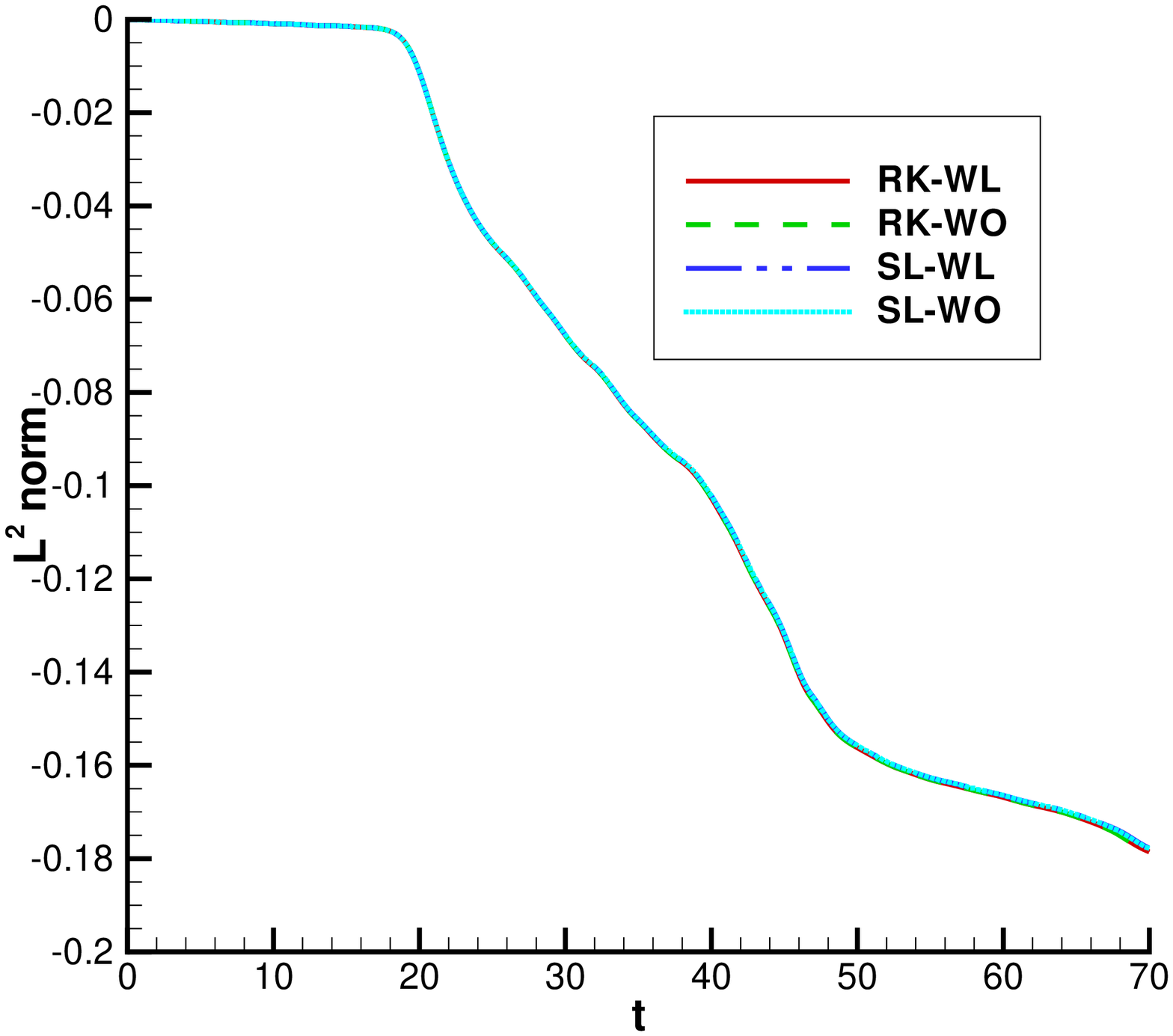}\\
\includegraphics[width=3in,clip]{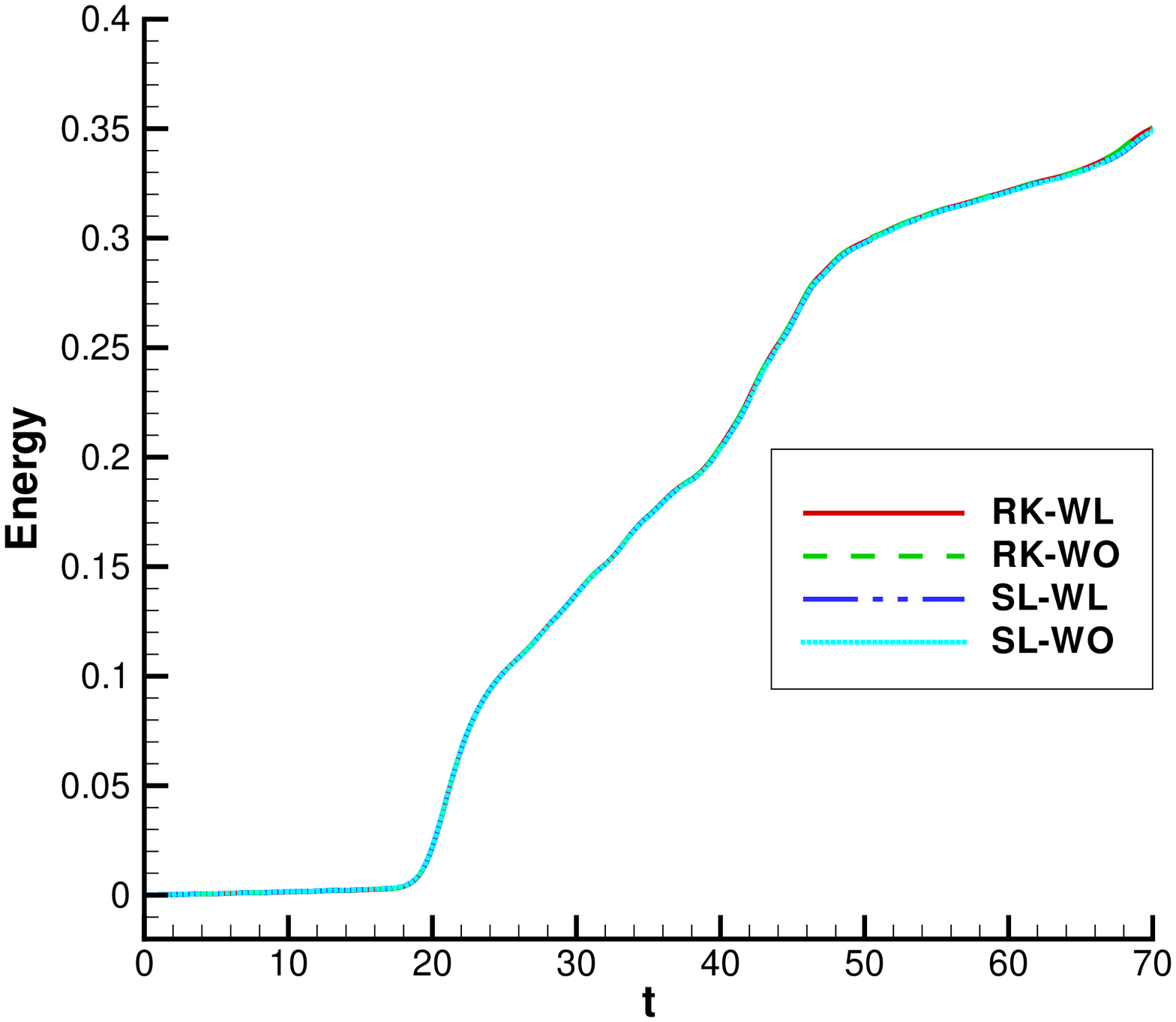},
\includegraphics[width=3in,clip]{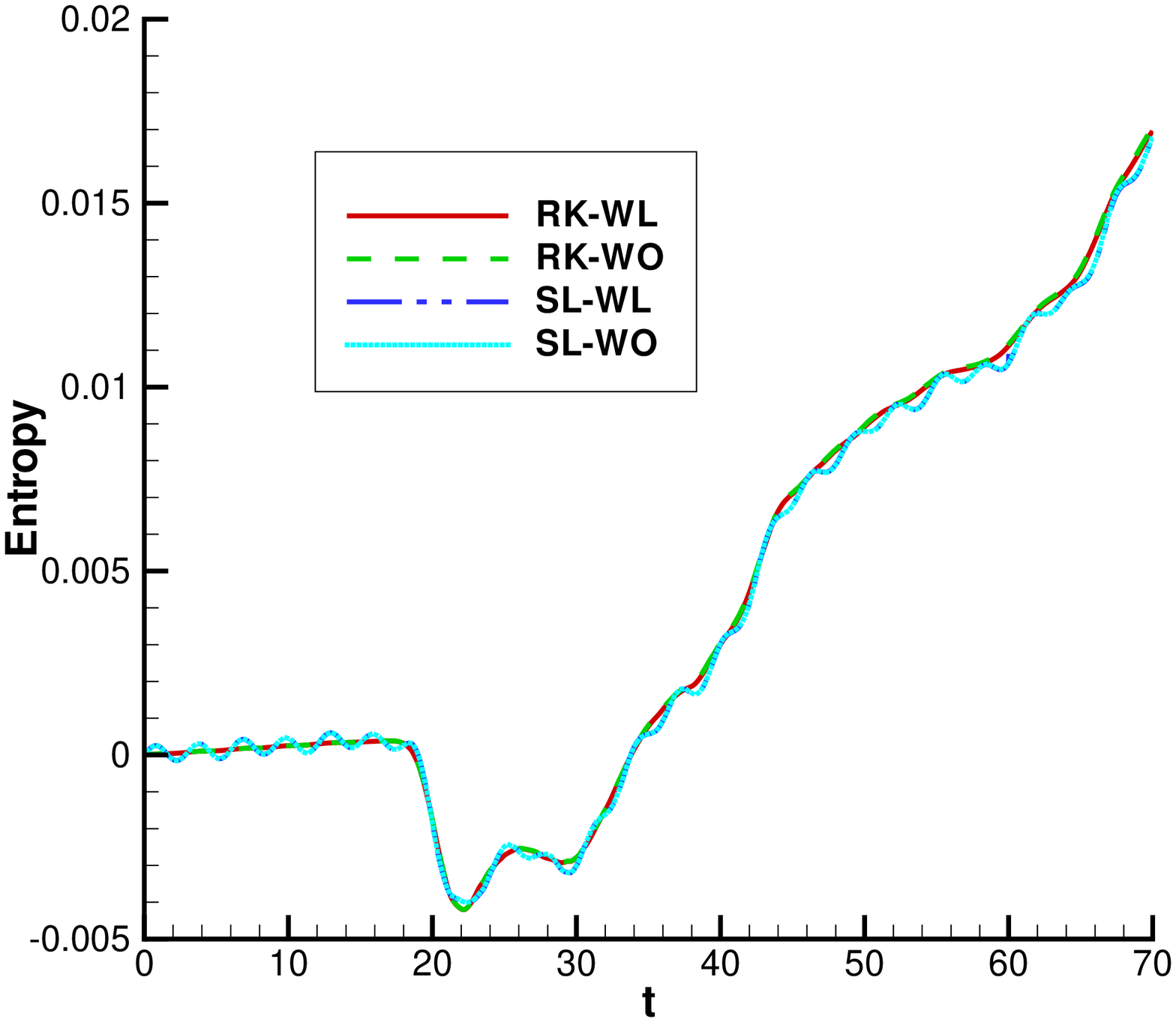}\\
\caption{Symmetric two stream instability with the initial condition (\ref{s2stream}). 
Time evolution of the relative
deviations of discrete $L^1$ norm (upper left), $L^2$ norm (upper right), kinetic
energy (lower left) and entropy (lower right). $N_x \times N_v = 80 \times 160 $. }
\label{fig42}
\end{figure}

\end{exa}

\begin{exa}(Bump-on-tail instability)
We consider an unstable bump-on-tail problem with the initial distribution as
\begin{align}
\label{bump}
f(0,x,v)=f_{b.o.t}(v)(1+\alpha\cos(kx)),
\end{align} 
where the bump-on-tail distribution is
\begin{align}
f_{b.o.t}(v)=\frac{n_p}{\sqrt{2\pi}}\exp(-\frac{v^2}{2})+\frac{n_b}{\sqrt{2\pi}}\exp(-\frac{(v-v_b)^2}{2v_t^2}).
\end{align}
The parameters are chosen to be $n_p=0.9$, $n_b=0.2$, $v_b=4.5$, $v_t=0.5$, $\alpha=0.04$, $k=0.3$.
The computational domain is $[0,\frac{2\pi}{k}]\times[-8, 8]$. We first show the time evolution of the $L^\infty$ norm for the electric field $E$ in Fig. \ref{ie11}. In Fig. \ref{ie12}, we display the relative deviations of discrete $L^1$ norm, $L^2$ norm, kinetic energy and
entropy for the distribution function $f$ with and without limiters. Clearly we can observe that without limiters the $L^1$ norm is around $10^{-5}$, which indicates negative numerical values of $f$, while with limiters the $L^1$ norm is around machine error. The contour of the phase space for $f$ with and without limiters are also presented in Fig. \ref{ie13}. The results match those in \cite{arber2002critical} and slight difference can be seen between with and without limiters.

\begin{figure}
\centering
\includegraphics[width=2in,angle=270,clip]{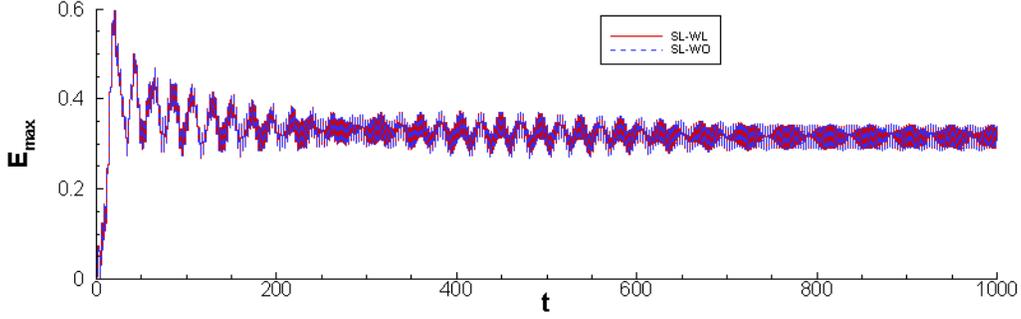}
\caption{Bump-on-tail problem with initial condition \eqref{bump}. 
Time evolution of the electric field 
in the $L^\infty$ norm $E_{max}$.
Mesh: $N_x\times N_v = 256 \times 256$. }
\label{ie11}
\end{figure}

\begin{figure}
\centering
\includegraphics[width=2.8in,angle=270,clip]{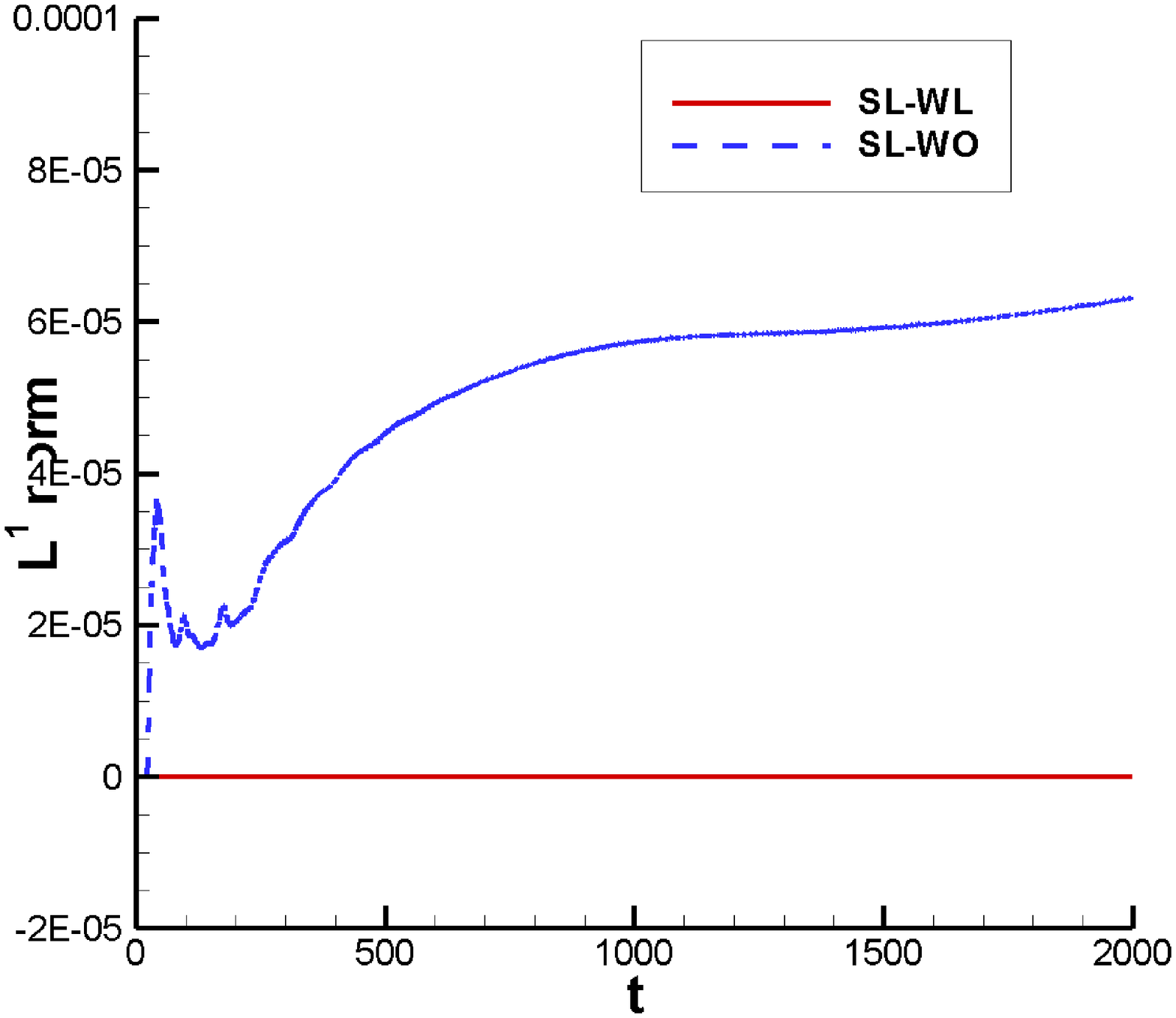}
\includegraphics[width=2.8in,angle=270,clip]{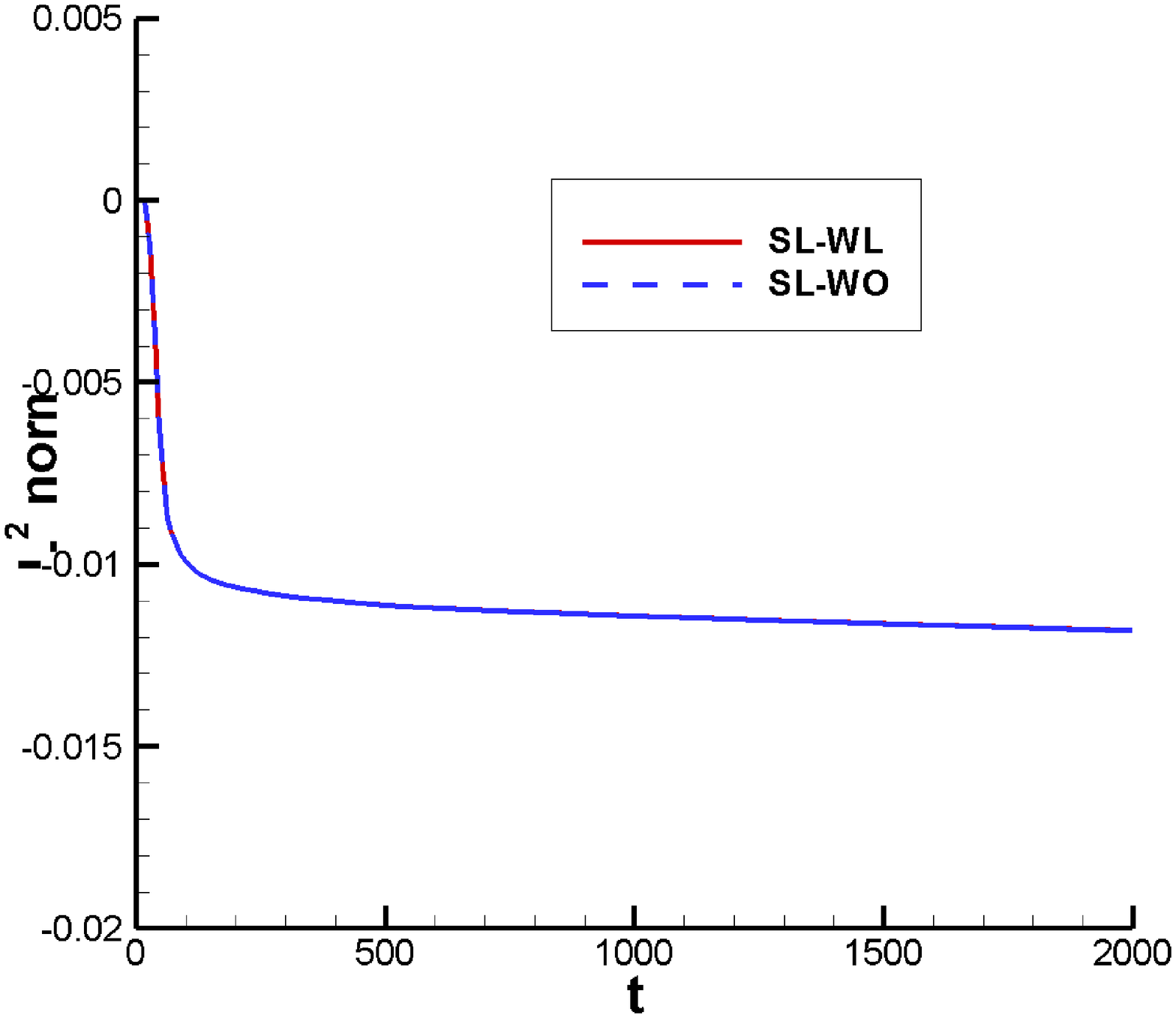}\\
\includegraphics[width=2.8in,angle=270,clip]{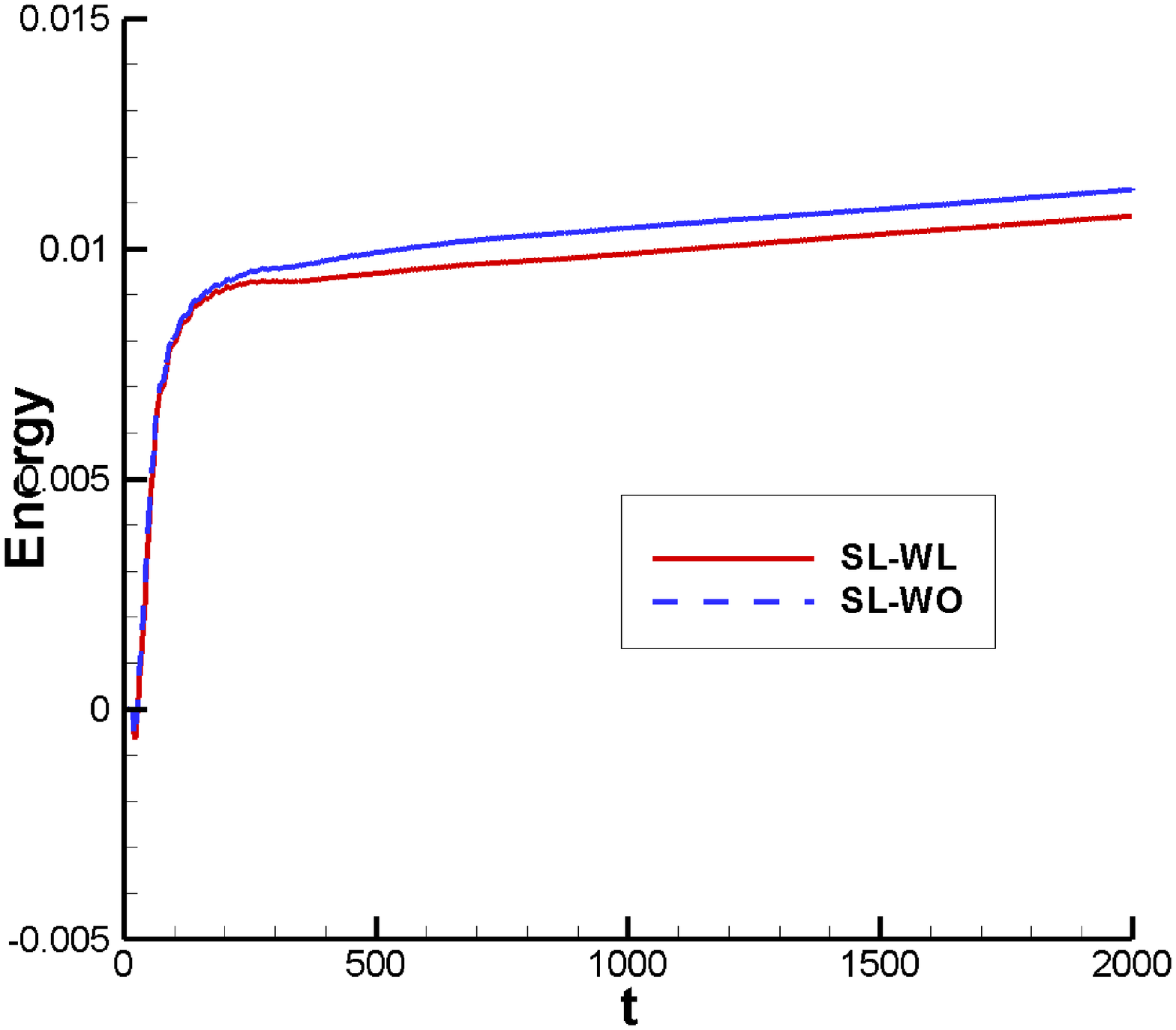}
\includegraphics[width=2.8in,angle=270,clip]{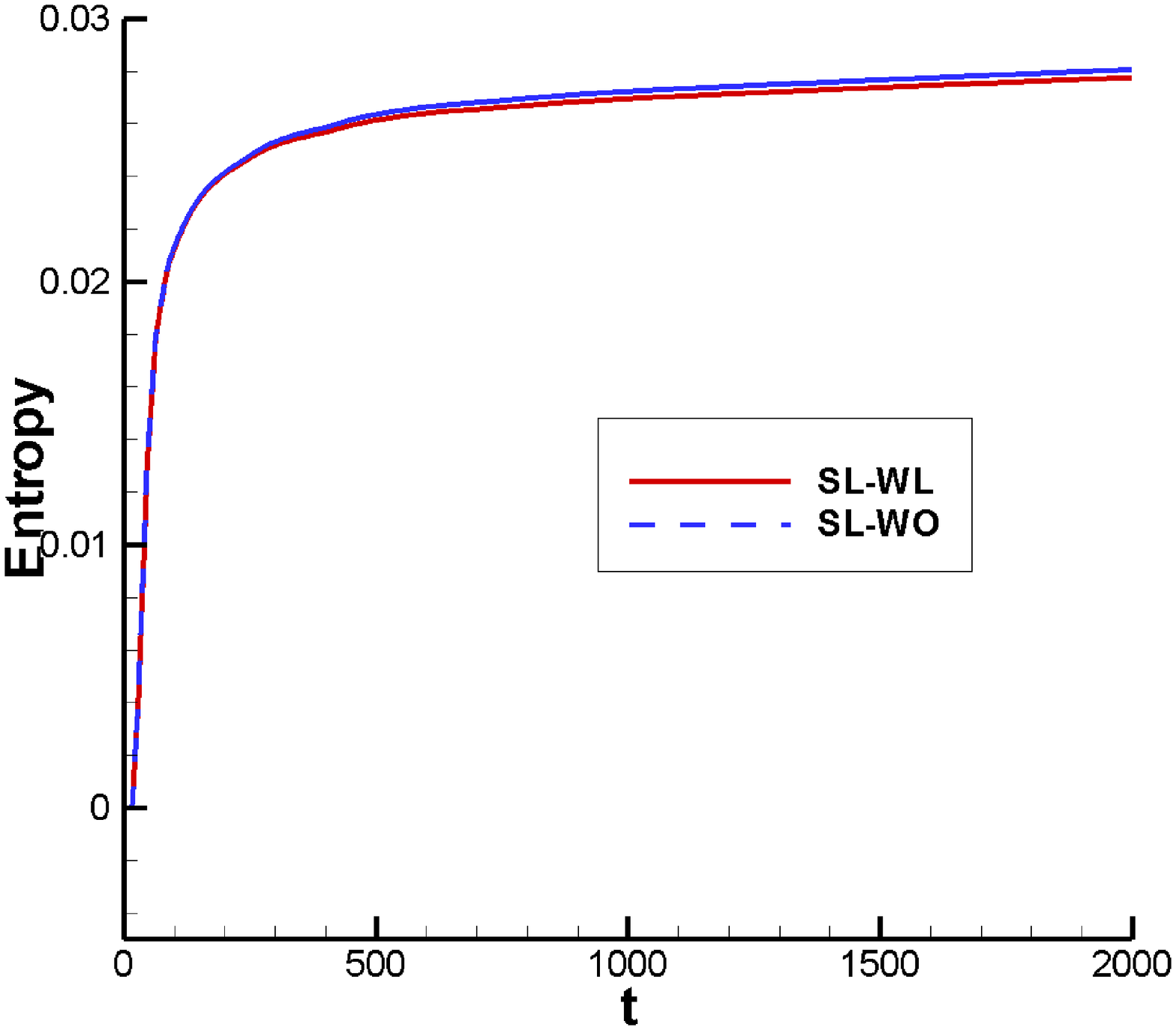}
\caption{Bump-on-tail problem 
with initial condition \eqref{bump}. Time evolution of the
relative deviations of discrete $L^1$ norm, $L^2$ norm, kinetic energy and
entropy for the distribution functions $f$.
Mesh: $N_x\times N_v = 256 \times 256$.}
\label{ie12}
\end{figure}

\begin{figure}
\centering
\includegraphics[width=2.8in,angle=270,clip]{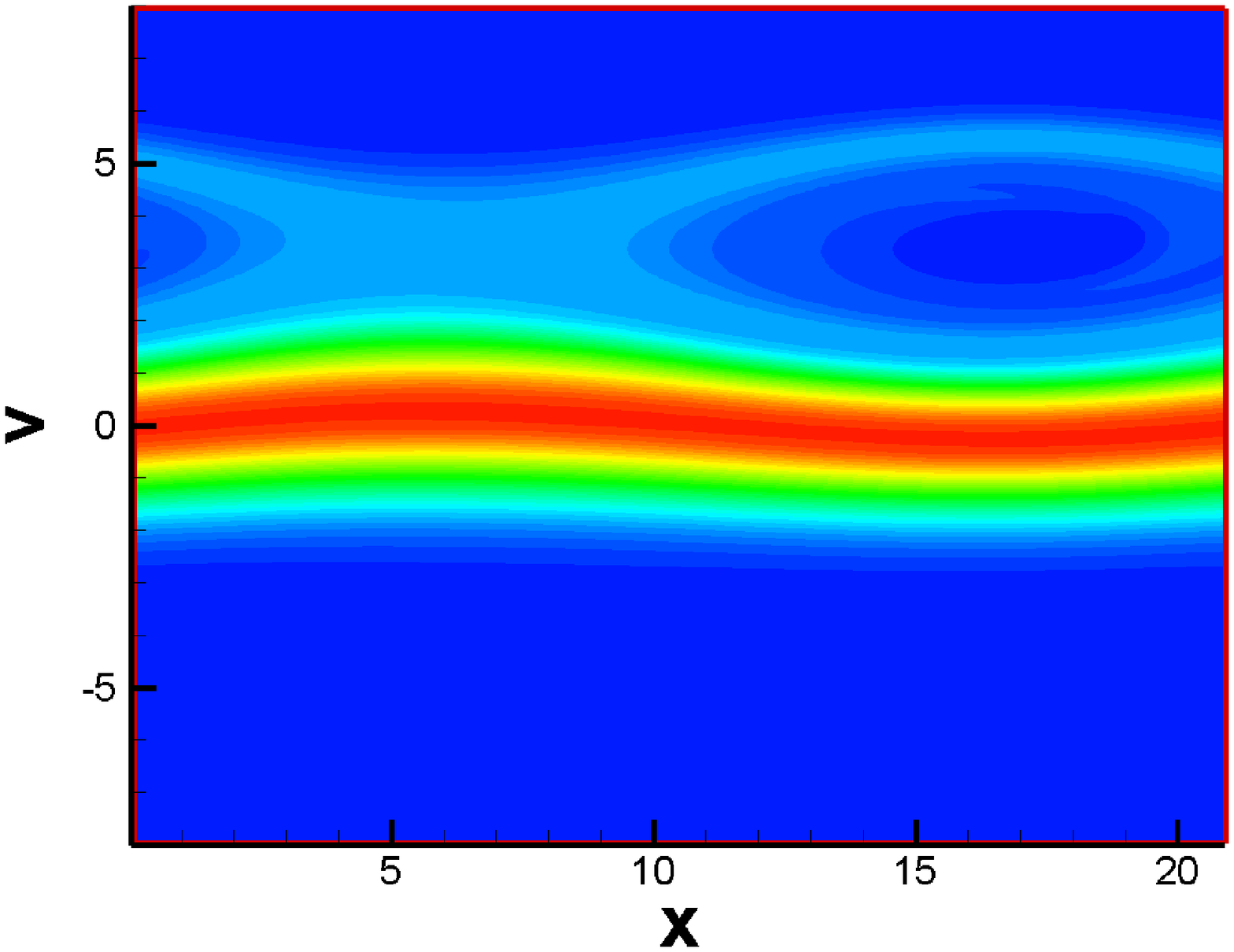}
\includegraphics[width=2.8in,angle=270,clip]{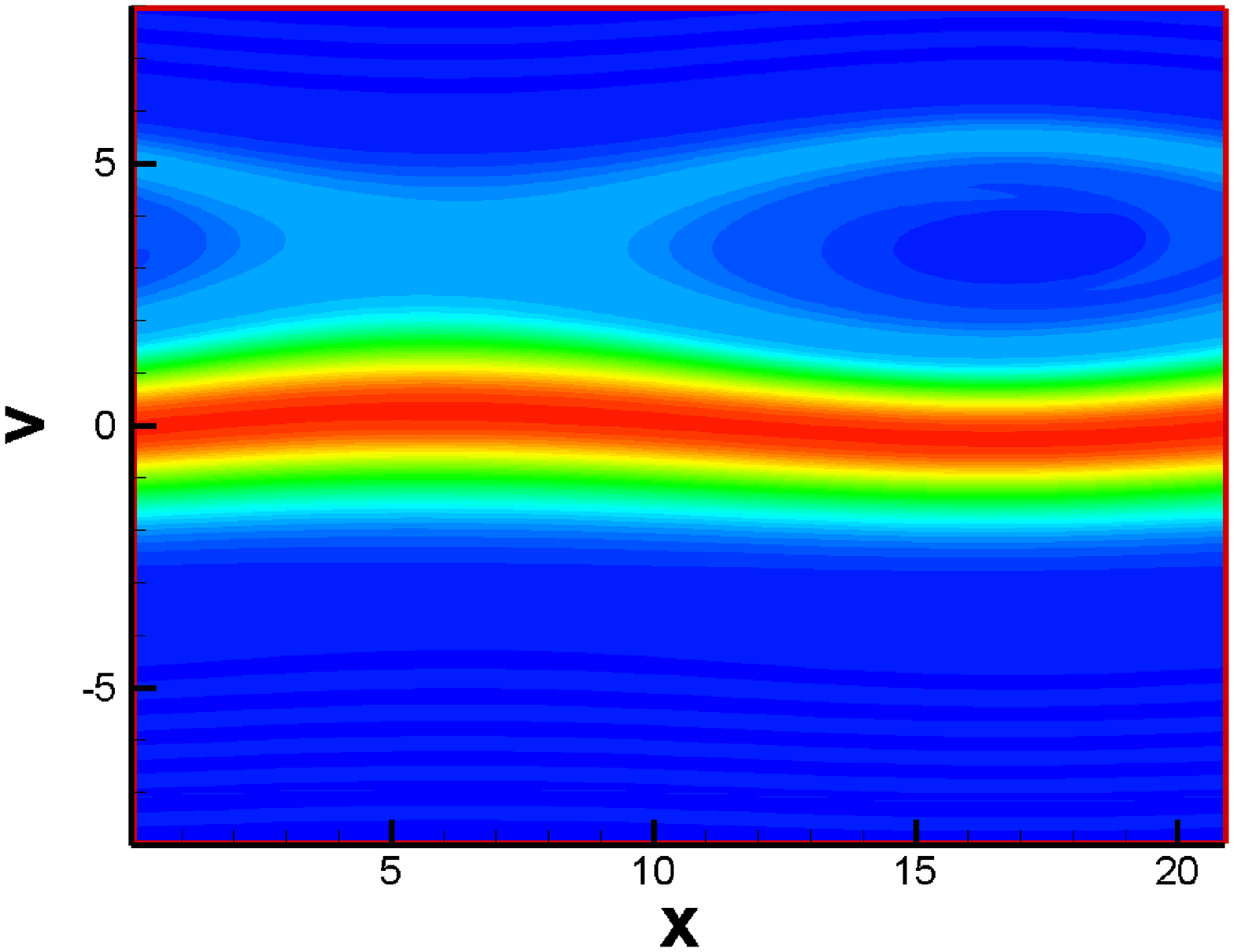}
\caption{Bump-on-tail problem 
with initial condition \eqref{bump}. Contour of phase space for $f$ at $t=500$.
37 equally space contour lines within the range $[0, 0.36]$.
Mesh: $N_x\times N_v = 256 \times 256$. Left: with limiters; Right: without limiters.}
\label{ie13}
\end{figure}
\end{exa}

\begin{exa} (KEEN Wave)
We consider the KEEN waves (kinetic electrostatic electron nonlinear waves)
with the Vlasov equation
\begin{equation}
\partial_t f + v\cdot \nabla_x f + (E(t,x)-E_{ext})\cdot \nabla_v f=0,
\label{keen}
\end{equation}
where $E_{ext}=A_d(t)\sin(kx-\omega t)$ is the external field with $\omega=0.37$. 
$A_d(t)$ is a temporal envelope that is ramped up to a plateau and then ramped down to zero.
Two external fields 
\begin{eqnarray}
A_d^J(t)=
\begin{cases}
A_m \sin(t\pi/100) \quad & 0<t<50, \\
A_m                \quad & 50 \le t<150, \\
A_m \cos((t-150)\pi/100) \quad & 150 \le t<200, \\
0                  \quad & 200 \le t<T,
\end{cases}
\label{adj}
\end{eqnarray}
with $A_m=0.052$ \cite{johnston2009persistent}, and
\begin{eqnarray}
A_d^A(t)=
\begin{cases}
A_m\frac{1}{1+e^{-40(t-10)}} \quad & 0<t<60, \\
A_m\big(1-\frac{1}{1+e^{-40(t-110)}}) \quad & 60\le t <T,
\end{cases}
\label{ada}
\end{eqnarray}
with $A_m=0.4$ \cite{afeyan2012kinetic} are considered. 
The system is initialized to be a Maxwellian $f(0,x,v)=1/\sqrt{2\pi}e^{-v^2/2}$. 
For both cases, the computational domain is taken to be $[0, 2\pi/k]\times[-8,8]$,
with $k=0.26$. We take the same mesh $N_x\times N_v=200\times 400$
as in \cite{cheng2012study}, and consider the $n$th Log Fourier mode
for the electric field $E(t,x)$ to be
\begin{equation}
log FM_n(t)=log_{10}\left(\f1L\sqrt{\left|\int_0^L E(t,x) \sin(k n x)dx\right|^2
+\left|\int_0^L E(t,x) \cos(k n x)dx\right|^2}\right).
\end{equation}

We consider the SL method with and without limiter for this example.
The first four Fourier modes are plotted in Figure \ref{figkw1} for the drive
$A_d^J$. The time evolution of the relative deviations of discrete $L^1$ norm 
is around machine error for the solution even without limiter, little difference
can be seen if with limiter, we omit the figures here to save space.
The phase space contours for the drive $A_d^A$ at $t=15s, 60s, 120s, 300s$ 
are plotted in Figure \ref{figkw21}. We show the time evolution of the relative deviations of discrete $L^1$ norm and the first four Fourier modes for the drive $A_d^A$ in Figure \ref{figkw22}. 
For this case, the $L^1$ norm has been significantly improved if with limiter.

\begin{figure}
\centering
\includegraphics[width=3in,clip]{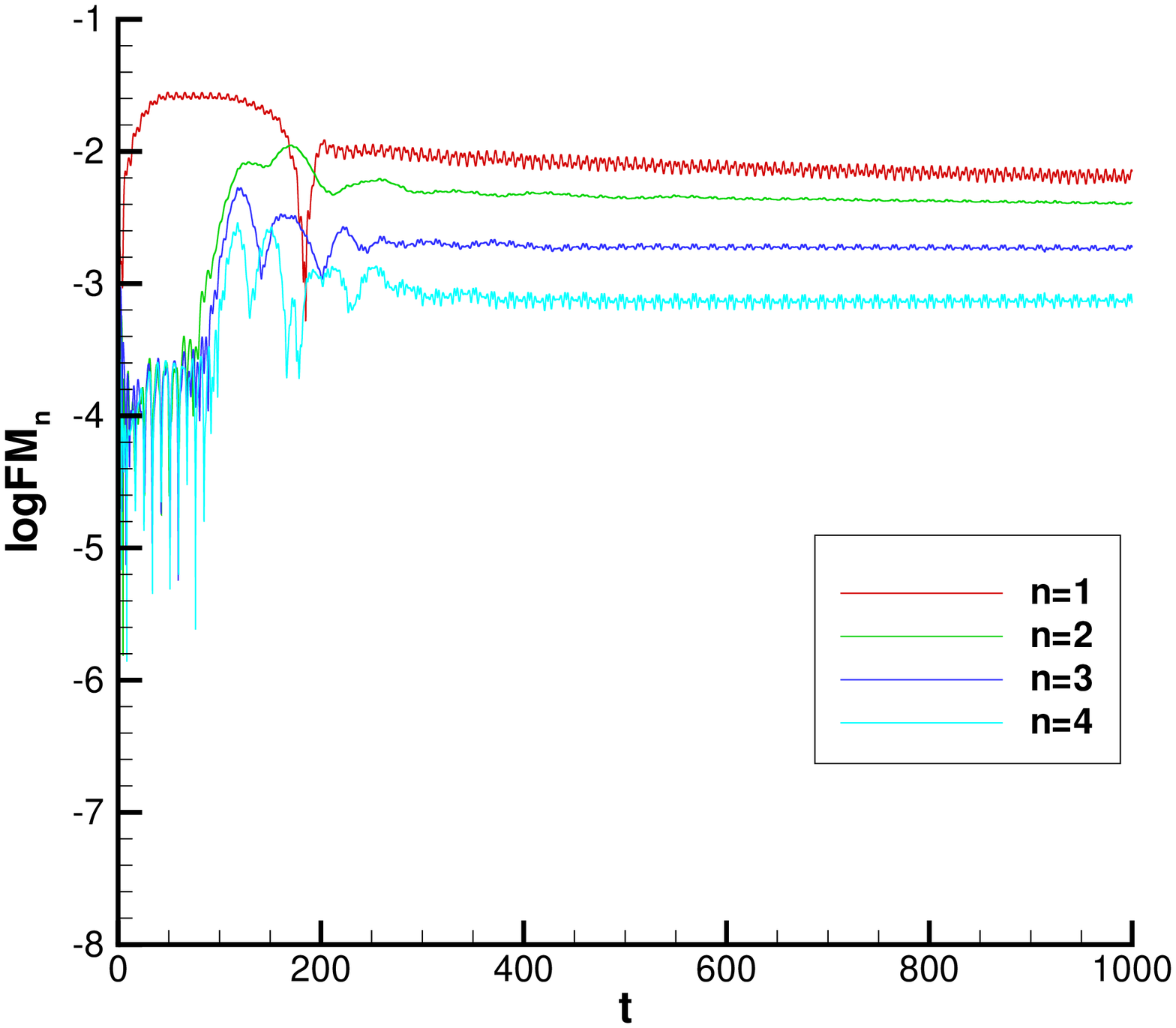},
\includegraphics[width=3in,clip]{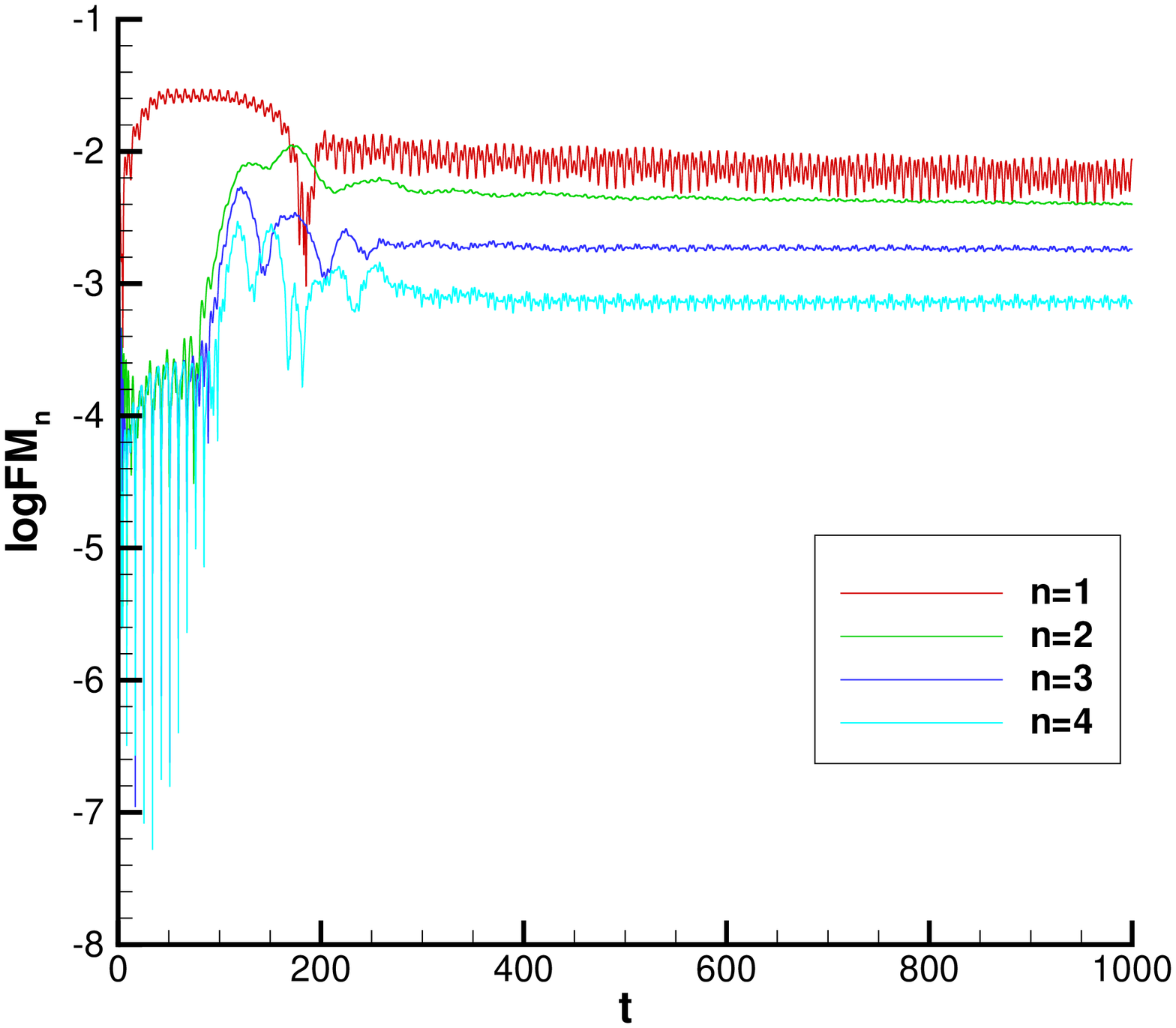}\\
\caption{KEEN wave (\ref{keen}) with the drive $A_d^J$ (\ref{adj}). 
The first four Log Fourier Modes, n=1, 2, 3, 4 from top to bottom.
Mesh: $N_x\times N_v = 200 \times 400$. Left: with limiter; Right: without limiter.}
\label{figkw1}
\end{figure}

\begin{figure}
\centering
\subfigure[t=15s]{\includegraphics[width=2.8in,angle=270,clip]{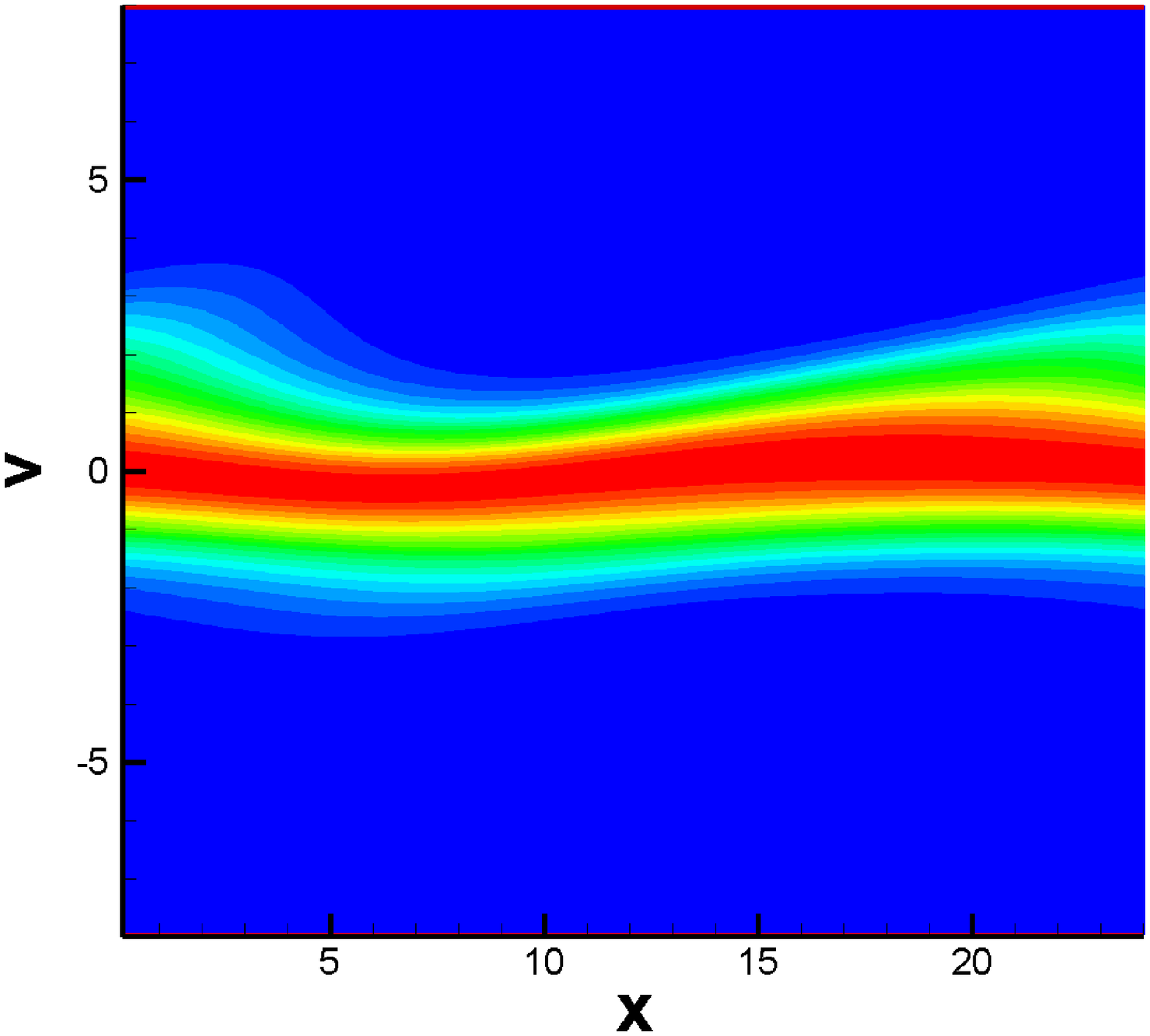}}
\subfigure[t=60s]{\includegraphics[width=2.8in,angle=270,clip]{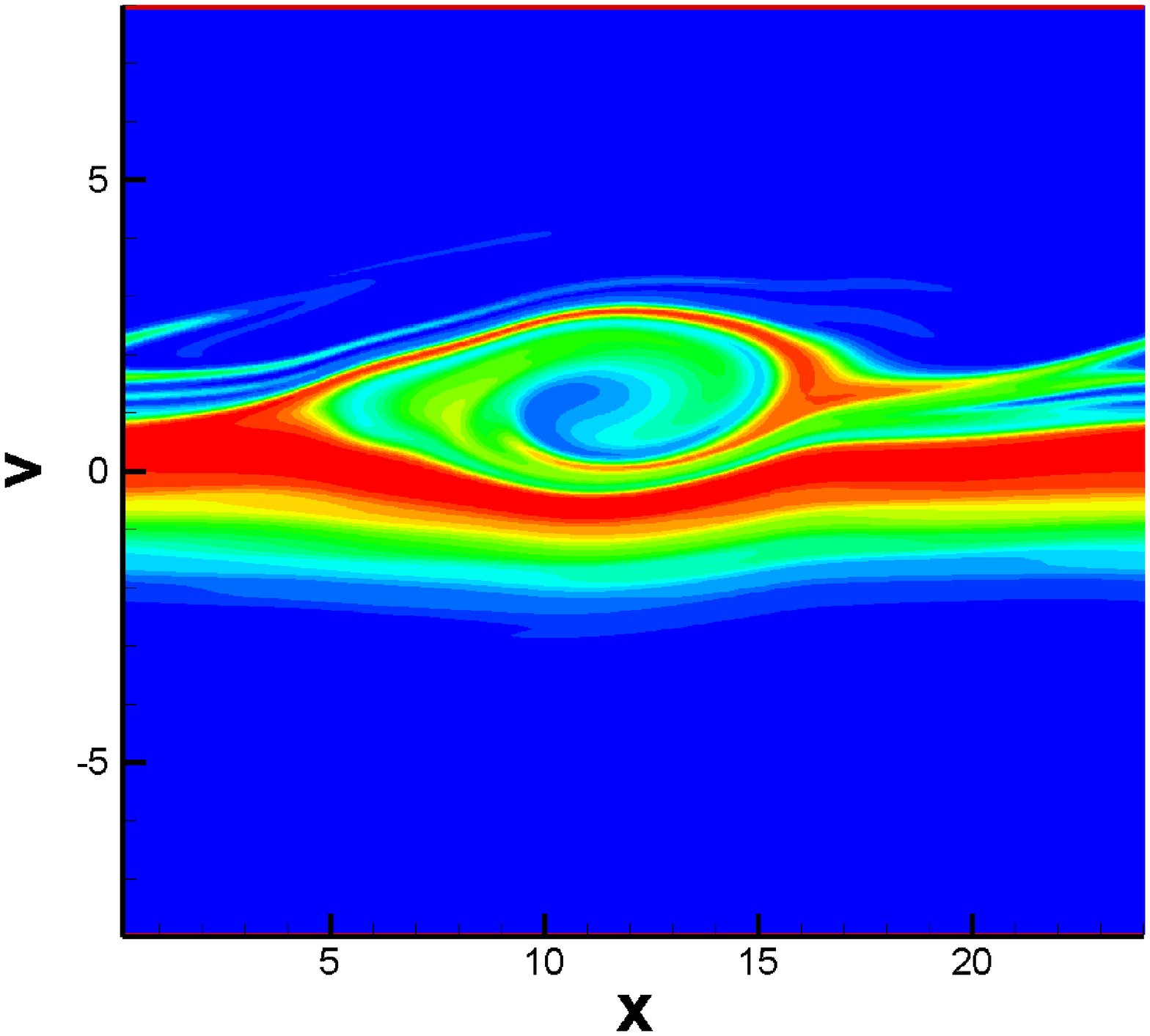}}
\subfigure[t=120s]{\includegraphics[width=2.8in,angle=270,clip]{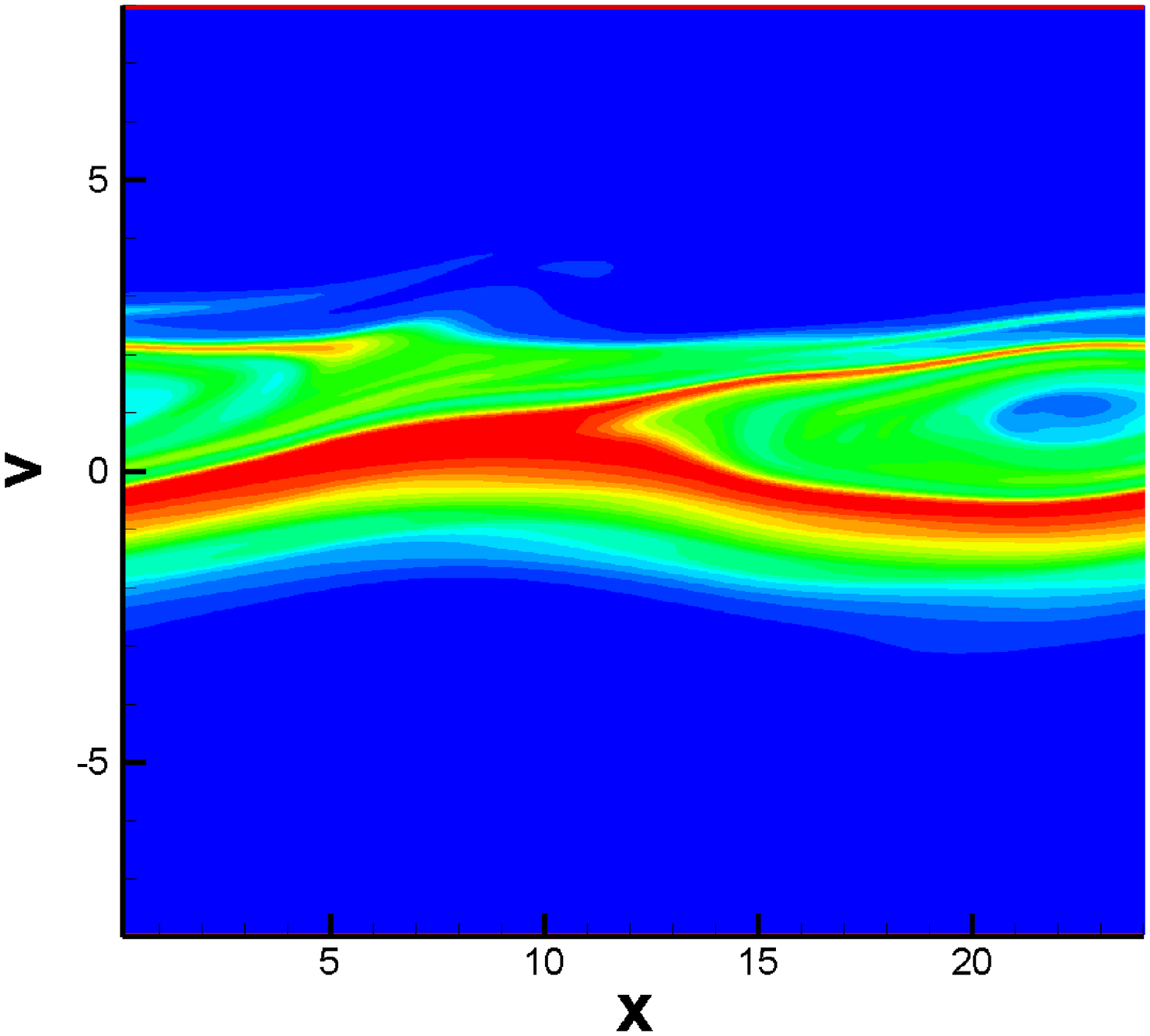}}
\subfigure[t=300s]{\includegraphics[width=2.8in,angle=270,clip]{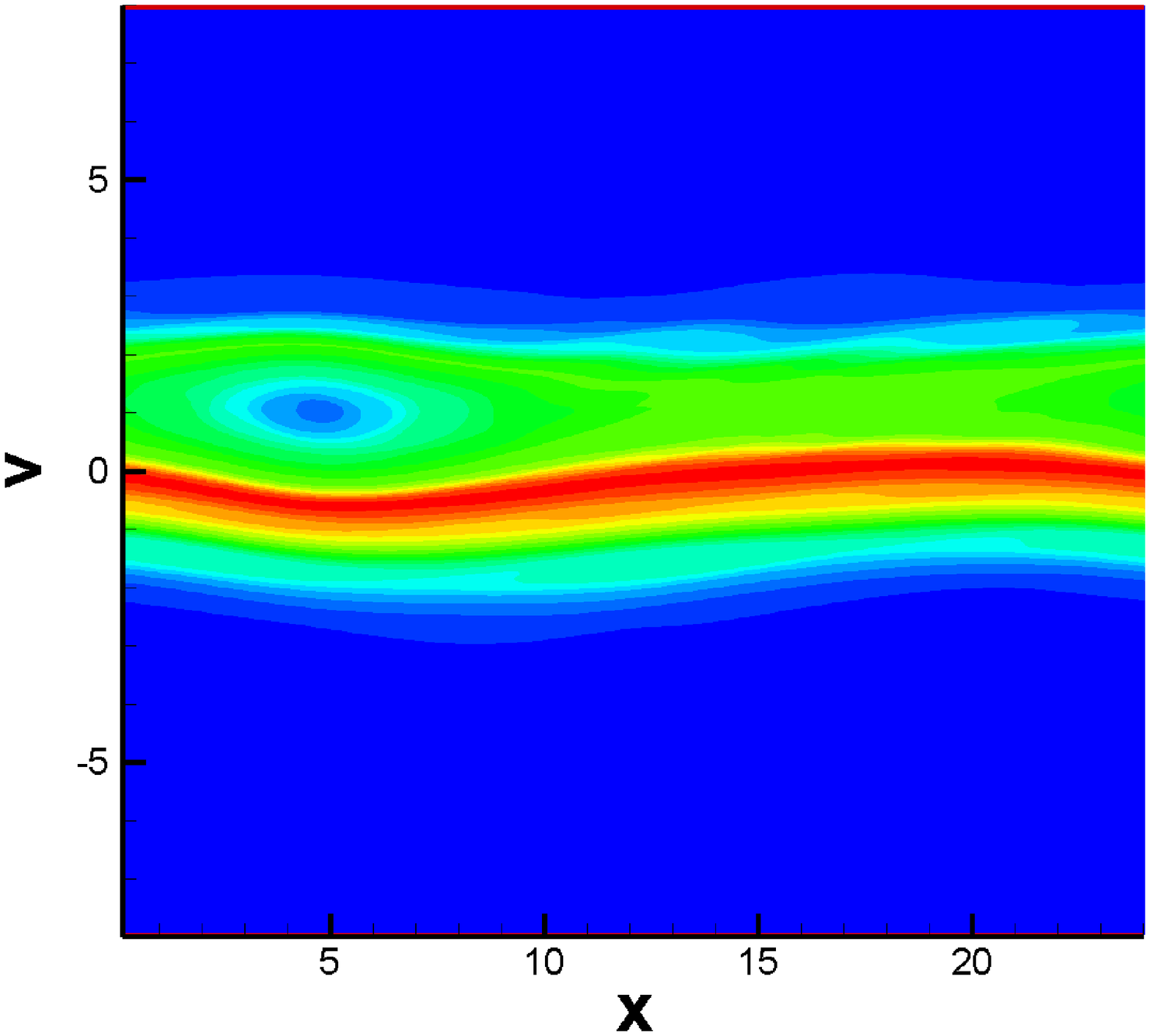}}
\caption{Phase space contour for the KEEN wave (\ref{keen}) with the drive $A_d^A$ (\ref{ada}).  
Mesh: $N_x\times N_v = 200 \times 400$. }
\label{figkw21}
\end{figure}

\begin{figure}
\centering
\includegraphics[width=3in,clip]{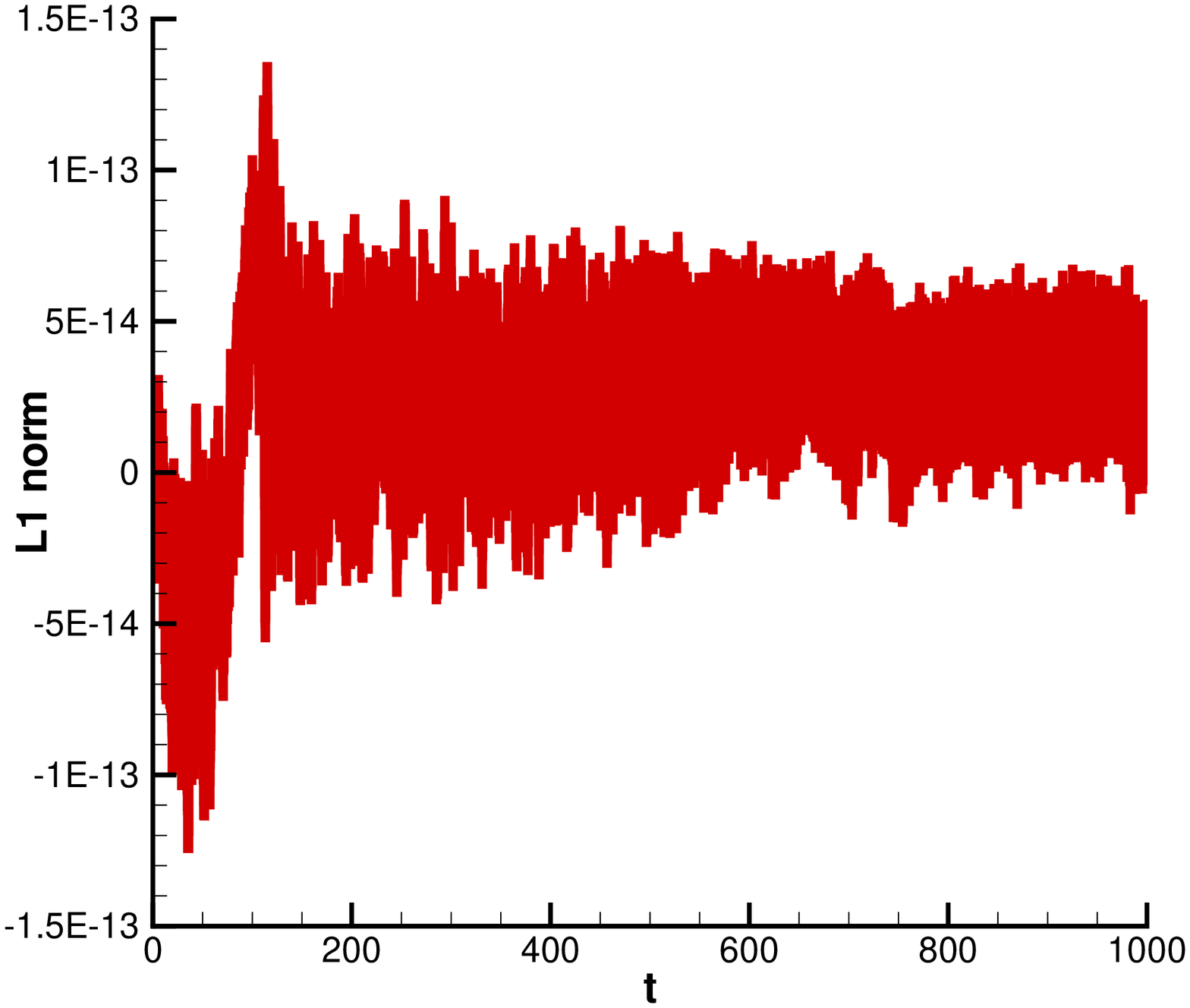},
\includegraphics[width=3in,clip]{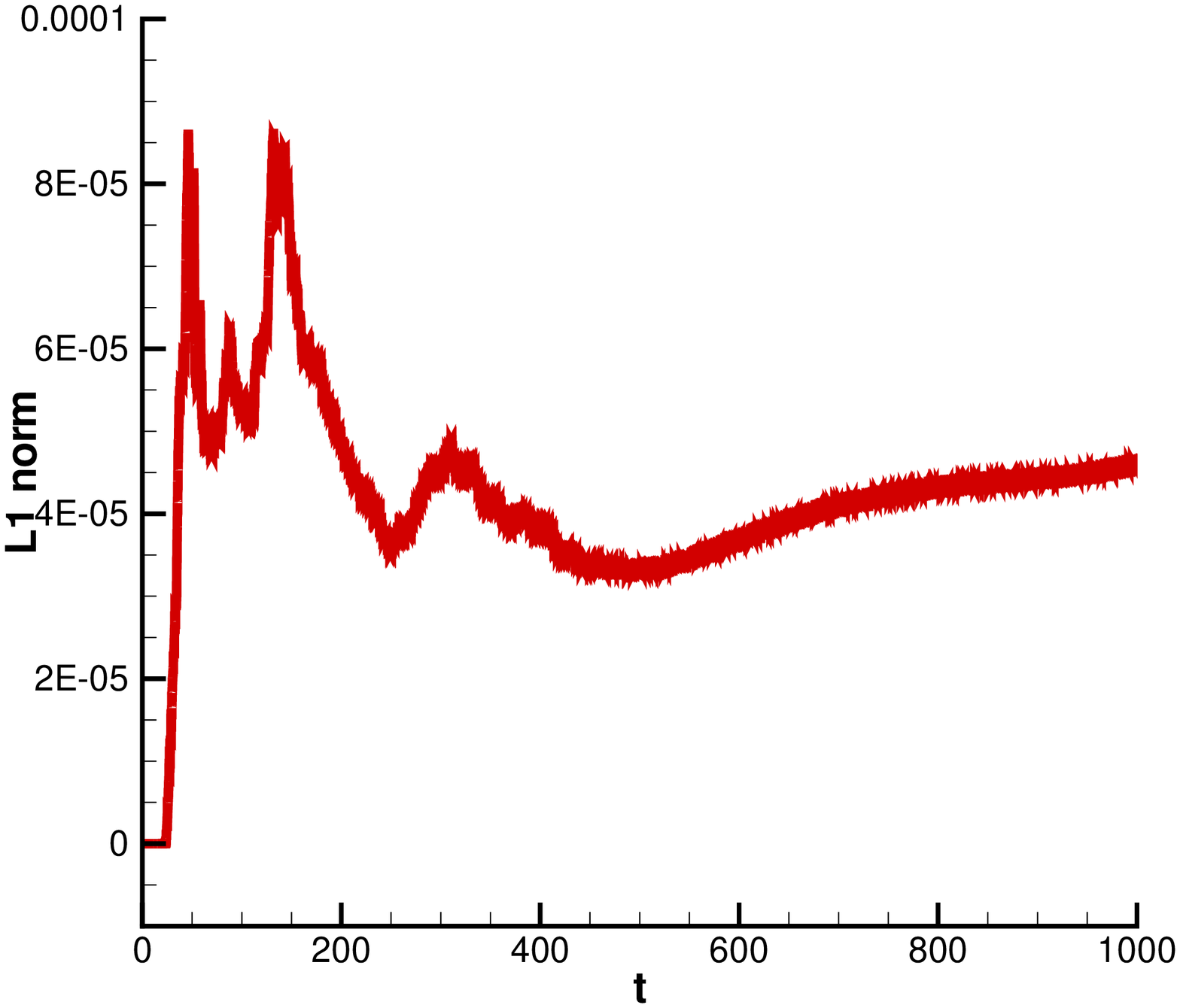}\\
\includegraphics[width=3in,clip]{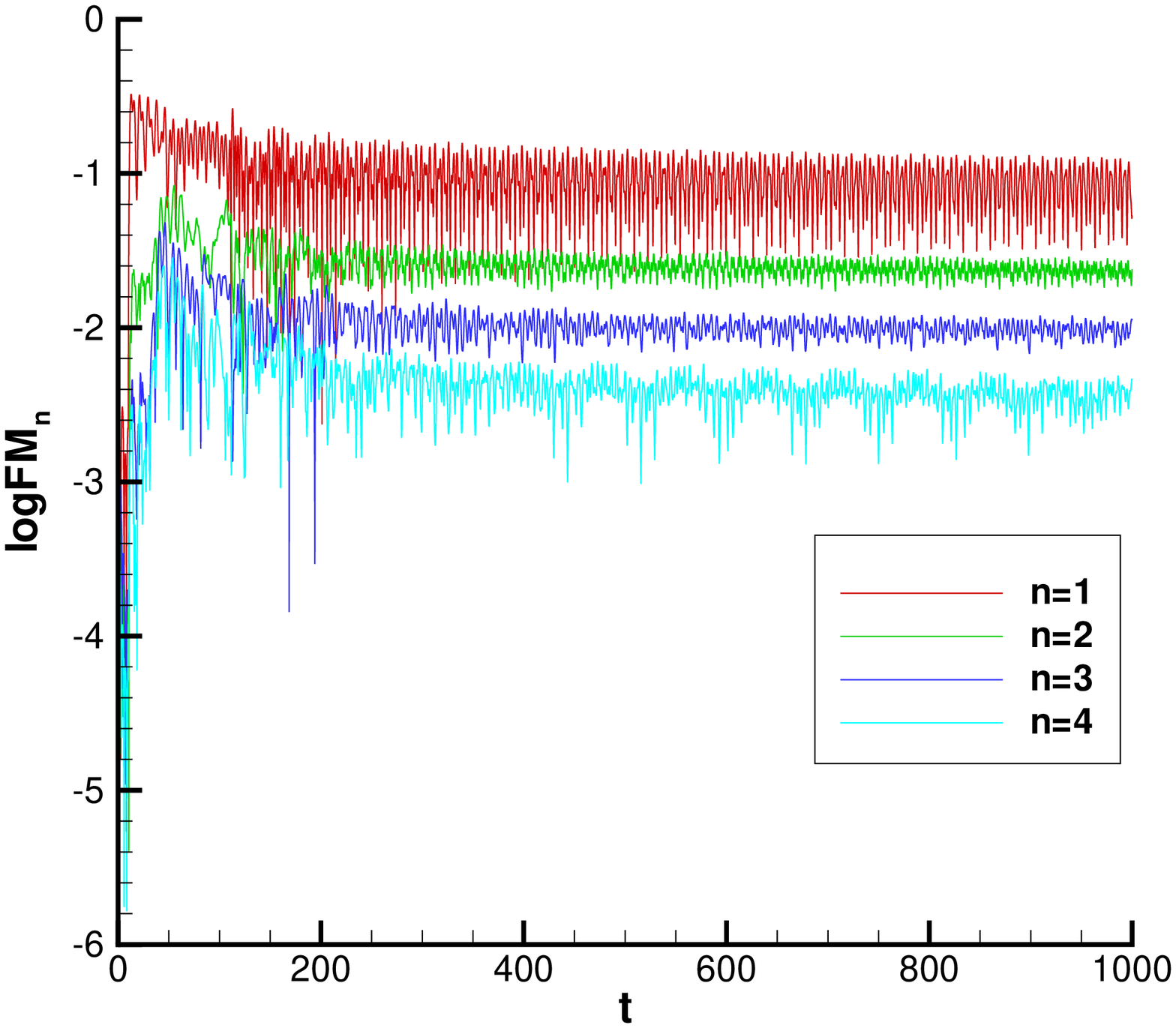},
\includegraphics[width=3in,clip]{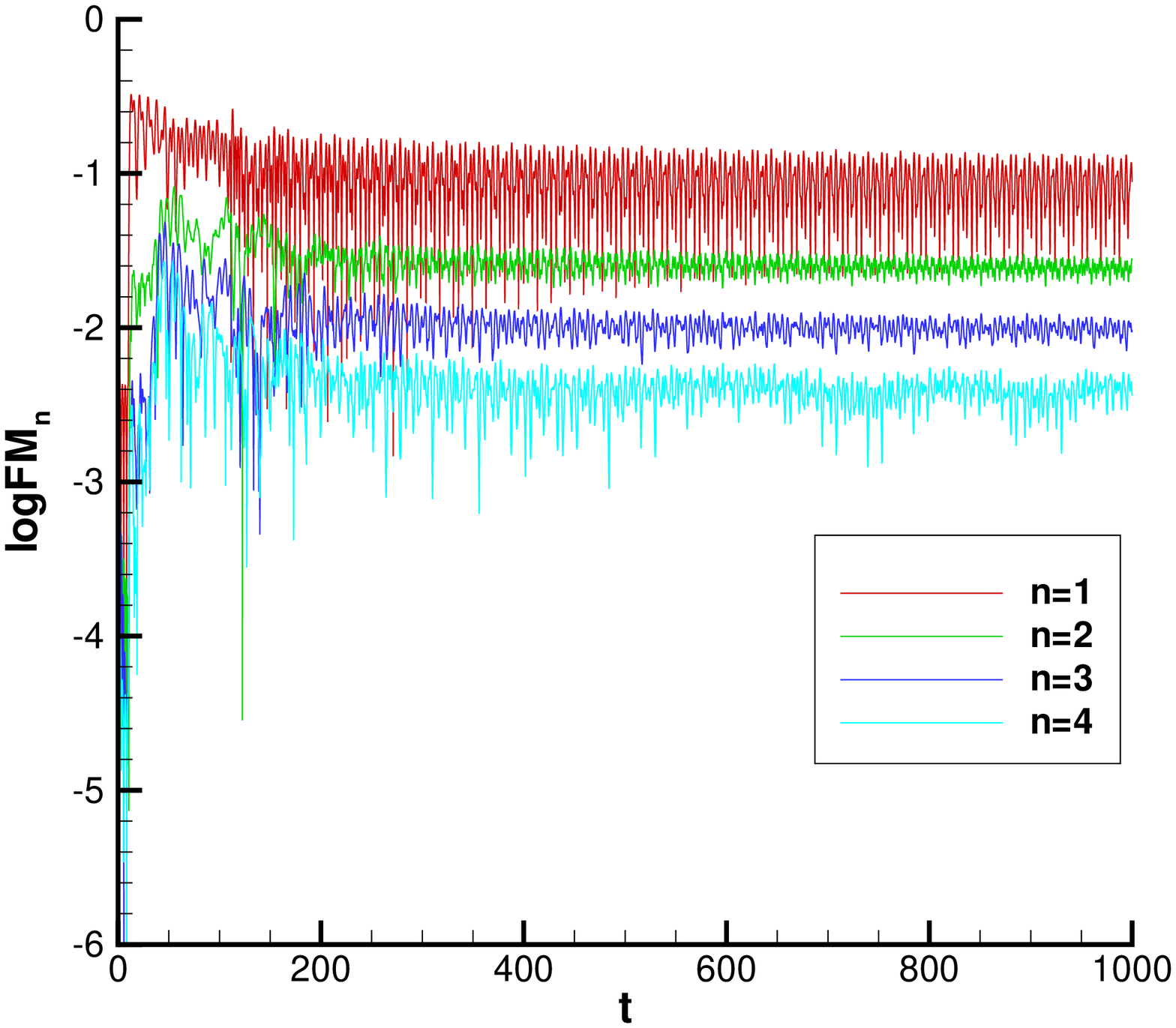}\\
\caption{KEEN wave (\ref{keen}) with the drive $A_d^A$ (\ref{ada}). 
Time evolution of the relative deviations of discrete $L^1$ norm 
(Top) and the first four Log Fourier Modes (Bottom), n=1, 2, 3, 4 from top to bottom.
Mesh: $N_x\times N_v = 200 \times 400$. Left: with limiter; Right: without limiter.}
\label{figkw22}
\end{figure}

\end{exa}

\begin{exa}
Now we consider a 1D Vlasov-Poisson system with two species of electrons and ions \cite{arber2002critical}.
The electrons and ions have opposite charges of equal magnitude and mass ratio $m_i/m_e=M_r$. The
Vlasov equations for electrons and ions are
\begin{align}
\partial_t f_e + v \partial_x f_e - E(t,x) \partial_v f_e =0, \label{elec} \\
\partial_t f_i + v \partial_x f_i + \frac{E(t,x)}{M_r} \partial_v f_i =0, \label{ion}
\end{align}
where $E(t,x)$ is the electric field. The Poisson equation is taken to be
\begin{align}
E(t,x)=-\nabla_x \phi(t,x), \qquad -\Delta_x \phi(t,x) = \int(f_i-f_e)dv.
\end{align}

We study the ion-acoustic turbulence problem. The problem is the onset and saturation of the ion-acoustic instability. We take $M_r=1000$ and the initial ion distribution function is defined to be
\begin{equation}
\label{ion0}
f_i(0,x,v)=\Big(\frac{M_r}{2\pi}\Big)^{1/2}\exp\Big(-\frac{M_r}{2}v^2\Big).
\end{equation}
The electrons are a drifting Maxwellian and the initial distribution function is defined as
\begin{equation}
\label{elec0}
f_e(0,x,v)=(1+a(x))\frac{1}{\sqrt{2\pi}}\exp\Big(-\f12(v-U_e)^2\Big),
\end{equation}
where $U_e=-2$ and
\begin{align*}
a(x)=& 0.01\big( \sin(x)+\sin(0.5x)+\sin(0.1x)+\sin(0.15x)+\sin(0.2x) \\
     &+\cos(0.25x)+\cos(0.3x)+\cos(0.35x) \big).
\end{align*}
The computational domain is $[0, \frac{2\pi}{0.05}]\times[-8, 8]$. In Fig. \ref{ie21}, we
show the difference of the ion and electron fluid speeds $u_i-u_e$, which indicates
the momentum transfer from electrons to ions. We can see the decay rate agrees with
the results in \cite{arber2002critical}. We also show the evoluntion of $L^1$ norms
for the distribution functions $f_e$ and $f_i$ of both electrons and ions in Fig. \ref{ie22}. We can see
the MPP limiters can effectively control the $L^1$ norms to machine error. The phase space for
$f_e$ with and without limiters are displayed in Fig. \ref{ie23}, we can find that the result with
limiters has more clear fine structures than the one without limiters.

\begin{figure}
\centering
\includegraphics[width=3in,clip]{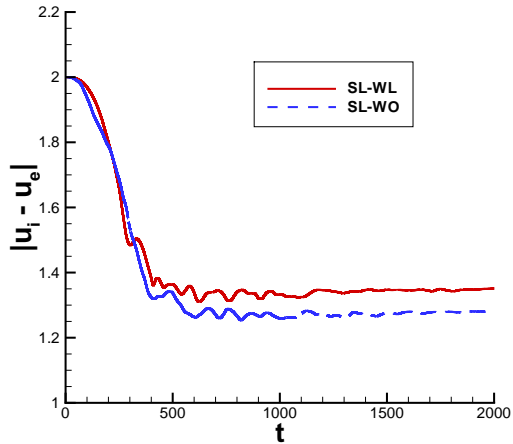}
\caption{Ion-acoustic turbulence problem \eqref{ion} and \eqref{elec}
with initial conditions \eqref{ion0} and \eqref{elec0}. Time evolution of $|u_i-u_e|$.
Mesh: $N_x\times N_v = 256 \times 256$. }
\label{ie21}
\end{figure}

\begin{figure}
\centering
\includegraphics[width=2.8in,angle=270,clip]{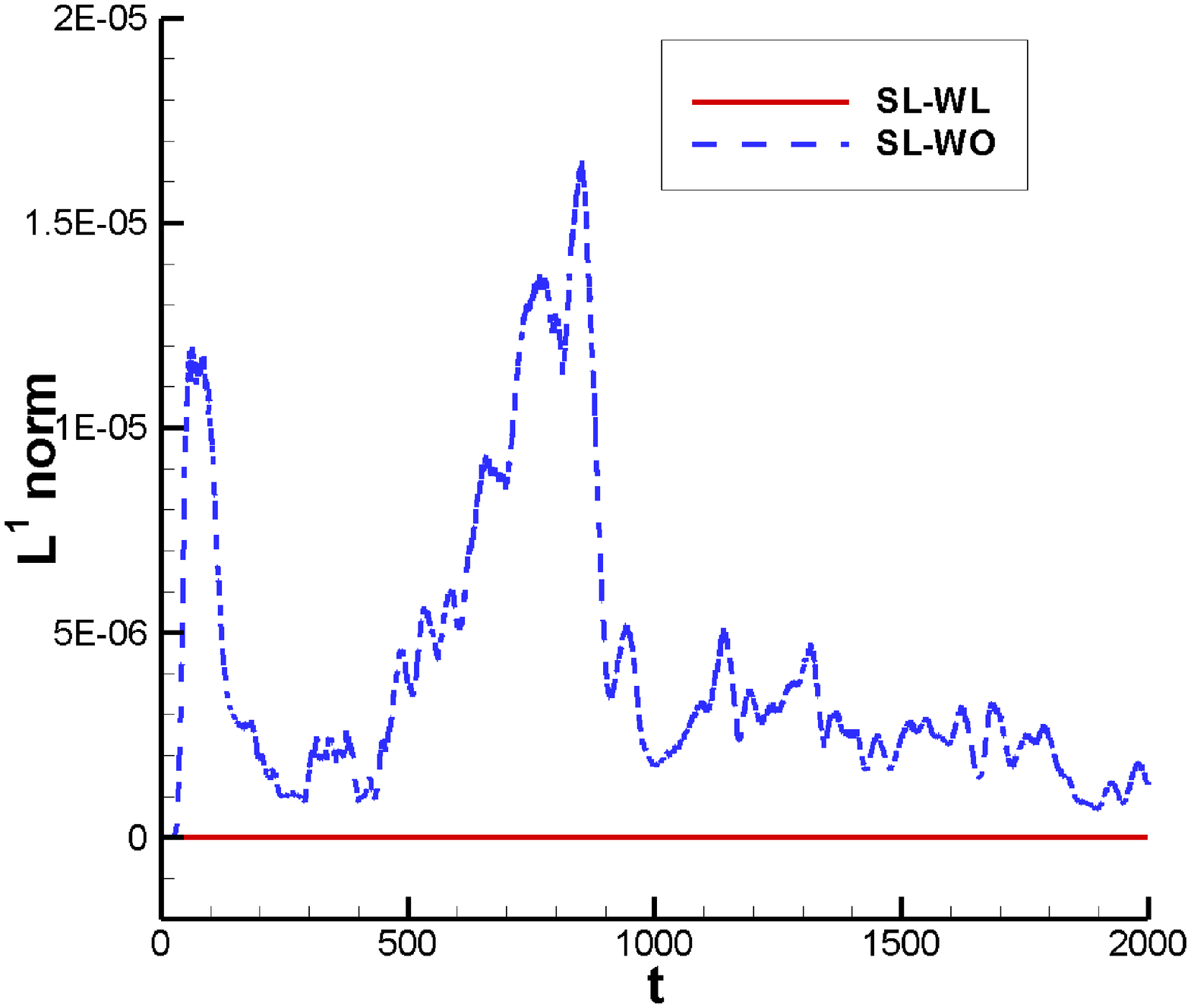}
\includegraphics[width=2.8in,angle=270,clip]{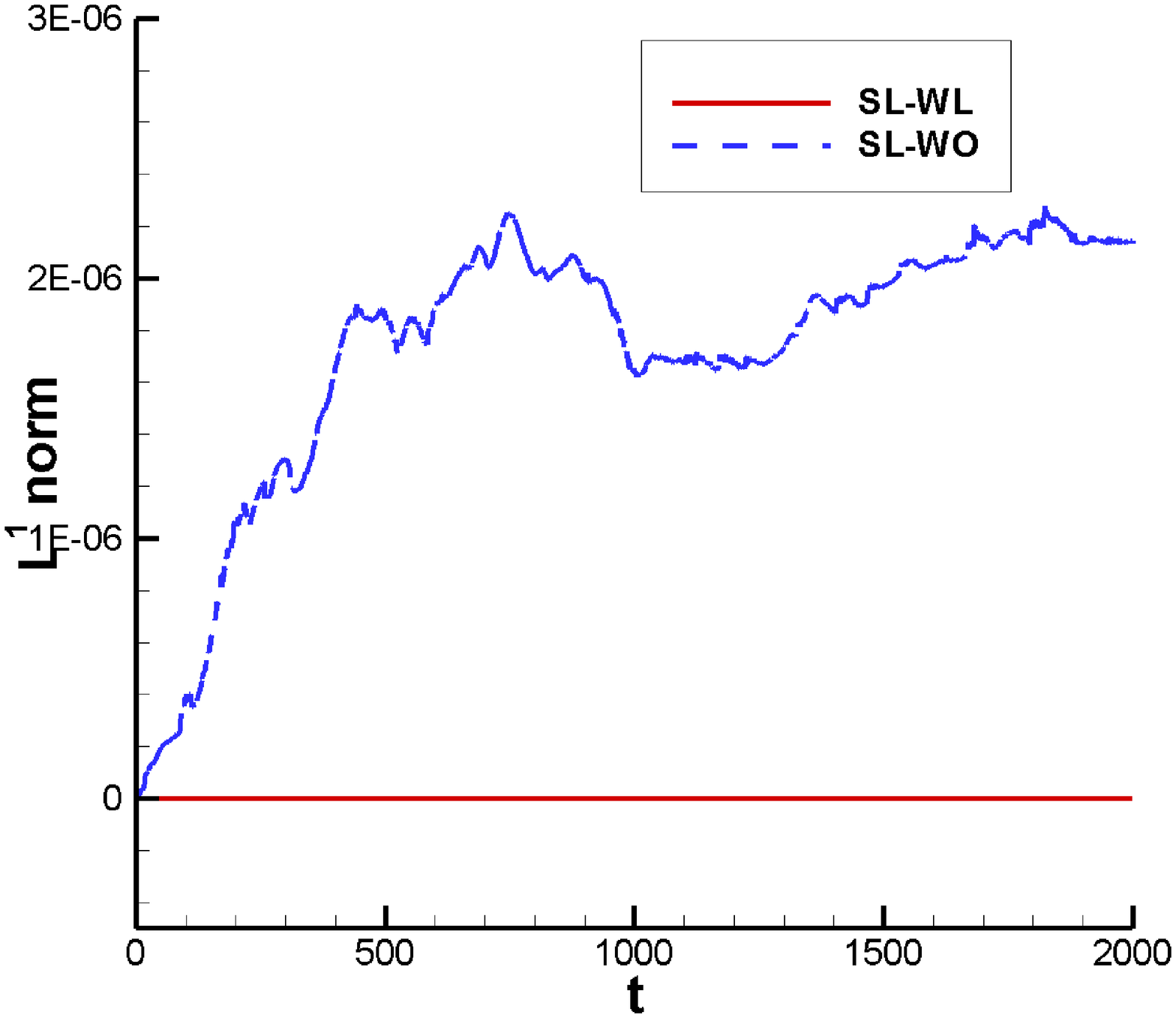}
\caption{Ion-acoustic turbulence problem \eqref{ion} and \eqref{elec}
with initial conditions \eqref{ion0} and \eqref{elec0}. Time evolution of 
$L^1$ norms for the distribution functions $f_e$ (left) and $f_i$ (right).
Mesh: $N_x\times N_v = 256 \times 256$.}
\label{ie22}
\end{figure}

\begin{figure}
\centering
\includegraphics[width=2.8in,angle=270,clip]{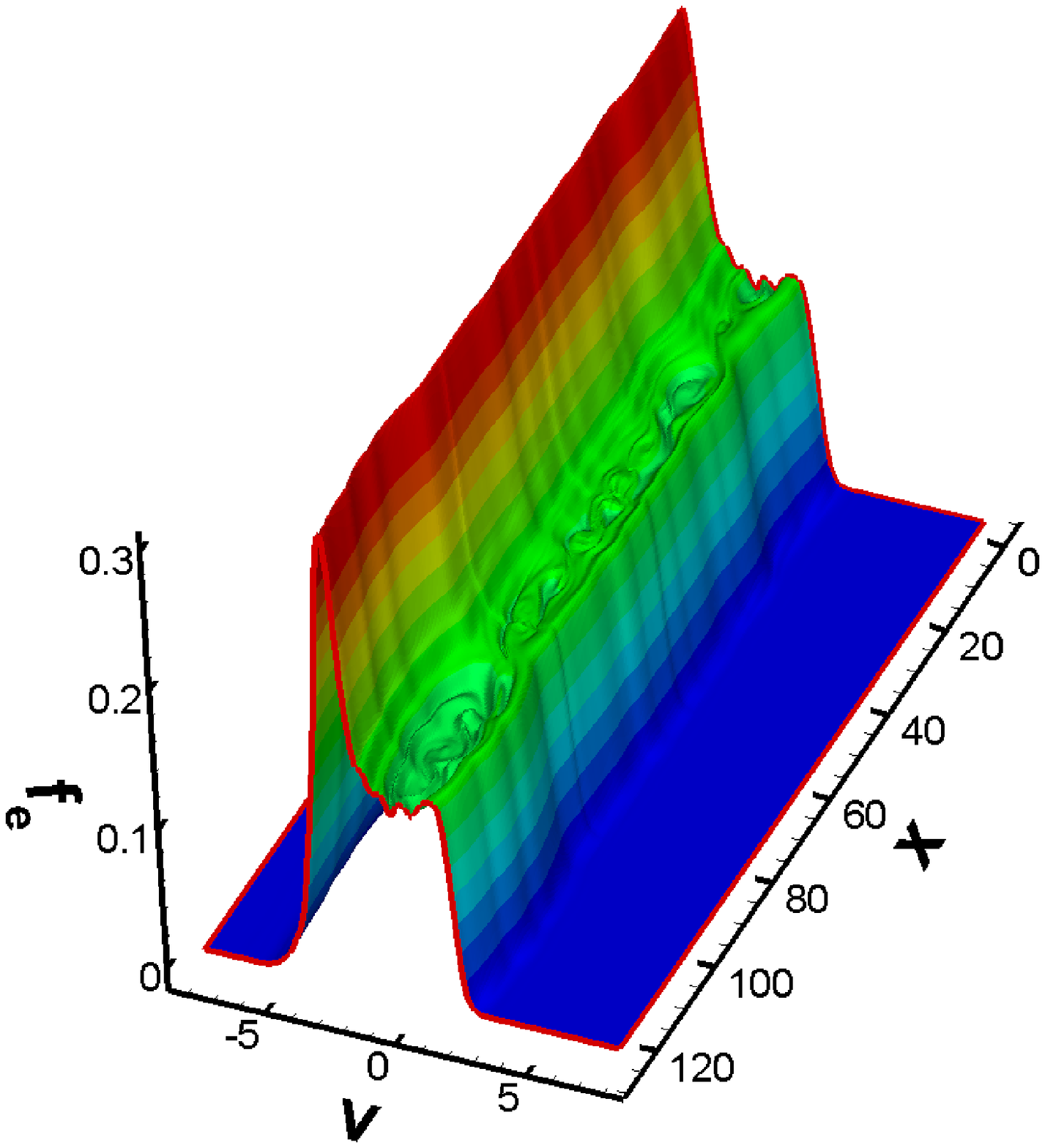}
\includegraphics[width=2.8in,angle=270,clip]{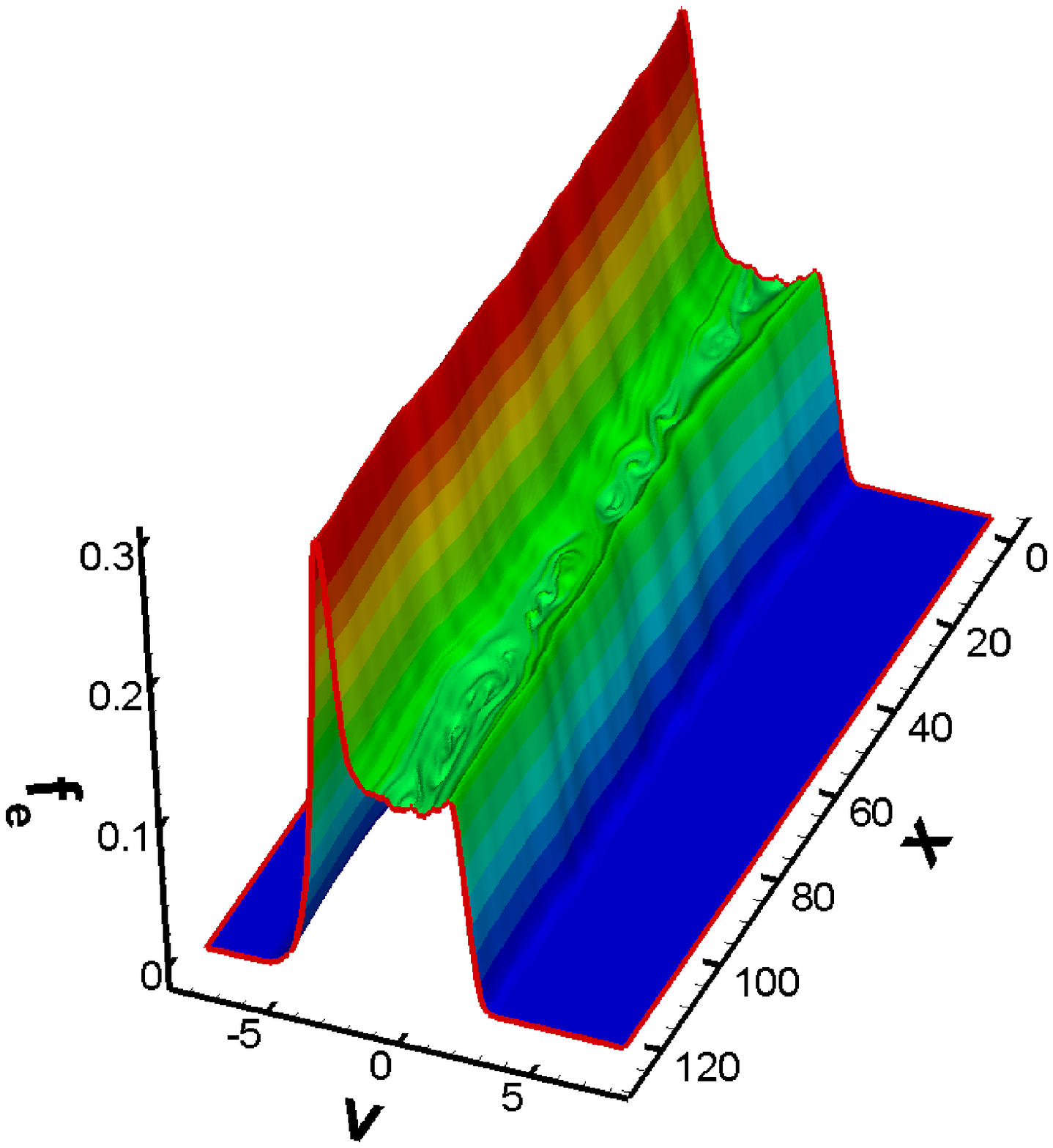}
\caption{Ion-acoustic turbulence problem \eqref{ion} and \eqref{elec}
with initial conditions \eqref{ion0} and \eqref{elec0}. Surface of $f_e$ at $t=2000$.
Mesh: $N_x\times N_v = 256 \times 256$. Left: with limiters; Right: without limiters.}
\label{ie23}
\end{figure}
\end{exa}
\section{Conclusion}
\label{sec6}
\setcounter{equation}{0}
\setcounter{figure}{0}
\setcounter{table}{0}

In this paper, we generalized the parametrized MPP flux limiter to the
semi-Lagrangian finite difference WENO scheme with application to the 
Vlasov-Poisson system. The MPP flux limiter preserves the maximum principle of the
numerical solutions by the design, while maintaining the mass conservation and high order accuracy.
Numerical studies demonstrate decent performance of the MPP flux limiter. In addition, 
the scheme can preserve the discrete $L^1$ norm up to the machine error. 

\appendix
\section{Appendix}

\label{appendix}

\renewcommand{\theequation}{A.\arabic{equation}}

Below, the first order numerical flux $h_{j+\f12}$ and the fifth order fluxes in (\ref{weno_flux}) for the case of $|a|\Delta t\le \Delta x$ are provided. For more details, see \cite{qiu2011bconservative}.

If $a>0$, let $\xi=a\f{\Delta t}{\Delta x}$, the first order reconstructed flux would be
\begin{equation}
h_{i+\f12}=a\Delta t u^n_i.
\label{eq11}
\end{equation}
The fifth order WENO reconstruction of $\mathcal{H}(x_{i+\frac12})$ is a convex combination of three 3rd order fluxes, and is given by
\begin{equation}
\mathcal{H}(x_{i+\frac12})=\omega_1 \mathcal{H}^{(1)}(x_{i+\frac12}) + \omega_2 \mathcal{H}^{(2)}(x_{i+\frac12}) + \omega_3 \mathcal{H}^{(3)}(x_{i+\frac12}),
\label{eq10}
\end{equation}
where $\mathcal{H}^{(r)}(x_{i+\frac12})$ is reconstructed from the following three potential stencils for $r=1, 2, 3$,
\begin{equation*}
S_1=\{u^n_{i-2}, u^n_{i-1}, u^n_i\}, \quad S_2=\{u^n_{i-1}, u^n_{i}, u^n_{i+1}\} \text{ and }
S_3=\{u^n_{i}, u^n_{i+1}, u^n_{i+2}\}.
\end{equation*}
The reconstructed third order fluxes in (\ref{eq10}) are
\begin{eqnarray}
\mathcal{H}^{(1)}(x_{i+\frac12})&=&\Delta x\left((\frac16\xi^3-\frac12\xi^2+\frac13\xi)u^n_{i-2}
+(-\frac13\xi^3+\frac32\xi^2-\frac76\xi)u^n_{i-1}+(\frac16\xi^3-\xi^2+\frac{11}{6}\xi)u^n_i\right), 
\nonumber \\
\\
\mathcal{H}^{(2)}(x_{i+\frac12})&=&\Delta x\left((\frac16\xi^3-\frac16\xi)u^n_{i-1}
+(-\frac13\xi^3+\frac12\xi^2+\frac56\xi)u^n_{i}+(\frac16\xi^3-\frac12\xi^2+\frac{1}{3}\xi)u^n_{i+1}\right), \\
\mathcal{H}^{(3)}(x_{i+\frac12})&=&\Delta x\left((\frac16\xi^3+\frac12\xi^2+\frac13\xi)u^n_{i}
+(-\frac13\xi^3-\frac12\xi^2+\frac56\xi)u^n_{i+1}+(\frac16\xi^3-\frac{1}{6}\xi)u^n_{i+2}\right).
\end{eqnarray}
The linear weights $\gamma_1$, $\gamma_2$ and $\gamma_3$ are
\begin{eqnarray}
\gamma_1=\frac1{10}+\frac3{20}\xi+\frac1{20}\xi^2, \quad
\gamma_2=\frac35+\frac1{10}\xi-\frac1{10}\xi^2, \quad 
\gamma_3=\frac3{10}-\frac14{\xi}+\frac1{20}\xi^2.
\label{linearw1}
\end{eqnarray}
The nonlinear weights $\omega_1$, $\omega_2$ and $\omega_3$ are computed by
\begin{equation}
\omega_r=\tilde{\omega}_r/\sum^3_{i=1}\tilde \omega_i, \qquad \tilde \omega_r=\gamma_r/(\epsilon+\beta_r)^2, \quad r=1, 2, 3.
\label{nonlinearw}
\end{equation}
where $\epsilon$ is a small number to avoid the denominator becoming zero. In our numerical tests,
we take $\epsilon=10^{-6}$. The smooth indicators $\beta_1$, $\beta_2$ and $\beta_3$ are
\begin{eqnarray*}
\beta_1&=&\frac{13}{12}(u^n_{i-2}-2u^n_{i-1}+u^n_i)^2+\frac14(u^n_{i-2}-4u^n_{i-1}+3u^n_i)^2, \\
\beta_2&=&\frac{13}{12}(u^n_{i-1}-2u^n_{i}+u^n_{i+1})^2+\frac14(u^n_{i-1}-u^n_{i+1})^2, \\
\beta_3&=&\frac{13}{12}(u^n_{i}-2u^n_{i+1}+u^n_{i+2})^2+\frac14(3u^n_{i}-4u^n_{i+1}+u^n_{i+2})^2. 
\end{eqnarray*}

If $a<0$, let $\xi=|a|\f{\Delta t}{\Delta x}$, the first order flux is
\begin{equation}
h_{i+\f12}=-|a|\Delta t u^n_{i+1}.
\label{eq12}
\end{equation}
For the high order fluxes (\ref{eq10}), 
\begin{eqnarray}
\mathcal{H}^{(1)}(x_{i+\frac12})&=&-\Delta x\left((\frac16\xi^3-\frac16\xi)u^n_{i-1}
+(-\frac13\xi^3+\frac12\xi^2+\frac56\xi)u^n_{i}+(\frac16\xi^3-\frac12\xi^2+\frac{1}{3}\xi)u^n_{i+1}\right), \\
\mathcal{H}^{(2)}(x_{i+\frac12})&=&-\Delta x\left((\frac16\xi^3+\frac12\xi^2+\frac13\xi)u^n_{i}
+(-\frac13\xi^3-\frac12\xi^2+\frac56\xi)u^n_{i+1}+(\frac16\xi^3-\frac{1}{6}\xi)u^n_{i+2}\right), \\
\mathcal{H}^{(3)}(x_{i+\frac12})&=&-\Delta x\left((\frac16\xi^3-\xi^2+\frac{11}6\xi)u^n_{i+1}
+(-\frac13\xi^3+\frac32\xi^2-\frac76\xi)u^n_{i+2}+(\frac16\xi^3-\frac12\xi^2+\frac{1}{3}\xi)u^n_{i+3}\right), \nonumber \\ 
\end{eqnarray}
with linear weights
\begin{eqnarray}
\gamma_1=\frac3{10}-\frac14{\xi}+\frac1{20}\xi^2, \quad
\gamma_2=\frac35+\frac1{10}\xi-\frac1{10}\xi^2, \quad
\gamma_3=\frac1{10}+\frac3{20}\xi+\frac1{20}\xi^2.
\label{linearw2}
\end{eqnarray}
The smooth indicators are
\begin{eqnarray*}
\beta_1&=&\frac{13}{12}(u^n_{i-1}-2u^n_{i}+u^n_{i+1})^2+\frac14(u^n_{i-1}-4u^n_{i}+3u^n_{i+1})^2, \\
\beta_2&=&\frac{13}{12}(u^n_{i}-2u^n_{i+1}+u^n_{i+2})^2+\frac14(u^n_{i}-u^n_{i+2})^2, \\
\beta_3&=&\frac{13}{12}(u^n_{i+1}-2u^n_{i+2}+u^n_{i+3})^2+\frac14(3u^n_{i+1}-4u^n_{i+2}+u^n_{i+3})^2. \\
\end{eqnarray*}

\bibliographystyle{siam}
\bibliography{refer}

\end{document}